\documentclass[a4paper]{amsart}

\usepackage[utf8]{inputenc}
\usepackage[english]{babel}
\usepackage{stmaryrd,mathtools,commath,mathrsfs}
\usepackage{bbm} %use \mathbbm for double-struck letters
\usepackage{amssymb,amsthm}
\numberwithin{equation}{section}
\usepackage{float}
\usepackage{graphicx}
\usepackage{yfonts,verbatim,calrsfs,enumitem}
\usepackage{comment}
\usepackage{todonotes}

\usepackage{amsmath}
\usepackage{xcolor}
\definecolor{cornellred}{rgb}{0.7,0.11,0.11}

\usepackage{cite}
\usepackage[colorlinks=false, allcolors=cornellred, pdfpagelabels]{hyperref}

\allowdisplaybreaks

\theoremstyle{plain}
\newtheorem{theorem}{Theorem}[section]
\newtheorem{lemma}[theorem]{Lemma}
\newtheorem{proposition}[theorem]{Proposition}
\newtheorem{corollary}[theorem]{Corollary}

\theoremstyle{definition}
\newtheorem{definition}[theorem]{Definition}

\newtheorem{remark}[theorem]{Remark}

\newtheorem{notation}[theorem]{Notation}

\DeclareMathAlphabet{\mathcal}{OMS}{cmsy}{m}{n}

\newcommand{\Cs}{\ensuremath{\mathrm{C}^\ast}}

\DeclareMathOperator{\Cu}{\mathbf{Cu}}
\DeclareMathOperator{\Cus}{\mathrm{Cu}}
\DeclareMathOperator{\1}{\mathbf{1}}

\DeclareMathOperator{\id}{\mathrm{id}}

\DeclareMathOperator{\Aut}{\mathrm{Aut}}

\DeclareMathOperator{\ev}{\mathrm{ev}}
\DeclareMathOperator{\Ad}{\mathrm{Ad}}
\DeclareMathOperator{\N}{\mathbb{N}}
\DeclareMathOperator{\Z}{\mathbb{Z}}
\DeclareMathOperator{\R}{\mathbb{R}}
\DeclareMathOperator{\T}{\mathbb{T}}
\DeclareMathOperator{\M}{\mathcal{M}}
\DeclareMathOperator{\U}{\mathcal{U}}
\DeclareMathOperator{\K}{\mathcal{K}}
\DeclareMathOperator{\I}{\mathcal{I}}
\DeclareMathOperator{\C}{\mathcal{C}}
\DeclareMathOperator{\B}{\mathcal{B}}
\DeclareMathOperator{\Hil}{\mathcal{H}}

\DeclareMathOperator{\O2}{\mathcal{O}_2}
\DeclareMathOperator{\Oinf}{\mathcal{O}_{\infty}}
\newcommand{\e}[0]{\varepsilon}
\DeclareMathOperator{\f}{\varphi}

\DeclareMathOperator{\p}{\mathfrak{p}}
\DeclareMathOperator{\q}{\mathfrak{q}}
\DeclareMathOperator{\bbu}{\mathbbm{u}}
\DeclareMathOperator{\bbv}{\mathbbm{v}}

\DeclareMathOperator{\cc}{\simeq_{\mathrm{cc}}}

\DeclareMathOperator{\Prim}{\mathrm{Prim}}

\usepackage{tikz}
\usetikzlibrary{arrows,matrix,positioning,automata}
\usepackage{quiver}
\usepackage{pgfplots}

\setlist[enumerate,itemize]{noitemsep,nolistsep}

\begin{document}

\title[Classification of equivariantly {$\O2$-stable} amenable actions]{{Classification of equivariantly $\O2$-stable amenable actions on nuclear \Cs-algebras}}
%Alternative: A dynamical version of Kirchberg's $\O2$-stable classification.

%    author one information
\author{Matteo Pagliero}
\address{Department of Mathematics, KU Leuven, Celestijnenlaan 200b, box 2400, B-3001 Leuven, Belgium}
\curraddr{}
\email{matteo.pagliero@kuleuven.be}
\thanks{}

%    author two information
\author{Gábor Szabó}
%\address{Department of Mathematics, KU Leuven, Celestijnenlaan 200b, box 2400, B-3001 Leuven, Belgium}
\curraddr{}
\email{gabor.szabo@kuleuven.be}
\thanks{}



\date{}

\dedicatory{}

\begin{abstract}
Given a second-countable, locally compact group $G$, we consider amenable $G$-actions on separable, stable, nuclear \Cs-algebras that are isometrically shift-absorbing and tensorially absorb the trivial action on the Cuntz algebra $\O2$.\
We show that such actions are classified up to cocycle conjugacy by the induced $G$-action on the primitive ideal space.\
In the special case when $G$ is exact, we prove a unital version of our classification theorem.\
For compact groups, we obtain a classification up to conjugacy.
\end{abstract}

\maketitle

\tableofcontents

\section*{Introduction}
\renewcommand\thetheorem{\Alph{theorem}}

Group actions have played a central role in the field of operator algebras since their inception as an aid to describe natural physical phenomena occurring in quantum mechanics.\
If quantum systems are described by operator algebras, time evolution and symmetries are expressed in terms of group actions on them.\
These turned out to be fundamental in the purely mathematical context for understanding structural properties and the classification of operator algebras.\
A prime example is Connes's celebrated classification of injective factors \cite{Con76} (complemented by Haagerup \cite{Haa87}), which relies on the classification of cyclic group actions.\
Consequently, Connes's result opened up a line of research that, through many subsequent works, culminated in the classification of actions of amenable groups on injective factors up to cocycle conjugacy \cite{Con75,Con77,Jon83,Ocn85,ST89,KST92,KST98,Mas13}.\
This represented a sufficiently significant milestone to spark a lot of interest in the classification of group actions on other classes of operator algebras, most prominently \Cs-algebras.\
The historical overview of progress with respect to the latter can be found in Izumi's survey \cite{Izu11}, which describes the state of the art at the time of its publication in 2011.\
For an overview of more recent works, the reader can consult the introduction of \cite{Sza21} and the references therein.

Although we do not give a complete historical account of the classification of \Cs-dynamical systems, we would like to mention the historical importance of the techniques developed by Evans and Kishimoto for classifying single automorphisms \cite{EK97}, which was followed by work of Nakamura \cite{Nak00} on single automorphisms of simple purely infinite \Cs-algebras.\
 Partially in parallel to the study of single automorphisms, the Rokhlin property for finite group actions was introduced by Izumi in \cite{Izu04a,Izu04b}, where he also gave a general classification of these up to conjugacy that is easily applicable to concrete situations.\
 More recently, Izumi--Matui \cite{IM21a,IM21b} established a groundbreaking classification of poly-$\Z$ group actions on Kirchberg algebras, generalising their previous result for $\Z^2$-actions \cite{IM10}.\
 A more general classification theorem for (second-countable) locally compact group actions on Kirchberg algebras was proved yet more recently by Gabe and the second named author \cite{GS22b}.\
 This can be considered as the dynamical version of the Kirchberg--Phillips theorem \cite{Phi00,Kir}, in which the $KK$-theoretical invariant is replaced by equivariant Kasparov theory $KK^G$ \cite{Kas88}.\
 The successful classification of group actions on separable, nuclear, purely infinite \Cs-algebras that are simple makes this the right time to turn the attention to group actions on not necessarily simple, separable, nuclear, purely infinite \Cs-algebras.\
In the present work we develop the first classification theory for actions of arbitrary second-countable, locally compact groups on non-simple \Cs-algebras.

There are several important examples of strongly self-absorbing \Cs-algebras \cite{TW07} that play a key role in the classification of nuclear \Cs-algebras.
These objects are equally important in the dynamical setting, where strongly self-absorbing actions \cite{Sza18a,Sza18,Sza17} generalise them in an appropriate sense.\
One of the notable examples of strongly self-absorbing \Cs-algebras is the Cuntz algebra $\O2$ \cite{Cun77}.
Let us give a brief overview of the prominence of $\O2$ in the theory of \Cs-dynamical systems.\
First of all, note that despite its classification invariant being trivial, $\O2$ admits group actions in abundance, for instance in the sense that the relations of cocycle conjugacy of automorphisms of $\O2$, and isomorphisms of simple crossed products of the form $\O2\rtimes\Z_2$ are not Borel \cite{GL16}.\
In fact, Izumi \cite{Izu04a} provided a hands-on construction of uncountably many mutually non-cocycle conjugate outer actions of $\Z_2$ on $\O2$.\
To show that these actions are not cocycle conjugate, Izumi computes the $K$-theory groups of their crossed product explicitly, showing that they are all different and therefore the actions in question cannot be cocycle conjugate.\
In point of fact, it is known that the behavior of actions of groups with torsion cannot generally be predicted by considering just the classification invariant of the underlying \Cs-algebra, as evidenced by the aforementioned $K$-theoretical obstructions for actions on $\O2$.\
In contrast, for many torsion-free groups $G$, there is a unique outer $G$-action on $\O2$ up to cocycle conjugacy.\
This is true for $G=\Z^k$ by work of Matui \cite{Mat07}, although it should be noted that the case $k=1$ follows from work of Nakamura \cite{Nak00}.\
More generally, it has been shown in \cite{GS22b} that there exists a unique amenable outer $G$-action on every strongly self-absorbing Kirchberg algebra (such as $\O2$) when $G$ is assumed to be countable, discrete, exact, torsion-free with the Haagerup property.

In the following paragraph, we recall some of the seminal results involving $\O2$ and explain their dynamical generalisations with special emphasis on the relevance of equivariantly $\O2$-stable actions.\
Recall that the $\O2$-absorption theorem \cite{Ror94,KP00} states that a \Cs-algebra $A$ is separable, nuclear, unital and simple precisely when $A\otimes\O2\cong\O2$.\
Its dynamical counterpart for actions of amenable groups was established by the second author in \cite{Sza18b}.\
This states that if $G$ is a countable, discrete, amenable group, and $\delta:G\curvearrowright \O2$ an outer action that is cocycle conjugate to $\delta\otimes\id_{\O2}$, then any action $\alpha:G\curvearrowright A$ on a separable, nuclear, simple, unital \Cs-algebra is tensorially absorbed --- up to cocycle conjugacy --- by $\delta$.
(We note that this result had been known for finite groups earlier \cite{Izu04a}.)

\begin{definition}
A group action $\alpha:G\curvearrowright A$ on a \Cs-algebra is \textit{equivariantly $\O2$-stable} if it is cocycle conjugate to $\alpha\otimes\id_{\O2}:G\curvearrowright A\otimes\O2$.
\end{definition}

As a direct consquence of what was said before, for any countable, discrete, amenable group $G$, $\O2$ admits precisely one outer and equivariantly $\O2$-stable $G$-action.
This serves as evidence that equivariant $\O2$-stability for actions is the right dynamical counterpart of $\O2$-stability.\
One may notice that assuming equivariant $\O2$-stability rules out the type of examples like Izumi's uncountably many non-cocycle conjugacy $\Z_2$-actions on $\O2$ mentioned above.\
The dynamical $\O2$-absorption theorem was then extended by Suzuki \cite{Suz21} to amenable actions of countable, exact groups.\
It should be noted that Suzuki's result is stated for actions with the quasicentral approximation property (abbreviated QAP).\
Roughly, an action $\alpha:G\curvearrowright A$ has the QAP if $L^2(G,A)$, viewed as an equivariant $A$-bimodule, admits a net of approximately fixed elements with certain desirable properties; see Definition \ref{def:amenability} for the precise statement.
Since the QAP and amenability have been shown to be equivalent for locally compact group actions as a result of work of Buss--Echterhoff--Willet \cite{BEW20,BEW20a}, Suzuki \cite{Suz19} and Ozawa--Suzuki \cite{OS21}, one can use these two properties interchangeably.\
This turns out to be a useful new input for classifying amenable actions of locally compact groups, instead of restricting the attention to the more special case of actions of amenable groups.\
It is well-known that, in order to classify a sufficiently broad class of actions, one must impose some non-triviality conditions on the actions at hand, which is usually expressed as an outnerness-type property.\
The right outerness condition in the context of our article is imported from \cite{GS22b}, and is therein called isometric shift-absorption.\
Simply put, an action of $G$ on $A$ is isometrically shift-absorbing when $A$ can be locally approximated by $L^2(G,A)$ in a suitable sense (see Definition \ref{def:isa} and Remark \ref{rem:isa}), which forces at least $A\cong A\otimes\mathcal{O}_\infty$ for the underlying \Cs-algebra.\
In \cite{GS22b}, it is shown that an action of a countable, discrete group on a Kirchberg algebra is isometrically shift-absorbing precisely when it is outer.\
Although outerness is a strictly weaker assumption in greater generality, this particular case suggests that isometric shift-absorption is a reasonable condition to ask.\
In the same article it is observed that isometric shift-absorption is equivalent to the Rokhlin property for $\mathbb{R}^k$-actions on $\Oinf$-stable \Cs-algebras.\
However, the Rokhlin property does not coincide with (because it is strictly stronger than) isometric shift-absorption for compact group actions as one might be tempted to guess from this.
Since isometric shift-absorption and amenability can both be expressed as properties of the equivariant $A$-bimodule $L^2(G,A)$, they blend well together.\
In fact, when an action is both amenable and isometrically shift-absorbing, one gains access to a useful averaging argument developed in \cite{GS22b}, which was a key piece of methodology to proving the dynamical Kirchberg--Phillips theorem  and also plays a similar role in the present work.
Let us briefly recall the appropriate notion of morphism between group actions that is suitable to develop a classification up to cocycle conjugacy.

\begin{definition}
Let $\alpha:G\curvearrowright A$ and $\beta:G\curvearrowright B$ be group actions on \Cs-algebras.\
A \emph{proper cocycle morphism} is a pair $(\f,\bbu):(A,\alpha)\to(B,\beta)$, where $\f:A\to B$ is a $\ast$-homomorphism and $\bbu:G\to\U(\1+B)$ is a norm-continuous $\beta$-cocycle satisfying $\Ad(\bbu_g) \circ \beta_g \circ \f = \f \circ \alpha_g$ for all $g\in G$.
\end{definition}

Historically, the $\O2$-embedding theorem \cite{KP00} was proved as an intermediate step to the Kirchberg--Phillips theorem.\
Likewise, as an intermediate step to the dynamical Kirchberg--Phillips theorem \cite{GS22b}, the appropriate dynamical counterpart of the $\O2$-embedding theorem was established in \cite[Theorem G]{GS22b}.\
This states that an amenable action on a separable, exact \Cs-algebra admits a proper cocycle embedding into any equivariantly $\O2$-stable, isometrically shift-absorbing action on a Kirchberg algbera.

Beyond the realm of simple \Cs-algebras, one has to mention the classification of all separable, nuclear, strongly purely infinite \cite{KR02} \Cs-algebras via ideal-related $KK$-theory, announced by Kirchberg \cite{Kir00} and later proved by Gabe \cite{Gab21} (see also the latest version of Kirchberg's unpublished book \cite{Kir}).\
As an intermediate result, Kirchberg outlined an $\O2$-stable classification, stating that separable, nuclear, stable \Cs-algebras that tensorially absorb $\O2$ are classified by their primitive ideal space alone.\
Within the class of separable, nuclear, strongly purely infinite \Cs-algebras, the $\O2$-stable \Cs-algebras are exactly those whose closed, two-sided ideals are $KK$-contractible (meaning $KK$-equivalent to the zero \Cs-algebra) by \cite{Gab22} (and the same is true for hereditary subalgebras, quotients, inductive limits and extensions of such algebras \cite{Kir04, TW07}).\
This is a large class of operator algebras that deserves special attention.\
Motivated by the well established result of Gabe and Kirchberg for $\O2$-stable \Cs-algebras, in the present work we establish a dynamical version thereof.\
The first full proof of Kirchberg's $\O2$-stable classification theorem was published by Gabe \cite{Gab20}, and it features a slightly different methodology from Kirchberg's original draft.\
In his approach, Gabe observes that the lattice of ideals $\I(A)$ of a separable \Cs-algebra $A$ can be endowed with a structure of abstract Cuntz semigroup in the sense of \cite{CEI08}.\
That is, $\I(-)$ can be viewed as a covariant functor that sends a $\ast$-homomorphism $\f:A\to B$ between \Cs-algebras to a morphism of abstract Cuntz semigroups $\I(\f):\I(A)\to\I(B)$.\
As it turns out, for purely infinite \Cs-algebras $A$, $\I(A)$ is canonically isomorphic to the Cuntz semigroup $\Cus(A)$.\
The Cuntz semigroup of a \Cs-algebra is an invariant that can be traced back to the seminal work of Cuntz \cite{Cun78}.
It was later used by Toms \cite{Tom08} to construct a large family of counterexamples to the original Elliott conjecture for simple, separable, nuclear \Cs-algebras.\
Further study on the Cuntz semigroup led to the definition and study of the abstract Cuntz semigroup category \cite{CEI08} (see also \cite{APT11,APT18}), to which it belongs.
An alternative classification invariant for a separable, nuclear, stable and $\O2$-stable \Cs-algebra $A$ is the topological space $\Prim(A)$ given by primitive ideals of $A$ equipped with the Jacobson topology.\
Although it lacks a rich categorical structure, $\Prim(A)$ captures the isomorphisms class of $\I(A)$.\
For this reason, $\Prim(A)$ appears in the main theorem of \cite{Gab20}, while $\I(A)$ is a finer auxiliary invariant that appears in proofs.\
We follow a similar approach to prove a dynamical version of Kirchberg's $\O2$-stable classification.
Let us give a more precise idea of our framework.\
We assume $G$ to be a second-countable, locally compact group acting on a separable \Cs-algebra $A$.\
An element $g\in G$ acts by order isomorphisms on the ideal lattice $\I(A)$ of $A$ by associating $\alpha^{\sharp}_g(I)=\{\alpha_g(x) \mid x\in I\}$ to the ideal $I$.\
It follows that the restriction of $\alpha^{\sharp}$ to $\Prim(A)$ is a continuous $G$-action with respect to the Jacobson topology (see Lemma \ref{lem:contaction}).\
Since an inner automorphism of $A$ induces the identity map on $\Prim(A)$, a cocycle conjugacy between \Cs-dynamical systems induces an equivariant homeomorphism, or conjugacy, between the induced topological dynamical systems of primitive ideals.\footnote{Note that we use the term `topological dynamical system' in a rather broad sense here, as the topological space $\Prim(A)$ can be highly non-Hausdorff.}\
Our main result is a dynamical generalisation of Kirchberg's $\O2$-stable classification that uses the topological dynamical system given by $\alpha^{\sharp}:G\curvearrowright\Prim(A)$ as a classification invariant for $\alpha:G\curvearrowright A$.\
We state here a shortened rendition of our theorem, and the reader is referred to Theorem \ref{thm:classification} for the full version.

\begin{theorem}\label{thm:intro}
Let $G$ be a second-countable, locally compact group, and $A,B$ two separable, nuclear, stable \Cs-algebras.\
Then, two amenable, isometrically shift-absorbing and equivariantly $\O2$-stable actions $\alpha:G\curvearrowright A$ and $\beta:G\curvearrowright B$ are cocycle conjugate if and only if $\alpha^{\sharp}:G\curvearrowright\Prim(A)$ and $\beta^{\sharp}:G\curvearrowright\Prim(B)$ are conjugate.
\end{theorem}

Our methodology blends together various techniques from \cite{GS22b} and \cite{Gab20}.\
As one of the extra complications compared to what happens in \cite{GS22b}, we need to control in a suitable sense the ideal structure of the involved \Cs-algebras, which comes with additional technical challenges in light of the dynamical structure that we need to keep track of.\
The dynamical nature of the invariant makes it impossible to directly apply results from \cite{Gab20}, hence we need to reprove a number of intermediate results suited for our context, which can be considered the most challenging component of the present work.

The article is organised as follows.\
The first section presents the needed notation and terminology and recalls various well-known facts.\
This includes basics about sequence algebras, isometric shift-absorption, amenability, and (proper) cocycle morphisms between \Cs-dynamical systems.\
In the second section we recall some properties of the ideal lattice of a \Cs-algebra from \cite{Gab20} and describe the topological invariant that we associate to a $G$-\Cs-algebra.\
In Section \ref{sec:uniqueness}, we prove the uniqueness theorem underpinning our classification theory, which is reported below (see also Theorem \ref{thm:uniqueness}), by adapting some techniques from \cite{GS22b} to the non-simple case.\
Note that we assume $\beta$ to be strongly stable, i.e., conjugate to $\beta\otimes\id_{\K}$, where $\K$ denotes the \Cs-algebra of compact operators on a separable, infinite-dimensional Hilbert space.\
By virtue of an observation in \cite{GS22b}, this assumption can eventually be dropped in the classification theorem.

\begin{theorem}
Let $\alpha:G\curvearrowright A$ be an action on a separable, exact \Cs-algebra, and $\beta:G\curvearrowright B$ a strongly stable, amenable, equivariantly $\O2$-stable, isometrically shift-absorbing action on a separable \Cs-algebra.\
If $(\f,\bbu),(\psi,\bbv):(A,\alpha)\to(B,\beta)$ are two proper cocycle morphisms with $\f$ and $\psi$ nuclear, then $\I(\f)=\I(\psi)$ if and only if $(\f,\bbu)$ is strongly asymptotically unitarily equivalent to $(\psi,\bbv)$, i.e., there exists a norm-continuous unitary path $v:[0,\infty)\to\U(\1+B)$ with $v_0=\1$ and
\begin{equation*}
\psi(a)=\lim_{t\to\infty} v_t\f(a)v_t^*, \quad \lim_{t\to\infty} \max_{g\in K}\|\bbv_g-v_t\bbu_g\beta_g(v_t)^*\|=0
\end{equation*}
for all $a\in A$ and compact sets $K\subseteq G$.
\end{theorem}

We also establish a unital version of the theorem above under the assumption that $G$ is exact (see Corollary \ref{cor:uniqueness}).\
In Section \ref{sec:existence}, we prove a number of technical results that eventually lead to the existence theorem for proper cocycle morphisms, which can be considered as a dynamical, ideal-related version of the $\O2$-embedding theorem (see also Theorem \ref{thm:existence}).

\begin{theorem}
Let $\alpha:G\curvearrowright A$ be an amenable action on a separable, exact \Cs-algebra, and $\beta:G\curvearrowright B$ an isometrically shift-absorbing and equivariantly $\O2$-stable action on a separable \Cs-algebra.\
Then, for every equivariant $\Cu$-morphism $\Phi:(\I(A),\alpha^{\sharp})\to(\I(B),\beta^{\sharp})$, there exists a proper cocycle morphism $(\f,\bbu):(A,\alpha)\to(B,\beta)$ such that $\I(\f)=\Phi$.
\end{theorem}

Finally, in the last section, we combine the aforentioned theorems to deduce a bijective correspondence between equivariant morphisms of abstract Cuntz semigroups of ideals and proper cocycle morphisms up to asymptotic unitary equivalence.
From this we will deduce our main result, i.e.\ Theorem \ref{thm:classification}, the more detailed version of Theorem \ref{thm:intro}.\
Thanks to a fact that is implicitly proved in \cite{OS21}, we can furthermore show that a unital version of this theorem holds when $G$ is exact.\
In the case when the acting group $G$ is compact, results in \cite{GS22b} pertaining to asymptotic coboundaries allow us to obtain Corollary \ref{cor:compact-classification}, in which we achieve classification up to genuine conjugacy.\
As a final theorem, we observe that, when $G$ is exact, the assumptions on the classified actions can be replaced with absorption of a certain model action on $\O2$ up to cocycle conjugacy.\
We conclude the article by pointing out to the reader how in some special cases, namely when $G$ is $\mathbb{R},\Z$ or a compact group, one could have obtained our main result using existing techniques available from the literature.

\section{Preliminaries}
\renewcommand\thetheorem{\arabic{section}.\arabic{theorem}}

\begin{notation}
The multiplier algebra and the proper unitization\footnote{Add a unit even if $A$ is already unital.} of a \Cs-algebra $A$ are denoted by $\M(A)$ and $A^\dagger$, respectively.\
The unit of $\M(A)$ and $A^\dagger$ is denoted by $\1$.\
We adopt the notation $\U(\1+A)$ for $(\1+A)\cap\U(A^\dagger)$, which is canonically isomorphic to $\U(A)$ when $A$ is unital.\
The \Cs-algebra of compact operators on a Hilbert space $\Hil$ is denoted by $\K(\Hil)$, and when $\Hil$ is separable and infinite-dimensional one normally uses the shorthand $\K$.\
Throughout the paper, $G$ will denote a second-countable, locally compact group, unless specified otherwise.\
The notation $\id_A$ denotes the identity map on the \Cs-algebra $A$, and, with slight abuse of notation, also the trivial $G$-action on $A$.\
If $\alpha:G\curvearrowright A$ is an action on a \Cs-algebra, $A^\alpha$ denotes the $\Cs$-subalgebra of $A$ of fixed points under the action $\alpha$.\
Moreover, the canonical extension of $\alpha$ to an action on $\M(A)$ is denoted by $\alpha$ as well.\
If $u$ is a unitary in $\U(\1+A)$ or $\U(\M(A))$, then $\Ad(u)\in\Aut(A)$ denotes the inner automorphism of $A$ given by $a \mapsto uau^*$.\
Additionally, we use $A\odot B$, $A\otimes B$ and $A\otimes_{\max}B$ for algebraic, spatial and maximal tensor product of two $\Cs$-algebras $A$ and $B$, respectively.\
For two elements $a,b$ in a \Cs-algebra and $\e>0$, the shorthand $a=_{\e} b$ stands for the inequality $\|a-b\|\leq\e$.\
Closure and interior of a subset $U$ of a topological space $X$ are denoted by $\overline{U}$ and $U^{\circ}$, respectively.
Finally, we adopt the convention that $0\in\N$.
\end{notation}

\begin{definition}\label{def:sequencealgebra}
Let $\alpha:G\curvearrowright A$ be an action on a \Cs-algebra.\
The \textit{sequence algebra} of $A$ is
\begin{equation*}
A_{\infty}=\ell^{\infty}(\N,A)/c_0(\N,A).
\end{equation*}
We denote by $\iota_{\infty}$ the canonical embedding $A \hookrightarrow A_{\infty}$ when $A$ is understood from the context.

The \textit{$\alpha$-continuous sequence algebra} of $A$ is given by
\begin{equation*}
A_{\infty,\alpha}=\big\{ x\in A_\infty \mid [g\mapsto\alpha_{\infty,g}(x)] \text{ is continuous}\big\},
\end{equation*}
where $\alpha_{\infty}:G\to\Aut(A_\infty)$ is the not necessarily continuous action obtained by applying $\alpha$ componentwise.
By construction, $\alpha_\infty$ is a continuous action on $A_{\infty,\alpha}$.\
Moreover, as a consequence of \cite[Theorem 2]{Bro00}, $A_{\infty,\alpha}$ coincides with the quotient \Cs-algebra
\begin{equation*}
\ell^{\infty}_\alpha(\N,A)/c_0(\N,A),
\end{equation*}
where $\ell^{\infty}_\alpha(\N,A)=\big\{ (x_n)_{n\in\N}\in\ell^\infty(\N,A) \mid [g\mapsto (\alpha_g(x_n))_{n\in\N}] \text{ is continuous}\big\}$.

Let $B$ be an $\alpha_\infty$-invariant $\Cs$-subalgebra of $A_{\infty,\alpha}$.\
Then, $\alpha_\infty$ restricts to an action on the \textit{$\alpha$-continuous relative central sequence algebra},
\begin{equation*}
A_{\infty,\alpha} \cap B' = \{x\in A_{\infty,\alpha} \mid xb=bx \text{ for all } b\in B\},
\end{equation*}
as well as on the annihilator of $B$ in $A_{\infty,\alpha}$,
\begin{equation*}
A_{\infty,\alpha}\cap B^{\perp} = \{ x\in A_{\infty,\alpha} \mid xb=bx=0 \text{ for all } b\in B\}.
\end{equation*}
The induced action on $F(B,A_{\infty,\alpha})=(A_{\infty,\alpha} \cap B') / (A_{\infty,\alpha}\cap B^{\perp})$ is denoted by $\tilde{\alpha}_\infty$.\
Note that $F(B,A_{\infty,\alpha})$ is unital whenever $B$ is $\sigma$-unital.
If $A=B$, then one writes $F_{\infty,\alpha}(A)=F(A,A_{\infty,\alpha})$.
\end{definition}

\begin{notation}
Let $\mu$ be a left Haar measure on $G$.\
We denote the Hilbert spaces $L^2(G,\mu)$ and $\ell^2(\N)\hat{\otimes}L^2(G,\mu)$ by $\Hil_G$ and $\Hil_G^\infty$, respectively.\
Moreover, the left-regular representation of $G$ is denoted by $\lambda:G\to\U(\Hil_G)$, i.e., $\lambda_g(\xi)(h)=\xi(g^{-1}h)$ for all $\xi\in\Hil_G$ and $g,h\in G$, and its infinite repeat by $\lambda^\infty$, i.e., $\lambda^{\infty}=\1_{\ell^2(\N)}\otimes\lambda:G\to\U(\Hil_G^\infty)$.
\end{notation}

\begin{definition}[{see \cite[Definition 3.7]{GS22b}}]\label{def:isa}
Let $\alpha:G\curvearrowright A$ be an action on a separable $\Cs$-algebra.\
We say that $\alpha$ is \textit{isometrically shift-absorbing} if there exists an equivariant linear map
\begin{equation*}
\mathfrak{s}: (\Hil_G,\lambda)\to(F_{\infty,\alpha}(A),\tilde{\alpha}_{\infty})
\end{equation*}
that satisfies $\mathfrak{s}(\xi)^*\mathfrak{s}(\eta)=\langle \xi, \eta \rangle \cdot \1$ for all $\xi,\eta\in\Hil_G$.
\end{definition}

\begin{notation}\label{not:Hilbert}
Let $\alpha:G\curvearrowright A$ be an action on a \Cs-algebra, and $\Hil$ a Hilbert space.\
We consider the external tensor product $\Hil \otimes A$ (see \cite[Chapter 4]{Lan95}), where $A$ is viewed as a $G$-Hilbert bimodule with its identity $(A,\alpha)$-bimodule structure.\
Let $\sigma:G \to \U(\Hil)$ be a unitary representation that is continuous with respect to the strong operator topology.\
It follows that $\sigma\otimes\alpha:G\curvearrowright \Hil\otimes A$ is an action by linear isometries turning $\Hil\otimes A$ into a Hilbert $(A,\alpha)$-bimodule as well.\
We apply this construction to $\Hil_G$ and $\Hil_G^\infty$, and adopt the notation
\begin{equation*}
L^2(G,A):=\Hil_G\otimes A,\quad L^2_\infty(G,A):=\Hil_G^\infty\otimes A.
\end{equation*}
Note that $\C_c(G,A)$ is densely contained in $L^2(G,A)$ as a pre-Hilbert module with inner product given by
\begin{equation*}
\langle \xi, \eta \rangle = \int_{G}\xi(g)^*\eta(g) \, d\mu(g),
\end{equation*}
for all $\xi,\eta \in \C_c(G,A)$, and norm given by $\|\cdot\|_2=\|\langle \cdot, \cdot \rangle\|^{1/2}$.\
Moroever, we denote by $\bar{\alpha}=\lambda\otimes\alpha$ the $G$-action by isometries on $L^2(G,A)$, which is given by $\bar{\alpha}_g(\xi)(h)=\alpha_g(\xi(g^{-1}h))$ for all $g,h\in G$, and $\xi\in\C_c(G,A)$.
\end{notation}

\begin{remark}[see {\cite[Remark 3.1]{GS22b}}]\label{rem:gamma}
Recall \cite{Eva80} that for any separable, infinite-dimensional Hilbert space $\Hil$, the Cuntz algebra $\Oinf$ is isomorphic to $\mathcal{O}_{\Hil}$, the universal unital \Cs-algebra generated by the range of a linear map $\mathfrak{s}:\Hil \to \B(\Hil)$ subject to the relation $\mathfrak{s}(\xi)^*\mathfrak{s}(\eta)=\langle \xi, \eta \rangle \cdot \1$ for all $\xi,\eta \in \Hil$.\
As a consequence, every unitary $U$ on $\Hil$ gives rise to a unique automorphism on $\mathcal{O}_{\Hil}$ such that $\mathfrak{s}(\xi)$ is sent to $\mathfrak{s}(U\xi)$ for all $\xi\in\Hil$.\
The resulting assignment $\U(\Hil)\to\Aut(\mathcal{O}_{\Hil})$ is a group homomorphism that is continuous with respect to the strong operator topology on $\U(\Hil)$ and the point-norm topology on $\Aut(\mathcal{O}_{\Hil})$.\
Any group action on $\Oinf$ that is conjugate to one factoring through this homomorphism is said to be \textit{quasi-free}.
\end{remark}

\begin{definition}[{see \cite[Definition 3.4]{GS22b}}]\label{def:gamma}
Define $\gamma:G\curvearrowright \Oinf\cong\mathcal{O}_{\Hil_G^{\infty}}$ to be the quasi-free action determined by $\gamma_g \circ \mathfrak{s} = \mathfrak{s} \circ \lambda^{\infty}_g$ for all $g\in G$.
\end{definition}

\begin{remark}[{see \cite[Proposition 3.8]{GS22b}}]\label{rem:isa}
Assume that $G\neq\{1\}$.
Let $\beta:G\curvearrowright B$ be an action on a separable \Cs-algebra, and $\gamma:G\curvearrowright \Oinf$ the quasi-free action from Definition \ref{def:gamma}.
The following are equivalent,
\begin{enumerate}[label=(\roman*),leftmargin=*]
\item $\beta$ is isometrically shift-absorbing;
\item There exists a unital equivariant $\ast$-homomorphism $(\Oinf,\gamma)\to(F_{\infty,\beta}(B),\tilde{\beta}_\infty)$.
\end{enumerate}
\end{remark}

We may adopt the quasicentral approximation property as our definition of amenability since they are equivalent as a result of recent work by Buss--Echterhoff--Willett \cite{BEW20a,BEW20}, Suzuki \cite{Suz19} and Ozawa--Suzuki \cite{OS21}.

\begin{definition}[{cf.\ \cite[Definition 2.11]{OS21}}]\label{def:amenability}
Let $\alpha: G\curvearrowright A$ be an action on a $\Cs$-algebra.
Then $\alpha$ is said to be \textit{amenable} if there exists a net of contractions $\zeta_i\in\C_c(G,A)$ such that
\begin{align*}
&\langle \zeta_i , \zeta_i \rangle \to \1 \ \text{ strictly},
&\|a\zeta_i - \zeta_ia\|_2 \to 0,&
&\max_{g\in K}\|(\zeta_i - \bar{\alpha}_g(\zeta_i)) a\|_2 \to 0
\end{align*}
for all $a\in A$ and compact sets $K\subseteq G$.
\end{definition}

%\subsection{(Proper) cocycle morphisms}

In the rest of the preliminary section, we remind the reader of basic definitions and results regarding the categorical framework for classifying \Cs-dynamical systems developed by the second named author in \cite{Sza21}.

\begin{definition}
Let $A$ and $B$ be \Cs-algebras.
A $\ast$-homomorphism $\f:A\to\M(B)$ is \textit{extendible} if, for any increasing approximate unit $\{e_\lambda\}_{\lambda\in\Lambda}\subseteq A$, the net $\{\f(e_\lambda)\}_{\lambda\in\Lambda}$ converges strictly to a projection in $\M(B)$.
\end{definition}

\begin{definition}
Let $\beta:G\curvearrowright B$ be an action on a \Cs-algebra.
A strictly continuous map $\bbu:G\to\U(\M(B))$ is a \textit{$\beta$-cocycle} if it satisfies $\bbu_{gh}=\bbu_g\beta_g(\bbu_h)$ for all $g,h\in G$.
\end{definition}

In the categorical framework defined in \cite{Sza21}, the right notion of cocycle morphism between \Cs-dynamical systems is given by a pair $(\f,\bbu)$ consisting of an extendible $\ast$-homomorphism $\f$ and a $\beta$-cocycle $\bbu$ satisfying the equivariance condition $\f \circ \alpha_{\bullet} = \Ad(\bbu_{\bullet}) \circ \beta_{\bullet} \circ \f$.\
Since extendibility of $\ast$-homomorphisms is not necessary in the context of the present work, we adopt the following definition of cocycle morphism, which is also the convention used in \cite[Section 6]{Sza21} and does not require $\f$ to be extendible.

\begin{definition}[{cf. \cite[Definition 1.10]{Sza21}}]\label{def:comorp}
Let $\alpha:G\curvearrowright A$ and $\beta:G\curvearrowright B$ be actions on \Cs-algebras.\
A \textit{cocycle morphism} from $(A,\alpha)$ to $(B,\beta)$ is a pair $(\f,\bbu)$ consisting of a $\ast$-homomorphism $\f:A\to B$ and a $\beta$-cocycle $\bbu:G\to\U(\M(B))$ satisfying the equivariance condition
\begin{equation*}
\f \circ \alpha_g = \Ad(\bbu_g) \circ \beta_g \circ \f
\end{equation*}
for all $g\in G$.
When $\bbu$ is a norm-continuous $\beta$-cocycle with values in $\U(\1+B)$, $(\f,\bbu)$ is said to be a \textit{proper cocycle morphism}.
\end{definition}

\begin{remark}[see {\cite[Proposition 1.15]{Sza21}}]\label{rem:composition}
%The \textit{(proper) cocycle category} with respect to $G$ consists of $G$-$\Cs$-algebras as objects, and arrows given by (proper) cocycle morphisms.
The composition of two proper cocycle morphisms $(\f,\bbu):(A,\alpha)\to(B,\beta)$ and $(\psi,\bbv):(B,\beta)\to(C,\gamma)$, is the proper cocycle morphism given by the pair
\begin{equation*}
(\psi,\bbv) \circ (\f,\bbu) := (\psi \circ \f, \psi^{\dagger}(\bbu) \bbv):(A,\alpha)\to(C,\gamma),
\end{equation*}
where $\psi^\dagger:B^\dagger \to C^\dagger$ is the unique unital extension of $\psi$.
\end{remark}

\begin{notation}
When $(\f,\bbu):(A,\alpha)\to(B,\beta)$ is a (proper) cocycle morphism and $\f$ is injective, $(\f,\bbu)$ is said to be a \textit{(proper) cocycle embedding}.\
In the same spirit, $(\f,\bbu)$ is said to be a \textit{(proper) cocycle conjugacy} if $\f$ is an isomorphism, and when this exists, $(A,\alpha)$ and $(B,\beta)$ are said to be \textit{(properly) cocycle conjugate}.\
Moreover, we will write $\alpha\cc\beta$ when $(A,\alpha)$ and $(B,\beta)$ are cocycle conjugate.
\end{notation}

\begin{definition}[{cf.\ \cite[Definitions 2.8+4.1]{Sza21} and \cite[Definition 1.15]{GS22b}}]
Let
\[
(\f,\bbu),(\psi,\bbv):(A,\alpha)\to(B,\beta)
\]
be two cocycle morphisms.\
\begin{enumerate}[label=\textup{(\roman*)},leftmargin=*]
\item $(\f,\bbu)$ and $(\psi,\bbv)$ are said to be \textit{(properly) unitarily equivalent} if there exists a unitary $u\in\U(\M(B))$ (respecively, $u\in\U(\1+B)$) such that
\begin{equation*}
 \psi(a) = u\f(a)u^*, \quad \bbv_g = u \bbu_g \beta_g(u)^*
\end{equation*}
for all $a\in A$, and $g\in G$.
\item $(\f,\bbu)$ and $(\psi,\bbv)$ are said to be \textit{strongly asymptotically unitarily equivalent} if there exists a norm-continuous path of unitaries $u:\left[0,\infty\right) \to \U(\1+B)$ starting at $u_0=\1$ that satisfies
\begin{equation*}
\psi(a) = \lim_{t\to\infty} u_t\f(a)u_t^*, \text{ for all }a\in A,
\end{equation*}
for all $a\in A$, and
\begin{equation*}
\lim_{t\to\infty} \max_{g\in K} \|\bbv_g -u_t\bbu_g\beta_g(u_t)^*\| = 0,
\end{equation*}
for all compact subsets $K\subseteq G$.
\end{enumerate}
\end{definition}

\begin{remark}[{see \cite[Notation 1.3+Proposition 1.4]{GS22b}}]\label{rem:strstable}
An action $\beta:G\curvearrowright B$ on a $\Cs$-algebra is said to be \textit{strongly stable} if $(B,\beta)$ is conjugate to $(B\otimes\K,\beta\otimes\id_{\K})$.\
If $B$ is a stable \Cs-algebra, then every action $\beta:G\curvearrowright B$ is cocycle conjugate to $\beta\otimes\id_{\K}:G\curvearrowright B\otimes\K$.
\end{remark}

The following definition is a generalization of strongly self-absorbing \Cs-algebra \cite{TW07} for group actions.

\begin{definition}[{see \cite[Definition 3.1]{Sza18a} and \cite[Definition 5.3]{Sza21}}]
Let $\delta:G\curvearrowright \mathcal{D}$ be an action on a separable, unital \Cs-algebra.\
The action $\delta$ is said to be \textit{strongly self-absorbing} if there exists a cocycle conjugacy
\begin{equation*}
(\f,\bbu) : (\mathcal{D},\delta) \to (\mathcal{D}\otimes \mathcal{D},\delta\otimes\delta)
\end{equation*}
and a sequence of unitaries $v_n\in\mathcal \U(\mathcal{D}\otimes\mathcal{D})$ such that
\[
\f(x)=\lim_{n\to\infty} v_n(x\otimes\textbf{1})v_n^* \quad\text{and}\quad \lim_{n\to\infty} \max_{g\in K} \|\bbu_g -v_n(\delta\otimes\delta)_g(v_n)^*\| = 0
\]
for all $x\in\mathcal D$ and every compact set $K\subseteq G$.

An action $\alpha:G\curvearrowright A$ on a separable \Cs-algebra is said to be \textit{$\delta$-stable} or \textit{$\delta$-absorbing}, if $\alpha$ is cocycle conjugate to $\alpha\otimes\delta$.
More specifically, $\alpha$ is called \textit{equivariantly $\mathcal{D}$-stable} if $\alpha\otimes\id_{\mathcal{\mathcal{D}}}$ is cocycle conjugate to $\alpha$.
\end{definition}

\begin{remark}[{see \cite[Theorem 5.6]{Sza21}}]\label{rem:thm5.6}
Let $\alpha:G\curvearrowright A$ be an action on a separable \Cs-algebra.\
Then, $\alpha$ is equivariantly $\O2$-stable if and only if the equivariant first factor embedding
\begin{equation*}
\id_A \otimes \1_{\O2}: (A,\alpha) \to (A\otimes\O2,\alpha\otimes\id_{\O2})
\end{equation*}
is strongly asymptotically unitarily equivalent to a proper cocycle conjugacy.
\end{remark}

\section{Ideal-related invariants and induced dynamics}

\begin{notation}
Let $A$ be a $\Cs$-algebra.
The set of all closed two-sided ideals of $A$ is denoted by $\I(A)$.\
Note that the partial order induced by inclusion of ideals turns $\I(A)$ into a partially ordered set.\
One may moreover define the join (or supremum) of two ideals $I,J\in\I(A)$ to be $I+J\in\I(A)$ and their meet (or infimum) to be $I\cap J=I\cdot J$.\
Consequently, $\I(A)$ is an order-theoretic lattice, called the \textit{ideal lattice} of $A$.\
Furthermore, $\I(A)$ is complete, i.e., every family of ideals has an infimum and a supremum.
%In particular, note that the supremum of a (possibly infinite) family $\{I_\lambda\}_{\lambda\in\Lambda} \subseteq \I(A)$ is $\overline{\sum_{\lambda\in\Lambda}I_\lambda}$, while its infimum is $\bigcap_{\lambda\in\Lambda}I_\lambda$.
\end{notation}

Let us now recall some useful properties of the ideal lattice of a \Cs-algebra from \cite[Section 2]{Gab20}.

\begin{definition}
Let $I$ and $J$ be ideals of a $\Cs$-algebra $A$.\
One says that $I$ is \textit{compactly contained} in $J$, and write $I\Subset J$, if for any family of ideals $\{I_\lambda\}_{\lambda\in\Lambda}\subseteq\I(A)$ such that $J\subseteq \overline{\sum_{\lambda}I_\lambda}$ there exist finitely many $\lambda_1,\dots,\lambda_n\in\Lambda$ such that $I\subseteq \sum_{i=1}^n I_{\lambda_n}$.
\end{definition}

\begin{remark}\label{rem:increasingfamily}
Observe that two ideals $I,J$ of a \Cs-algebra $A$ satisfy $I\Subset J$ if and only if for every upward directed family of ideals $\{I_\lambda\}_{\lambda\in\Lambda}$ with $J\subseteq \overline{\bigcup_{\lambda\in\Lambda}I_\lambda}$, there exists $\lambda\in\Lambda$ such that $I\subseteq I_\lambda$.\
Let us show the nontrivial implication.\
Assume that $J\subseteq \overline{\sum_{\lambda}I_\lambda}$ for some (not necessarily directed) family of ideals $\{I_\lambda\}_{\lambda\in\Lambda}$.\
The upward directed family of ideals given by $\{\sum_{\gamma\in\Gamma}I_\gamma \mid \Gamma\subseteq\Lambda\text{ finite}\}$ has the same supremum as $\{I_\lambda\}_{\lambda\in\Lambda}$.\
Hence, there exist $\lambda_1,\dots,\lambda_n \in \Lambda$ such that $I\subseteq\sum_{i=1}^n I_{\lambda_i}$, i.e., we have verified that $I\Subset J$.
\end{remark}

\begin{notation}
Let $a$ be a positive element of a \Cs-algebra $A$.\
We write $(a-\e)_+$ for the element of $A$ given by $f(a)$ via functional calculus, where $f(t)=\max\{0,t-\e\}$.
\end{notation}

\begin{remark}[{see \cite[Lemma 2.2+Proposition 6.5]{Gab20}}]\label{rem:ideallattice}
Let $A$ be a \Cs-algebra.\
The following statements hold true.\
\begin{enumerate}[label=\textup{(\roman*)},ref=\textup{\roman*},leftmargin=*]
\item For any positive element $a\in A$, one has that $\overline{A(a-\e)_{+}A} \Subset \overline{AaA}$.
\item If $A$ is weakly purely infinite (see \cite[Definition 1.2]{BK04}) and $I \in \I(A)$, then
\begin{equation*}
\overline{A_{\infty} I A_{\infty}} = \overline{\bigcup_{J\Subset I} J_{\infty}}.
\end{equation*}
\end{enumerate}
Recall that if $A$ is purely infinite (see \cite[Definition 4.1]{KR00}) then it is weakly purely infinite.
\end{remark}

We define here the category of abstract Cuntz semigroups because it will be used as ambient category for the ideal lattice of separable \Cs-algebras.

\begin{definition}\label{def:abstractCu}
Let $S$ be a positively ordered abelian monoid, i.e., an abelian monoid $(S,+,0)$ with partial order $\leq$ such that $0\leq x$ for all $x\in S$, and $x+y\leq x'+y'$ whenever $x,x',y,y'\in S$ and $x\leq x'$ and $y\leq y'$.\
Let $\ll$ denote the way-below relation on $S$, which is defined as follows:\ For any $x,y\in S$, one writes $x\ll y$ if for every increasing sequence $(z_n)_{n\in\N}$ in $S$ with supremum $z$ such that $y\leq z$, there exists $n\in\N$ such that $x\leq z_n$.
One says that $S$ is an \textit{abstract Cuntz semigroup} if the following conditions are satisfied:
\begin{enumerate}[label=\textup{(\roman*)},leftmargin=*]
\item Every increasing sequence in $S$ admits a supremum,
\item For every $x\in S$, there exists an increasing sequence $(x_n)_n$ in $S$ with $x_n\ll x_{n+1}$ for all $n\in\N$ such that $x=\sup_n x_n$,
\item If $x\ll x'$ and $y\ll y'$ for $x,x',y$ and $y'$ in $S$, then $x+y\ll x'+y'$,
\item For every pair of increasing sequences $(x_n)_n$ and $(y_n)_n$ in $S$, $\sup_n (x_n+y_n)=\sup_n x_n + \sup_n y_n$.
\end{enumerate}

We denote by $\Cu$ the category of abstract Cuntz semigroups with morphisms given by ordered monoid homomorphisms that preserve countable increasing suprema and way-below relation.
\end{definition}

%As pointed out in \cite[Remark 2.7]{Gab20}, when an increasing net admits a supremum, then they any $\Cu$-morphism preserves this supremum because this element is also the supremum of a countable $\ll$-increasing sequence.\marginpar{why???}
The Cuntz semigroup of a \Cs-algebra is the motivating example for the category of abstract Cuntz semigroups defined above.

\begin{definition}[{see \cite{Cun78,CEI08,APT11}}]\label{def:Cu(A)}
Let $x$ and $y$ be positive elements in a \Cs-algebra $A$.\
Then $x$ is said to be \textit{Cuntz subequivalent} to $y$ if there exists a sequence $(a_n)_{n\in\N}$ in $A$ such that $x=\lim_{n\to\infty}a_n^*ya_n$.\
Two positive elements $x$ and $y$ are said to be \textit{Cuntz equivalent} if they are Cuntz subequivalent to each other.\
Cuntz subequivalence is a preorder on the set of positive elements of $A$, and Cuntz equivalence is an equivalence relation.\
The \textit{Cuntz semigroup} of $A$, denoted by $\Cus(A)$, is the abelian monoid given by the quotient set of $(A\otimes\K)_+$ by Cuntz equivalence, equipped with the operation induced by direct sum in $(A\otimes\K)_+$.
\end{definition}

\begin{remark}\label{rem:RordamLemma}
We recall here a few results on Cuntz comparison of positive elements in a \Cs-algebra that were established by R{\o}rdam \cite[Section 2]{Ror92}, and Kirchberg--R{\o}rdam \cite[Section 2]{KR00}, \cite[Section 2]{KR02}.\
Let $a$ and $b$ be positive elements of a \Cs-algebra $A$.\
\begin{enumerate}[label=\textup{(\roman*)},leftmargin=*]
\item If $\e>0$ and $\|a-b\|<\e$, then there exists a contraction $c\in A$ such that $(a-\e)_+ = c^*bc$.\
In particular, $(a-\e)_+$ is Cuntz subequivalent to $b$ if $a=_\e b$.
\item If $a$ is Cuntz subequivalent to $b$, then for every $\e>0$ there exist $d_0,d_1\in A$ and $n\in\N$ such that $(a-\e)_+=d_0^*bd_0=d_1^*(b-2^{-n})_+d_1$.
\end{enumerate}
Statement (i) is \cite[Lemma 2.2]{KR02} (see also \cite[Proposition 2.2]{Ror92} and \cite[Lemma 2.5(ii)]{KR00}), while (ii) follows from the following observation.\
By \cite[Proposition 2.4]{Ror92} (or \cite[Proposition 2.6]{KR00}), for every $\e>0$ there exists $\delta>0$ and $d\in A$ such that $(a-\e)_+=d^*(b-\delta)_+d$.\
Then one may argue with \cite[Lemma 2.4(i)]{KR02} to conclude that there exists $d_0\in A$ such that $(a-\e)_+=d_0^*bd_0$.\
Choose $n\in\N$ such that $2^{-n}<\delta$, and set $\delta':=\delta-2^{-n}$.\
It follows from (i) that there exists a contraction $c\in A$ such that $(b-\delta)_+=((b-2^{-n})_+-\delta')_+=c^*(b-2^{-n})_+c$.\
Therefore, we have that $(a-\e)_+=d_1^*(b-2^{-n})_+d_1$, where $d_1:=cd$.
\end{remark}

The object we name below will allow us to view (a subset of) the ideal lattice as an object in $\Cu$.

\begin{notation}\label{not:Isigma}
Let $A$ be a \Cs-algebra.\
We denote by $\I_\sigma(A)\subseteq\I(A)$ the subset of ideals that contain a full element.
\end{notation}

\begin{remark}\label{rem:waybelowideals}
Let $A$ be a \Cs-algebra, $I\in\I(A)$, and $J\in\I_{\sigma}(A)$.\
We observe that $I\Subset J$ if and only if for every increasing sequence of ideals $(I_n)_{n\in\N}$ of $A$ such that $J\subseteq\overline{\bigcup_{n\in\N}I_n}$, there exists $n\in\N$ such that $I\subseteq I_n$.\
Let us show the nontrivial implication.\
Assume that $J\subseteq\overline{\bigcup_{\lambda}I_\lambda}$ for an upward directed family of ideals $\{I_\lambda\}_{\lambda\in\Lambda}$ of $A$, and write $J=\overline{AaA}$ for some $a\in A_+$.\
By Remark \ref{rem:ideallattice}(i) we have that $\overline{A(a-2^{-n})_+A}\Subset J$ for all $n\in\N$.\
Hence, for each $n\in\N$, there exists $\lambda_n\in\Lambda$ such that $\overline{A(a-2^{-n})_+A}\subseteq I_{\lambda_n}$.\
This implies that
\begin{equation*}
J=\overline{\bigcup_{n\in\N}\overline{A(a-2^{-n})_+A}}\subseteq\overline{\bigcup_{n\in\N}I_{\lambda_n}},
\end{equation*}
and by assumption there exists $n\in\N$ such that $I\subseteq I_{\lambda_n}$.\
In light of Remark \ref{rem:increasingfamily}, we have shown the equivalence.
\end{remark}

\begin{proposition}[{cf.\ \cite[Proposition 2.5]{Gab20}}]\label{prop:Cu}
Let $A$ be a $\Cs$-algebra.\
Then $\I_{\sigma}(A)$ is an abstract Cuntz semigroup.\
In particular, $\I(A)$ is an abstract Cuntz semigroup if and only if $\I(A)=\I_{\sigma}(A)$.
\end{proposition}
\begin{proof}
We have that $(\I_{\sigma}(A),\subseteq,+,0)$ is a positively ordered submonoid of $(\I(A),\subseteq,+,0)$, and the way-below relation is equivalent to compact containment by Remark \ref{rem:waybelowideals}.\
In the following paragraph, we show that $\I_{\sigma}(A)$ is an abstract Cuntz semigroup by showing that all the attributes from Definition \ref{def:abstractCu} hold true.\
Let us first show that it is closed under passing to suprema of increasing sequences.\
Fix an increasing sequence of ideals $(I_n)_{n\in\N}$, where $I_n=\overline{Aa_nA}$ and $a_n\in A$ are positive contractions of $A$ for all $n\in\N$.\
Then, it follows that $\sum_{n\in\N}2^{-n}a_n$ is a full element of $\overline{\bigcup_{n\in\N}I_n}$, which is therefore contained in $\I_{\sigma}(A)$.
Now, let $I=\overline{AaA}\in\I_\sigma(A)$ for some positive element $a\in A$.\
Then, the sequence in $\I_{\sigma}(A)$ given by $I_n=\overline{A(a-2^{-n})_+A}$ for $n\in\N$ is such that $\overline{A(a-2^{-n})_+A}\Subset\overline{A(a-2^{-(n+1)})_+A}$ for all $n\in\N$.\
Since $\overline{\bigcup_{n\in\N}I_n}=I$, we have shown the second point in the definition of abstract Cuntz semigroups.\
Suppose now that $I,I'$ are compactly contained in $J,J'$, respectively, where $I,I',J,J'\in\I_{\sigma}(A)$.\
For any increasing sequence $(K_n)_{n\in\N}$ in $\I_{\sigma}(A)$ such that $(I'+J')\subseteq\overline{\bigcup_{n\in\N}K_n}$, there exist $n_I,n_J \in \N$ for which $I\subseteq K_{n_I}$, and $J\subseteq K_{n_J}$.\
This implies that $(I+J)\subseteq K_{\max(n_I,n_J)}$, and thus $(I+J)\Subset(I'+J')$.\
Assume now that $(I_n)_n$ and $(J_n)_n$ are increasing sequences in $\I_{\sigma}(A)$.\
One has that $\overline{\bigcup_{n}I_n}, \overline{\bigcup_{n}J_n} \subseteq \overline{\bigcup_{n}(I_n+J_n)}$.\
Moreover, $\overline{\bigcup_{n}(I_n+J_n)}\subseteq(\overline{\bigcup_{n}I_n}+\overline{\bigcup_{n}J_n})$ because the former is the smallest ideal containing $I_n+J_n$ for all $n\in\N$.
Since from \cite[Corollary 2.3]{Gab20} an ideal $I\in\I(A)$ has a full element if and only if it is the supremum of a $\Subset$-increasing sequence of ideals, it follows that $\I(A)$ is an abstract Cuntz semigroup if and only if $\I(A)=\I_{\sigma}(A)$.
\end{proof}

\begin{definition}
Let $S$ be an abstract Cuntz semigroup.\
An \textit{ideal} $I$ of $S$ is a submonoid $I\subseteq S$ closed under passing to suprema of increasing sequences, and such that $a\leq b$ and $b\in I$ imply that $a\in I$.\
The set of ideals of $S$ is denoted by $\I(S)$, and it is a complete lattice.
%An ideal $I$ of $S$ is said to be \textit{countably generated} if there exists a increasing sequence $(x_n)_n$ in $S$ such that $I=\{x\in S \mid x\leq \sup_{n}x_n\}$.\
%We denote the set of countably generated ideals of $S$ by $\I_{\sigma}(S)$.
\end{definition}

\begin{remark}\label{rem:Cu(A)}
Contrarily to $\I(A)$, the Cuntz semigroup of $A$ is always an object in the category $\Cu$ by \cite[Theorem 1]{CEI08}.\
Therefore, they can be substantially different in general.\
However, the ideal lattice of $\Cus(A)$ is always order-isomorphic to $\I(A)$ via the map that associates $\Cus(I)\in\I(\Cus(A))$ to the ideal $I\in\I(A)$.\
As observed by Gabe in \cite{Gab20}, note that if $B$ is another \Cs-algebra, any order-theoretic isomorphism $\I_{\sigma}(A)\to\I_{\sigma}(B)$ is automatically an isomorphism of abstract Cuntz semigroups.

When $A$ is a purely infinite \Cs-algebra in the sense of \cite[Definition 4.1]{KR00}, it is well-known that two elements of $A_+$ are Cuntz equivalent precisely when they generate the same ideal in $A$.\
Therefore, the map $\Cus(A)\to\I(A)$, realised by associating to the class of an element in $(A\otimes\K)_+$ the ideal it generates as an element of $\I(A\otimes\K)\cong\I(A)$, is an order-theoretic isomorphism onto its image.\
In particular, we claim that $\Cus(A)\cong\I_\sigma(A)$ under this map.\
In fact, an equivalence class $[a]\in\Cus(A)$, where $a\in(A\otimes\K)_+$, is mapped to the ideal in $\I_{\sigma}(A\otimes\K)\cong\I_{\sigma}(A)$ that contains $a$ as a full element.\
Furthermore, for any $I\in\I_{\sigma}(A\otimes\K)\cong\I_{\sigma}(A)$ with full element $a$, we have by pure infiniteness that the only class in $\Cus(A)$ mapping to $I$ is the class of $a$.

On the other hand, under the assumption that $\I_{\sigma}(A)=\I(A)$, such as when $A$ is separable, and that $A$ is purely infinite, one has that $\I(A)$ is order isomorphic to $\Cus(A)$.
\end{remark}

\begin{definition}[{see \cite[Definition 2.6]{Gab20}}]\label{def:Cumorphism}
Let $A,B$ be $\Cs$-algebras.\
A monoid order homomorphism $\Phi:\I(A)\to\I(B)$ is said to be a \textit{$\Cu$-morphism} if it preserves increasing suprema and compact containment.
\end{definition}

\begin{notation}
Let $A,B$ be $\Cs$-algebras, and $\f:A\to B$ a $\ast$-homomorphism.\
Denote by $\I(\f):\I(A)\to\I(B)$ the order preserving map given by
\begin{equation*}
 \I(\f)(I)=\overline{B\f(I)B} \in\I(B)
 \end{equation*}
 for all $I\in\I(A)$.\
 We denote by $\I_{\sigma}(\f)$ the restriction of $\I(\f)$ to $\I_{\sigma}(A)$.
\end{notation}

\begin{lemma}
Let $A$ and $B$ be $\Cs$-algebras.\
For any $\ast$-homomorphism $\f:A\to B$, the map $\I(\f)$ is a $\Cu$-morphism, and $\I_{\sigma}(\f)$ is a $\Cu$-morphism with range in $\I_{\sigma}(B)$.
\end{lemma}
\begin{proof}
By \cite[Lemma 2.12(iii)]{Gab20}, $\I(\f)$ is a $\Cu$-morphism.\
Moreover, it follows from \cite[Corollary 2.3]{Gab20}, that the image of an ideal with a full element under a $\Cu$-morphism contains a full element.\
Hence, the image of $\I_{\sigma}(A)$ under $\I_{\sigma}(\f)$ is contained in $\I_{\sigma}(B)$.
\end{proof}

\begin{remark}[{see \cite[Remark 2.8]{Gab20}}]
Note that, for any pair of \Cs-algebras $A,B$, a map $\Phi:\I(A)\to\I(B)$ is a $\Cu$-morphism if and only if it preserves suprema of arbitrary families of ideals, and compact containment.\
In fact, if $\Phi$ preserves suprema, then it must be an ordered monoid homomorphism as a consequence of the following basic observations.
First, we have that $\Phi$ preserves zero element and sums because $\sup\big\{\{0\}\big\}=\{0\}$, and $\sup\{I,J\}=I+J$ for all $I,J\in\I(A)$.\
Moreover, the order is preserved because, when $I\subseteq J$, one has that $\Phi(I)\subseteq\sup\{\Phi(I),\Phi(J)\}=\Phi(\sup\{I,J\})=\Phi(J)$.
\end{remark}

\begin{notation}
Let $\Phi,\Psi:\I(A) \to \I(B)$ be $\Cu$-morphisms.\
One writes $\Phi\leq \Psi$ whenever $\Phi(I)\subseteq\Psi(I)$ for all $I\in\I(A)$.
\end{notation}

Observe that $\I_{\sigma}(-)$ (respectively, $\I(-)$) is functorial on the category of (separable) $\Cs$-algebras with $\ast$-homomorphisms as arrows to the category $\Cu$.
%If one drops separability, then $\I(-)$ need not be a functor, but $\I_{\sigma}(-)$ is.

\begin{definition}[{see \cite[Definition 3.13.7]{Ped18}}]
Let $A$ be a \Cs-algebra and $I\subseteq A$ a closed two-sided ideal.\
$I$ is said to be \textit{prime} if whenever $J,K\subseteq A$ are ideals such that $J\cap K\subseteq I$, then $J\subseteq I$ or $K\subseteq I$.\
$I$ is said to be \textit{primitive} if there exists a non-zero irreducible representation $\pi$ of $A$ such that $\ker(\pi)=I$.
\end{definition}

\begin{notation}[{see \cite[Definition 4.1.2]{Ped18}}]
We will denote by $\Prim(A)$ the \textit{primitive ideal space} of $A$, i.e., the set containing all primitive ideals of $A$.\
Given a subset $U\subseteq\Prim(A)$, the closure of $U$ is given by
\begin{equation*}
\overline{U}=\bigg\{\mathfrak{p}\in\Prim(A) \mid \mathfrak{p} \supseteq \bigcap_{\q\in U} \q \bigg\}.
\end{equation*}
The closure operation defined above satisfies Kuratowski's closure axioms, and thus defines a topology on $\Prim(A)$, which is usually called Jacobson (or hull-kernel) topology.\
Throughout the present article, we endow $\Prim(A)$ with the Jacobson topology.
\end{notation}

\begin{remark}\label{rem:prodprim}
By \cite[Propositions 2.16(iii)+2.17(2)]{BK04}, if $A$ and $B$ are \Cs-algebras, and at least one of them is exact, then the map given by
\begin{equation*}
\Prim(A)\times\Prim(B)\to\Prim(A\otimes B), \quad (\p,\q)\mapsto (\p\otimes B)+(A\otimes \q)
\end{equation*}
is a homeomorphism.
\end{remark}

\begin{definition}
Let $X$ be a topological space.
An open subset $P\subseteq X$ is said to be \textit{prime} if for every pair of open subsets $V,W\subseteq X$ such that $V\cap W\subseteq P$, it holds that $V\subseteq P$ or $W\subseteq P$.\
A prime set $P$ is \textit{proper} when $P\neq X$.
\end{definition}

\begin{definition}[{see \cite[Definition 2, p.\ 477]{MLM92}}]
A topological space $X$ is said to be \textit{sober} if for each proper prime open subset $P$ of $X$ there exists a unique point $x\in X$ such that $P=X \setminus \overline{\{x\}}$.
\end{definition}

\begin{remark}
First, note that any Hausdorff space is sober, and any sober space is $T_0$ (see \cite[Theorem 3, p.\ 477]{MLM92}).\
However, a sober space need not be $T_1$, and vice versa.
Moreover, as a consequence of \cite[Theorem 2, p.\ 479]{MLM92}, a sober space $X$ is determined up to homeomorphism by its lattice of open subsets $\mathcal{O}(X)$ in the sense that, if $Y$ is another sober space with $\mathcal{O}(Y)\cong\mathcal{O}(X)$, then $X$ is homeomorphic to $Y$.
\end{remark}

\begin{remark}\label{rem:isomideals}
Let $A$ be a \Cs-algebra.\
Recall from \cite[Theorem 4.1.3]{Ped18} that the lattice of open subsets of $\Prim(A)$ is order-isomorphic to $\I(A)$.\
In particular, an open subset $U\subseteq\Prim(A)$ is sent to the ideal $I\in\I(A)$ given by $I=\bigcap_{\p\in\Prim(A)\setminus U}\p$ via this order isomorphism, while an ideal $I\in\I(A)$ is sent to $\{\p \mid \p\nsupseteq I\}$ via its inverse.\
We know from \cite[Proposition 4.4.4]{Ped18} that $\Prim(A)$ is locally quasi-compact\footnote{We follow the convention that a locally compact space is Hausdorff, while a locally quasi-compact space need not be Hausdorff.}.\
Observe that, from what was said before, if $B$ is a \Cs-algebra, and $f:\Prim(A)\to\Prim(B)$ a homeomorphism, then there exists a unique order isomorphism between $\I(A)$ and $\I(B)$ that sends $\p\in\Prim(A)$ to $f(\p)\in\Prim(B)$.

If $A$ is separable, it follows from \cite[Proposition 4.3.6]{Ped18} that the set of prime ideals of $A$ coincides with the set of primitive ideals\footnote{Every primitive ideal is prime, but the converse is not true in general.}, which is equivalent to $\Prim(A)$ being a sober space by \cite[Proposition 4.4.14]{Ped18}.\
In this case, $\Prim(A)$ is also second-countable by \cite[Theorem 4.4.13]{Ped18}.
\end{remark}

\begin{definition}
Let $\alpha:G\curvearrowright A$ be an action on a $\Cs$-algebra.\
Denote by $\alpha^{\sharp}$ the set-theoretic action induced by $\alpha$ on the ideal lattice of $A$, i.e.,
\begin{equation*}
\alpha^{\sharp}:G\curvearrowright \I(A),\quad \alpha^{\sharp}_g(I)=\{\alpha_g(a) \mid a\in I\}
\end{equation*}
for every $g\in G$ and $I\in\I(A)$.

With slight abuse of notation, we also denote by $\alpha^{\sharp}$ the restriction of $\alpha^{\sharp}$ to the primitive ideal space of $A$.
\end{definition}

\begin{lemma}\label{lem:contaction}
Let $\alpha:G\curvearrowright A$ be an action on a $\Cs$-algebra.\
Then $\alpha^{\sharp}:G\curvearrowright\Prim(A)$ is continuous with respect to the Jacobson topology.
\end{lemma}
\begin{proof}
We want to prove that, for any given open set $U\subseteq\Prim(A)$, the set given by
\begin{equation*}
\{(g,\p)\in G\times\Prim(A) \mid \alpha^{\sharp}_g(\p)\in U\}
\end{equation*}
is open in $G\times\Prim(A)$ endowed with the product topology.\
Equivalently, one can show that the set
\begin{equation*}
W = \{(g,\p)\in G\times\Prim(A) \mid \alpha^{\sharp}_g(\p)\in \Prim(A)\setminus U\}
\end{equation*}
is closed.\
Let us verify that this is true.

First of all, note that from Remark \ref{rem:isomideals} there exists a unique ideal $I\in\I(A)$ corresponding to $U$, and $W$ can be written as
\begin{equation*}
W = \{(g,\p)\in G\times\Prim(A) \mid \alpha^{\sharp}_g(\p)\supseteq I\}.
\end{equation*}
Consider a net $\{(g_{\lambda},\p_{\lambda})\}_{\lambda\in\Lambda}$ in $W$ converging to an element $(g,\p)\in G\times\Prim(A)$.\
In particular, it follows that $g_{\lambda}\to g$, and $\p_{\lambda}\to\p$.\
Let $e\in I$ be a positive element.\
By continuity of $\alpha$, we have that for any $\epsilon>0$ there exists an open neighbourhood $H$ around $1_G$ such that $\alpha_h(e)=_{\e}e$ for all $h\in H$.\
It follows from Remark \ref{rem:RordamLemma}(i) that $(e-\e)_+$ is Cuntz subequivalent to $\alpha_h(e)$ for all $h\in H$.\
Hence, we have that $(e-\e)_+ \in \bigcap_{h\in H}\alpha_h(I)$.\
By passing to a subnet, we may assume that $gg_{\lambda}^{-1}\in H$ for all $\lambda\in\Lambda$.

Now, since $\p_{\lambda}\to\p$, it is certainly true that $\p$ belongs to the closure of $\{\p_{\lambda} \mid \lambda\in\Lambda\}$, which means that $\bigcap_{\lambda\in\Lambda}\p_\lambda \subseteq \p$.\
Note moreover that $(g_\lambda,\p_\lambda)\in W$ means that $\alpha_{g_{\lambda}}^{\sharp}(\p_\lambda)\supseteq I$, or equivalently, $\p_{\lambda}\supseteq(\alpha^{\sharp}_{g_{\lambda}})^{-1}(I)$ for all $\lambda\in\Lambda$.\
Hence, we obtain that
\begin{equation*}
\alpha^{\sharp}_{g}(\p) \supseteq \bigcap_{\lambda\in\Lambda}\alpha^{\sharp}_{g}(\p_\lambda) \supseteq \bigcap_{\lambda\in\Lambda}\alpha^{\sharp}_{gg_{\lambda}^{-1}}(I) \ni (e-\e)_+.
\end{equation*}
Since $\e$ was arbitrarily chosen at the beginning, we get that $e\in \alpha^{\sharp}_g(\p)$.\
This shows that $I_+ \subseteq \alpha^{\sharp}_g(\p)$, and therefore we get that $I \subseteq \alpha^{\sharp}_g(\p)$, and $(g,\p) \in W$.
\end{proof}

\begin{remark}
Observe that for every cocycle morphism $(\f,\bbu):(A,\alpha) \to (B,\beta)$, the $\Cu$-morphism $\I(\f): \I(A)\to\I(B)$ is equivariant with respect to $\alpha^{\sharp}$ and $\beta^{\sharp}$.\
This follows from the fact that unitarily equivalent $\ast$-homomorphisms induce the same $\Cu$-morphism.\
As one may expect, $\I_\sigma(-)$ (respectively, $\I(-)$) becomes a functor from the category of (separable) $G$-\Cs-algebras and proper cocycle morphisms as arrows to the category of abstract Cuntz semigroups endowed with order-theoretic $G$-actions and arrows given by $G$-equivariant $\Cu$-morphisms.
\end{remark}

\begin{remark}\label{rem:invariants}
We record here an observation that will be used in Theorem \ref{thm:classification}.\
Let $\alpha:G\curvearrowright A$ and $\beta:G\curvearrowright B$ be actions on \Cs-algebras.\
Let $f:(\Prim(A),\alpha^{\sharp})\to(\Prim(B),\beta^{\sharp})$ be a conjugacy.\
As a consequence of Remark \ref{rem:isomideals}, $f$ induces a unique order isomorphism $\Phi:\I(A)\to\I(B)$ such that $\Phi(\p)=f(\p)$ for all $\p\in\Prim(A)$.\
Moreover, $\Phi$ is $\alpha^{\sharp}$-to-$\beta^{\sharp}$ equivariant because
\begin{align*}
\beta^{\sharp}_g \circ \Phi(I)&=\beta^{\sharp}_g \circ \Phi\bigg(\bigcap\{\p\in\Prim(A) \mid \p\supseteq I\}\bigg)=\bigcap\{\beta^{\sharp}_g \circ f(\p)\in\Prim(A) \mid \p\supseteq I\}\\
&=\bigcap\{f \circ \alpha^{\sharp}_g(\p)\in\Prim(A) \mid \p\supseteq I\}=\Phi \circ \alpha^{\sharp}_g(I)
\end{align*}
for all $g\in G$ and $I\in\I(A)$.
\end{remark}

\section{Uniqueness results}\label{sec:uniqueness}

In this section, we work our way towards the uniqueness result underpinning our classification theorem.\
Although the uniqueness and existence results contribute equally to the overarching structure of the present work, we should point out to the reader that, on the methodological level, the technical novelty is contained in Section \ref{sec:existence}.

\begin{definition}[see {\cite[Definition 3.3]{GS22a}}]\label{def:wc}
Let $(\f,\bbu),(\psi,\bbv):(A,\alpha)\to(B,\beta)$ be two cocycle morphisms.\
We say that $(\f,\bbu)$ \textit{approximately 1-dominates} $(\psi,\bbv)$, if for every compact subset $K\subseteq G$, finite subset $\mathcal{F}\subset A$, $\e>0$, and contraction $b\in B$, there exists $c\in B$ such that
\begin{equation}\tag{e1}\label{eq:e1}
\max_{g\in K} \| b^*\bbv_g\beta_g(b) - c^*\bbu_g\beta_g(c) \| \leq \e,
\end{equation}
and
\begin{equation}\tag{e2}\label{eq:e2}
\max_{a\in\mathcal{F}} \| b^*\psi(a)b - c^* \f(a) c \| \leq \e.
\end{equation}
\end{definition}

\begin{remark}\label{rem:contraction}
Note that, in Definition \ref{def:wc}, one may always choose $c$ as a contraction:\
Let a quadruple $(K,\mathcal{F},\e, b)$ as above be given.
Choose $\eta>0$ such that $(1+\max_{a\in\mathcal F}\|a\|)\eta\leq\e/2$.
If $c\in B$ is chosen to satisfy conditions \eqref{eq:e1} and \eqref{eq:e2} for the quadruple $(K\cup\{1_G\},\mathcal{F},\eta, b)$, then the first condition at $g=1_G$ entails that $\|c\|^2 \leq \|b^*b\|+\|b^*b-c^*c\|\leq 1+\eta$.\
Then $c_0=(1+\eta)^{-1/2} c\in B$ is a contraction satisfying
\[
 \| b^*\bbv_g\beta_g(b) - c_0^*\bbu_g\beta_g(c_0) \| \leq \eta+ \| c_0^*\bbu_g\beta_g(c_0) ((1+\eta)-1)  \| \leq \e/2+\e/2=\e
\]
for all $g\in K$, and
\[
\| b^*\psi(a)b - c_0^* \f(a) c_0 \| \leq \eta+ \| c_0^* \f(a) c_0 ((1+\eta)-1) \| \leq \e/2+\e/2=\e
\]
for all $a\in\mathcal F$.
\end{remark}

\begin{notation}
Let $(\f,\bbu),(\psi,\bbv):(A,\alpha)\to(B,\beta)$ be two cocycle morphisms.\
We say that \textit{$\f$ approximately 1-dominates $\psi$ as ordinary $\ast$-homomorphisms} if $(\f,\1)$ approximately 1-dominates $(\psi,\1)$ in the sense of Definition \ref{def:wc} when viewed as proper cocycle morphisms with respect to the action of $G=\{1\}$.
\end{notation}

\begin{remark}\label{rem:app-equiv}
Let $A$, $B$ be \Cs-algebras, and $\f,\psi:A\to B$ two $\ast$-homomorphisms.\
Suppose that for every finite subset $\mathcal{F}\subseteq A$, $\e>0$, and contraction $b\in B$, there exists $c\in B$ such that
\begin{equation*}
\max_{a\in\mathcal{F}} \| b^*\psi(a)b - c^* \f(a) c \| \leq \e.
\end{equation*}
Then it follows that $\I(\psi)\leq\I(\f)$.
Let us show why.
%First observe that for any $I\in\I(A)$, the ideal $\I(\psi)(I)\subseteq B$ is generated by $\{b^*\psi(a)b \mid a\in I_+, b\in B_{\leq1}\}$.
Fix an ideal $I\in\I(A)$.
For any element $a\in I$ and $\e>0$, we may find a positive contraction $b\in B$ such that $\psi(a) =_{\e} b^*\psi(a)b$.
By assumption, for $a,\e$ and $b$ given as above, there exists $c\in B$ such that
\begin{equation*}
  \psi(a) =_{\e} b^*\psi(a)b =_{\e} c^* \f(a) c \in \overline{B\f(I)B}.
\end{equation*}
Since $a\in I$ and $\e>0$ were arbitrary, it follows that $\psi(I)\subseteq\I(\f)(I)$.
\end{remark}

\begin{lemma}\label{lem:inductivelim}
Let $\alpha:G\curvearrowright A$ and $\beta:G\curvearrowright B$ be actions on \Cs-algebras.\
Let $\f,\psi:(A,\alpha)\to(B,\beta)$ be equivariant $\ast$-homomorphisms.\
Then the following statements hold true.
\begin{enumerate}[label=\textup{(\roman*)},leftmargin=*]
\item Assume that $A$ is separable.\
If $(\f,\1)$ approximately 1-dominates $(\psi,\1)$, there exists a separable, $\beta$-invariant \Cs-subalgebra $D\subseteq B$ such that $(\f,\1)$ approximately 1-dominates $(\psi,\1)$ when corestricted to $D$.\label{item:first}
\item Let $\kappa:(B,\beta)\hookrightarrow (B\otimes\K,\beta\otimes\id_{\K})$ be the equivariant inclusion given by $\kappa(b)=b\otimes e_{1,1}$.\
Then, if $(\f,\1)$ approximately 1-dominates $(\psi,\1)$, $\kappa \circ (\f,\1)$ approximately 1-dominates $\kappa \circ (\psi,\1)$.\label{item:second}
\item $(\f,\1)$ approximately 1-dominates $(\psi,\1)$ if and only if $\iota_\infty \circ (\f,\1)$ approximately 1-dominates $\iota_\infty \circ (\psi,\1)$.\footnote{Here, $\iota_\infty$ must be viewed as a map into the $\beta$-continuous sequence algebra $B_{\infty,\beta}$.}\label{item:third}
\end{enumerate}
\end{lemma}
\begin{proof}
We start by showing \ref{item:first}.\
Denote by $D_0$ the separable, $\beta$-invariant \Cs-subalgebra of $B$ generated by the union of the image of $\f$ and $\psi$.\
Fix an increasing sequence of finite subsets $\mathcal{F}_n\subseteq A$ such that $\overline{\bigcup_{n\in\N}\mathcal{F}_n} = A$, and an increasing sequence of compact subsets $K_n\subseteq G$ such that $\bigcup_{n\in\N}K_n^{\circ}=G$.\
Choose a countable approximate unit of contractions $(d^{(0)}_i)_{i\in\N}\subseteq D_0$ for $D_0$.\
By assumption, there exists a contraction $c^{(0)}_i\in B$ (cf.\ Remark \ref{rem:contraction}) such that
\begin{align*}
\max_{g\in K_0}\|(d^{(0)}_i)^*\beta_g(d^{(0)}_i) - (c^{(0)}_i)^*\beta_g(c^{(0)}_i)\|&\leq 1, \\
\max_{a\in\mathcal{F}_0}\|(d^{(0)}_i)^*\psi(a)d^{(0)}_i - (c^{(0)}_i)^*\f(a)c^{(0)}_i\|&\leq 1
\end{align*}
for all $i\in\N$.
Let $D_1$ be the separable, $\beta$-invariant \Cs-subalgebra of $B$ generated by $D_0$ and $(c^{(0)}_i)_{i\in\N}$.\
Inductively define $D_{n+1}\subseteq B$ for $n\geq 1$ in the following way.\
Suppose $D_n\subseteq B$ is a separable, $\beta$-invariant \Cs-subalgebra of $B$ containing $D_{n-1}$, and pick a countable approximate unit of contractions $(d^{(n)}_i)_{i\in\N}\subseteq D_n$ for $D_n$.\
By assumption, there exists a contraction $c^{(n)}_i\in B$ that satisfies
\begin{align*}
\max_{g\in K_n}\|(d^{(n)}_i)^*\beta_g(d^{(n)}_i) - (c^{(n)}_i)^*\beta_g(c^{(n)}_i)\|&\leq 2^{-n}, \\
\max_{a\in\mathcal{F}_n}\|(d^{(n)}_i)^*\psi(a)d^{(n)}_i - (c^{(n)}_i)^*\f(a)c^{(n)}_i\|&\leq 2^{-n}
\end{align*}
for all $i\in\N$.\
Let $D_{n+1}$ be the separable, $\beta$-invariant \Cs-subalgebra of $B$ generated by $D_n$ and $(c^{(n)}_i)_{i\in\N}$.\
Finally, set $D=\overline{\bigcup_{n\in\N}D_n}$.\
Let us check that $D$ is the algebra we are looking for.\
Fix a finite subset $\mathcal{F}\subseteq A$, a compact subset $K\subseteq G$, $\e>0$ and a contraction $b\in D$.\
If we choose $n\in\N$ large enough, we have that $2^{-n}\leq\e$, $K\subseteq K_n^{\circ}$, and there exists some $k_n\in\N$ such that $d^{(n)}_{k_n} b =_{\e} b$.\
Since the increasing union of the $\mathcal{F}_n$ was dense in $A$, we may assume (upon making $n$ larger if necessary) that for every $x\in\mathcal F$ there is $a\in\mathcal{F}_n$ with $x=_\e a$.
By construction, the contraction $c=c^{(n)}_{k_n}\in D$ satisfies
\begin{align*}
\max_{g\in K_n}\|d^{(n)*}_{k_n} \beta_g(d^{(n)}_{k_n}) - c^*\beta_g(c)\|&\leq \e, \\
\max_{a\in\mathcal{F}_n}\|d^{(n)*}_{k_n}\psi(a)d^{(n)}_{k_n} - c^*\f(a)c\|&\leq \e.
\end{align*}
Therefore, we end up with
\[
\max_{g\in K}\|b^*\beta_g(b) - (cb)^*\beta_g(cb)\| \leq 2\e + \max_{g\in K}\|b^*(d^{(n)*}_{k_n} \beta_g(d^{(n)}_{k_n}) - c^*\beta_g(c)) \beta_g(b)\| \leq 3\e
\]
and
\[
\begin{array}{cl}
\multicolumn{2}{l}{ \displaystyle \max_{a\in\mathcal{F}} \|b^*\psi(a)b - (cb)^*\f(a)cb\| } \\
\leq& \displaystyle 2\e+\max_{a\in\mathcal{F}_n}\|b^*\psi(a)b - (cb)^*\f(a)cb\| \\
\leq& \displaystyle 4\e+\max_{a\in\mathcal{F}_n}\|b^*( d^{(n)*}_{k_n}\psi(a)d^{(n)}_{k_n} - c^*\f(a)c )b\| \ \leq \ 5\e.
\end{array}
\]
Up to rescaling $\e$, this shows that $(\f,\1)$ approximately 1-dominates $(\psi,\1)$ when corestricted to $D$.

Now we show \ref{item:second}.\
Fix $\e>0$, a finite subset $\mathcal{F}\subseteq A$, a compact subset $K\subseteq G$, and a contraction $d\in B\otimes\K$.\
Let $b\in B$ and $e\in\K$ be contractions such that $(b\otimes e)d=_{\e}d$ and $e^* e_{1,1} e=e_{1,1}$.\
Since $(\f,\1)$ approximately 1-dominates $(\psi,\1)$, there exists a contraction $c\in B$ such that
\begin{align*}
\max_{g\in K}\|b^*\beta_g(b) - c^*\beta_g(c)\|&\leq \e, \\
\max_{a\in\mathcal{F}}\|b^*\psi(a)b - c^*\f(a)c\|&\leq \e.
\end{align*}
It follows that
\begin{align*}
\max_{g\in K}\|d^*(\beta\otimes\id_{\K})_g(d) - d^*(c\otimes e)^*(\beta\otimes\id_{\K})_g((c\otimes e)d)\|\leq 3\e,
\end{align*}
and
\begin{align*}
\max_{a\in\mathcal{F}}\|d^*(\psi(a)\otimes e_{1,1})d - d^*(c\otimes e)^*(\f(a)\otimes e_{1,1})(c\otimes e)d\| \leq \e+2\e\max_{a\in\mathcal{F}}\|a\|.
\end{align*}
Up to suitably rescaling $\e$, this shows that $\kappa \circ (\f,\1)$ approximately 1-dominates $\kappa \circ (\psi,\1)$.

To show \ref{item:third}, assume that $(\f,\1)$ approximately 1-dominates $(\psi,\1)$.\
Let $b\in B_{\infty,\beta}$ be a contraction, which we may represent by a sequence of contractions $(b_n)_{n\in\N} \in \ell^\infty_{\beta}(\N,B)$.
Choose an increasing sequence of finite subsets $\mathcal{F}_n\subseteq A$ with dense union and increasing compact subsets $K_n\subseteq G$ with $G=\bigcup_{n\geq 1} K_n^\circ$.
For every $n\geq 1$, we may find a contraction $c_n\in B$ satisfying
\begin{align*}
\max_{g\in K_n}\| b_n^*\beta_{g}(b_n) - c_n^*\beta_{g}(c_n)\| & \leq 2^{-n},\\
\max_{a\in\mathcal{F}_n }\|b_n^*\psi(a)b_n-c_n^*\f(a)c_n\| & \leq 2^{-n}.
\end{align*}
Denote by $c\in B_{\infty}$ the element represented by the sequence $(c_n)_n$.\
Then clearly $b^*\beta_{\infty,g}(b)=c^*\beta_{\infty,g}(c)$ and $b^*\psi(a)b=c^*\f(a)c$ for all $g\in G$ and $a\in A$.
We observe for every $g\in G$ that
\begin{align*}
\|\beta_{\infty,g}(c)-c\|^2 & =\|\beta_{\infty,g}(c^*c)+c^*c-\beta_{\infty,g}(c)^*c-c^*\beta_{\infty,g}(c)\| \\
&= \|\beta_{\infty,g}(b^*b)+b^*b-\beta_{\infty,g}(b)^*b-b^*\beta_{\infty,g}(b)\| \\
&= \|\beta_{\infty,g}(b)-b\|^2.
\end{align*}
It follows that $c\in B_{\infty,\beta}$, and $\iota_\infty \circ (\f,\1)$ approximately 1-dominates $\iota_\infty \circ (\psi,\1)$.

For the other direction, suppose that $\iota_\infty \circ (\f,\1)$ approximately 1-dominates $\iota_\infty \circ (\psi,\1)$.\
It follows that for any contraction $b\in B$, finite subset $\mathcal{F}\subseteq A$, compact subset $K\subseteq G$, and $\e>0$, we may find $c\in B_{\infty,\beta}$ with representing sequence $(c_n)_{n\in\N}\in\ell^\infty_\beta(\N,B)$, such that
\begin{align*}
\max_{g\in K}\ \limsup_{n\to\infty}\|b^*\beta_g(b)-c_n^*\beta_g(c_n)\| & \leq \frac{\e}{2}, \\
\max_{a\in\mathcal{F}}\ \limsup_{n\to\infty}\|b^*\psi(a)b-c_n^*\f(a)c_n\| & \leq \frac{\e}{2}.
\end{align*}
Hence, one may find a large enough $n\in\N$ such that
\begin{align*}
\max_{g\in K}\|b^*\beta_g(b)-c_n^*\beta_g(c_n)\| & \leq \e, \\
\max_{a\in\mathcal{F}}\|b^*\psi(a)b-c_n^*\f(a)c_n\| & \leq \e,
\end{align*}
and therefore $(\f,\1)$ approximately 1-dominates $(\psi,\1)$.
\end{proof}

\begin{theorem}\label{thm:weak-cont}
Let $A$ be an exact \Cs-algebra, and $B$ a stable, $\Oinf$-stable \Cs-algebra.\
If $\f,\psi:A\to B_{\infty}$ are nuclear $\ast$-homomorphisms such that  $\I(\f)\geq\I(\psi)$, then $\f$ approximately 1-dominates $\psi$.
\end{theorem}
\begin{proof}
To prove that $\f$ approximately 1-dominates $\psi$, fix a finite set $\mathcal{F}\subseteq A$, $\e>0$, and a contraction $b\in B_\infty$.\
From \cite[Theorem 3.3]{Gab20} one has that there exist elements $c_i\in B_{\infty}$, for $1\leq i\leq n$, such that
\begin{equation*}
  \bigg\|\sum_{i=1}^n c_i^*\f(a)c_i - \psi(a)\bigg\| \leq \e
\end{equation*}
for all $a\in\mathcal{F}$.\
Since $B$ is $\Oinf$-stable, by \cite[Proposition 3.18]{Gab20} there are elements $(s_k)_{k\in\N}$ in $B_\infty\cap\f(C^*(\mathcal{F}))'$ such that $s_i+B_{\infty}\cap\f(C^*(\mathcal{F}))^\perp$ form a sequence of isometries with orthogonal range projections.\
Note that
\begin{equation*}
\sum_{i=1}^n c_i^*\f(a)c_i = \sum_{i=1}^n c_i^*s_i^*\f(a)s_ic_i = c^*\f(a)c,
\end{equation*}
where $c=\sum_{i=1}^ns_ic_i$.\
Hence $c^*\f(a)c =_{\e} \psi(a)$ for all $a\in\mathcal{F}$.\
One can assume that $c$ is a contraction, for if it is not, we may choose a positive contraction $d\in A$ such that $dad =_\e a$ for all $a\in\mathcal{F}$, replace $\mathcal{F}$ with $\mathcal{F}'=\{d^2\}\cup\{dad\mid a\in\mathcal{F}\}$, and make the choice of $c$ accordingly.\
Then the element $c'=\f(d)c$ has norm at most $1+\e$ and satisfies $c'^*\f(a)c' = c^*\f(dad)c =_{\e} \psi(dad) =_\e \psi(a)$ for all $a\in\mathcal{F}$.
Since this argument works for all $\e>0$ and arbitrary finite subsets of $A$, we may in fact obtain a contraction $c_0\in B_\infty$ with $c_0^*\f(a)c_0=\psi(a)$ for all $a\in \mathcal{F}$.

Next, let $D$ be a separable \Cs-subalgebra of $B_\infty$ containing $\{\f(\mathcal{F}), \psi(\mathcal{F}), b,c_0\}$.\
Invoking \cite[Lemma 1.21]{GS22b}, we have a $\ast$-homomorphism $\rho:D\otimes\K\to B_\infty$ such that $\rho(d\otimes e_{1,1})=d$ for all $d\in D$, making the following diagram commutative,
\[\begin{tikzcd}
	{D} && {B_{\infty}} \\
	& {D \otimes \K}
	\arrow[from=1-1, to=1-3]
	\arrow["{\id_D\otimes e_{1,1}}"'{pos=0.3}, from=1-1, to=2-2]
	\arrow["{\rho}"'{pos=0.7}from=2-2, to=1-3]
\end{tikzcd}\]
By \cite[Lemma 7.4(i)]{KR02} there exists an isometry $s\in\M(D\otimes\K)$ such that $s^*(\f(a)\otimes e_{1,1})s=_{\e}(c_0^*\f(a)c_0)\otimes e_{1,1}$ for all $a\in\mathcal{F}$.\
Let $\hat{c}=\rho(s(b\otimes e_{1,1}))\in B_\infty$.\
We have that
\begin{equation*}
  \hat{c}^*\f(a)\hat{c} =\rho((b\otimes e_{1,1})^*s^*(\f(a)\otimes e_{1,1})s(b\otimes e_{1,1}))=_{\e} b^*\psi(a)b
\end{equation*}
for all $a\in\mathcal{F}$, which implies condition \eqref{eq:e2} of Definition \ref{def:wc}.\
Additionally, we have that $b^*b=\hat{c}^*\hat{c}$, which in particular implies condition \eqref{eq:e1} of Definition \ref{def:wc}, and therefore that $\f$ approximately 1-dominates $\psi$.
\end{proof}

The following lemma is, together with results from \cite{GS22a}, the driving force behind the dynamical Kirchberg--Phillips theorem \cite{GS22b}.
In the context of this article, it will be crucial in the uniqueness result, Theorem \ref{thm:uniqueness}, and remains an indispensable technical device in the overarching structure of our work.

\begin{lemma}[{see \cite[Lemma 3.16]{GS22b}}]\label{lem:3.15}
Let $\alpha:G\curvearrowright A$ and $\beta:G\curvearrowright B$ be two actions on separable \Cs-algebras, and assume that $\beta$ is amenable and isometrically shift-absorbing.\
Let $(\f,\bbu),(\psi,\bbv):(A,\alpha)\to(B,\beta)$ be two proper cocycle morphisms.\
If $\f$ approximately 1-dominates $\psi$ as ordinary $\ast$-homomorphisms, then $(\f,\bbu)$ approximately 1-dominates $(\psi,\bbv)$.
\end{lemma}

\begin{definition}\label{def:cuntz_sum}
Let $\alpha:G\curvearrowright A$ and $\beta:G\curvearrowright B$ be actions on \Cs-algebras.\
Suppose there exists a pair of isometries $s_1,s_2\in \M(B)^\beta$ satisfying $s_1s_1^*+s_2s_2^*=\1$.
One may define the \textit{Cuntz sum} of elements $a,b\in\M(B)$ as
\begin{equation*}
a \oplus_{s_1,s_2} b = s_1as_1^*+s_2bs_2^* \in\M(B).
\end{equation*}
Let $(\f,\bbu),(\psi,\bbv):(A,\alpha)\to(B,\beta)$ be two (proper) cocycle morphisms.\
The Cuntz sum of $(\f,\bbu)$ and $(\psi,\bbv)$ is the (proper) cocycle morphism $(\f\oplus_{s_1,s_2}\psi,\bbu\oplus_{s_1,s_2}\bbv):(A,\alpha)\to(B,\beta)$ given pointwise by $(\f\oplus_{s_1,s_2}\psi)(a)=\f(a)\oplus_{s_1,s_2}\psi(a)$, and $(\bbu\oplus_{s_1,s_2}\bbv)_g=\bbu_g\oplus_{s_1,s_2}\bbv_g$.

Note that whenever $a,b\in B$, their Cuntz sum lies in $B$ as well.\
Moreover, $a\oplus_{s_1,s_2} b$ is unitarily equivalent to $b\oplus_{s_1,s_2} a$ via the unitary $U=s_1s_2^*+s_2s_1^*$, and if one picks another pair of isometries $t_1,t_2\in\M(B)^{\beta}$ such that $t_1t_1^*+t_2t_2^*=\1$, then $a\oplus_{s_1,s_2}b$ is unitarily equivalent to $a\oplus_{t_1,t_2}b$ via the unitary $V=s_1t_1^*+s_2t_2^*\in\M(B)^{\beta}$.\
Given $a,b\in\M(B)$, we write ``$a \oplus b$" as a shorthand notation for ``$a \oplus_{s_1,s_2} b$".
\end{definition}

\begin{lemma}\label{lem:2.8}
Let $\alpha:G\curvearrowright A$ and $\beta:G\curvearrowright B$ be two actions on separable \Cs-algebras.
Suppose that $\beta$ is strongly stable.\
Let $(\f,\bbu),(\psi,\bbv):(A,\alpha)\to(B,\beta)$ be two proper cocycle morphisms such that $(\f,\bbu)$ and $(\psi,\bbv)$ approximately 1-dominate each other.\
Suppose there exists a unital embedding $\O2 \hookrightarrow \M(B)^\beta$ commuting with the range of $\f,\psi,\bbu$ and $\bbv$.
Then $(\f,\bbu)$ and $(\psi,\bbv)$ are strongly asymptotically unitarily equivalent.
\end{lemma}
\begin{proof}
The hypotheses of \cite[Lemma 2.8]{GS22b} are satisfied symmetrically, entailing that $(\f,\bbu)$ is strongly asymptotically unitarily equivalent to $(\f\oplus\psi,\bbu\oplus\bbv)$, and $(\psi,\bbv)$ is strongly asymptotically unitarily equivalent to $(\psi\oplus\f,\bbv\oplus\bbu)$.\
Hence, by transitivity of strong asymptotic unitary equivalence, it suffices to prove that $(\f\oplus\psi,\bbu\oplus\bbv)$ and $(\psi\oplus\f,\bbv\oplus\bbu)$ are strongly asymptotically unitarily equivalent.

Recall from Definition \ref{def:cuntz_sum} that $(\f\oplus\psi,\bbu\oplus\bbv)$ and $(\psi\oplus\f,\bbv\oplus\bbu)$ are unitarily equivalent via a unitary $U\in\M(B)^\beta$.\
Since the hypotheses of \cite[Lemma 2.7]{GS22b} hold true, we conclude that the equivariant automorphism given by $\Ad(U)$ is strongly asymptotically unitarily equivalent to $\id_B$ via a norm-continuous path of unitaries $u:[0,\infty)\to\U(\1+B)$.\
As a result, $u$ is the norm-continuous path that guarantees strong asymptotic unitary equivalence between $(\f\oplus\psi,\bbu\oplus\bbv)$ and $(\psi\oplus\f,\bbv\oplus\bbu)$.
\end{proof}

We are ready to prove the main result of this section:

\begin{theorem}[Uniqueness]\label{thm:uniqueness}
Let $\alpha:G\curvearrowright A$ be an action on a separable, exact $\Cs$-algebra, and $\beta:G\curvearrowright B$ an amenable, strongly stable, equivariantly $\O2$-stable, isometrically shift-absorbing action on a separable \Cs-algebra.\
If $(\f,\bbu),(\psi,\bbv):(A,\alpha)\to(B,\beta)$ are two proper cocycle morphisms with $\f$ and $\psi$ nuclear, then $\I(\f)=\I(\psi)$ if and only if $(\f,\bbu)$ and $(\psi,\bbv)$ are strongly asymptotically unitarily equivalent.
\end{theorem}
\begin{proof}
If $(\f,\bbu)$ and $(\psi,\bbv)$ are strongly asymptotically unitarily equivalent, then it follows straightaway that $\I(\f)=\I(\psi)$.

Assume now that $\I(\f)=\I(\psi)$.\
By Remark \ref{rem:thm5.6}, there exists a proper cocycle conjugacy
\begin{equation*}
(\theta,\mathbbm{x}) : (B,\beta) \to (B\otimes\O2, \beta\otimes\id_{\O2})
\end{equation*}
that is strongly asymptotically unitarily equivalent to the first factor embedding $\id_{B}\otimes\1_{\O2}$ with trivial cocycle.\
Let us show that the proper cocycle morphisms $(\f\otimes\1_{\O2},\bbu\otimes\1_{\O2})$ and $(\psi\otimes\1_{\O2},\bbv\otimes\1_{\O2})$ are strongly asymptotically unitarily equivalent.\
Since $\f\otimes\1_{\O2}$ and $\psi\otimes\1_{\O2}$ are nuclear and $\I(\f\otimes\1_{\O2})=\I(\psi\otimes\1_{\O2})$, Theorem \ref{thm:weak-cont} ensures that $\iota_\infty \circ (\f\otimes\1_{\O2})$ and $\iota_\infty \circ (\psi\otimes\1_{\O2})$ approximately 1-dominate each other.\
Thus, we may conclude with Lemma \ref{lem:inductivelim}\ref{item:third} that $\f\otimes\1_{\O2}$ and $\psi\otimes\1_{\O2}$ approximately 1-dominate each other.\
We can then apply Lemma \ref{lem:3.15} to get that $(\f\otimes\1_{\O2},\bbu\otimes\1_{\O2})$ and $(\psi\otimes\1_{\O2},\bbv\otimes\1_{\O2})$ approximately 1-dominate each other.\
By Lemma \ref{lem:2.8} it follows that $(\f\otimes\1_{\O2},\bbu\otimes\1_{\O2})$ and $(\psi\otimes\1_{\O2},\bbv\otimes\1_{\O2})$ are strongly asymptotically unitarily equivalent.\
As a consequence, we have that $(\f,\bbu)$ is strongly asymptotically unitarily equivalent to $(\theta,\mathbbm{x})^{-1} \circ (\f\otimes\1_{\O2},\bbu\otimes\1_{\O2})$, which is strongly asymptotically unitarily equivalent to $(\theta,\mathbbm{x})^{-1} \circ (\psi\otimes\1_{\O2},\bbv\otimes\1_{\O2})$, and the latter is strongly asymptotically unitarily equivalent to $(\psi,\bbv)$.\
By transitivity of strong asymptotic unitary equivalence, it follows that $(\f,\bbu)$ is strongly asymptotically unitarily equivalent to $(\psi,\bbv)$.
\end{proof}

We derive a unital version of the uniqueness theorem when $G$ is exact.\
Note that if $G$ is not exact, then an action $\beta$ as in the following corollary cannot exist as a consequence of \cite[Corollary 3.6]{OS21}.

\begin{corollary}\label{cor:uniqueness}
Assume $G$ is exact.\
Let $\alpha:G\curvearrowright A$ be an action on a separable, exact, unital $\Cs$-algebra, and $\beta:G\curvearrowright B$ an amenable,  equivariantly $\O2$-stable, isometrically shift-absorbing action on a separable, unital \Cs-algebra.\
If $(\f,\bbu),(\psi,\bbv):(A,\alpha)\to(B,\beta)$ are two proper cocycle morphisms with $\f$ and $\psi$ nuclear and unital, then $\I(\f)=\I(\psi)$ if and only if $(\f,\bbu)$ and $(\psi,\bbv)$ are strongly asymptotically unitarily equivalent.
\end{corollary}
\begin{proof}
Let us denote by $B^s$ the \Cs-algebra $B\otimes\K$ and by $\beta^s$ the action $\beta\otimes\id_{\K}:G\curvearrowright B\otimes\K$.
Consider the canonical inclusion
\begin{equation*}
\kappa: (B,\beta) \hookrightarrow (B^s,\beta^s), \quad \kappa(b)=b\otimes e_{1,1},
\end{equation*}
which is $\beta$-to-$\beta^s$-equivariant.\
We apply Theorem \ref{thm:uniqueness} to $\kappa \circ (\f,\bbu)$ and $\kappa \circ (\psi,\bbv)$,\footnote{Recall from Remark \ref{rem:composition} that $\kappa \circ (\f,\bbu)=(\kappa \circ \f, \kappa^{\dagger}(\bbu_{\bullet}))$.} and conclude that they are strongly asymptotically unitarily equivalent.\
Let us call by $v:[0,\infty) \to \U(\1+B^s)$ the norm-continuous path of unitaries witnessing the equivalence.
Then we have that
\begin{equation*}
\lim_{t\to\infty}\|[v_t, \1_B\otimes e_{1,1}]\|=0.
\end{equation*}
After possibly removing an initial segment of $v$, if necessary, the norm-continuous path of unitaries given by
\begin{equation*}
u_t=\kappa^{-1}\bigg((\1_B\otimes e_{1,1})v_t(\1_B\otimes e_{1,1}) \cdot |(\1_B\otimes e_{1,1})v_t(\1_B\otimes e_{1,1})|^{-1}\bigg)\in\U(B)
\end{equation*}
for $t\in[0,\infty)$, satisfies
\begin{align*}
&\lim_{t\to\infty}\|u_t\f(a)u_t^*-\psi(a)\|=0, \\
&\lim_{t\to\infty}\max_{g\in K}\|u_t\bbu_g\beta_g(u_t)^*-\bbv_g\|=0
\end{align*}
for all $a\in A$ and compact sets $K\subseteq G$.\
Moreover, one can arrange that $u_0=\1$ because $\U(B)$ is connected as a consequence of $B$ being $\O2$-stable (cf.\ \cite[Proposition 4.1.5]{Ror02} and \cite[Theorems 1.4+1.9+2.3]{Cun81}).\
As a result, $(\f,\bbu)$ and $(\psi,\bbv)$ are strongly asymptotically unitarily equivalent.
\end{proof}

\section{Existence results}\label{sec:existence}

In this section, we establish a number of technical lemmas that build up to the existence theorem, Theorem \ref{thm:existence}, which underpins our final classification result.\
Some of the intermediate results that we find on the way can be of independent interest, and may have applications elesewhere.

First, we establish useful notation that will be used throughout in this section.

\begin{remark}\label{rem:beta}
Let $\beta:G\curvearrowright B$ be an action on a \Cs-algebra.\
Recall from \cite[Corollary 3.4]{APT73} that $\M(\C_0(G,B))=\C_b^s(G,\M(B))$, where the latter is the \Cs-algebra of bounded strictly continuous functions on $G$ with values in $\M(B)$.\
Observe that $\beta$ induces a strictly-continuous $G$-action $\bar{\beta}$ on $\C_b^s(G,\M(B))$ given by
\begin{equation*}
\bar{\beta}_g(f)(h)=\beta_g(f(g^{-1}h))
\end{equation*}
for all $f\in\C_b^s(G,\M(B))$ and $g,h\in G$.
%Note that the notation $\overline{\beta}$ is also used for the action by linear isometries induced by $\beta$ on the pre-Hilbert \Cs-module $\C_c(G,B)$ as given in Notation \ref{not:Hilbert}, but this should not create confusion.
\end{remark}

\begin{notation}\label{not:psihat}
Let $\alpha:G\curvearrowright A$ and $\beta:G\curvearrowright B$ be actions on \Cs-algebras.\
Given a $\ast$-homomorphism $\psi:A\to\M(B)$, we denote by $\hat{\psi}$ the $\ast$-homomorphism from $A$ to $\M(\C_0(G,B))=\C_b^s(G,\M(B))$ given by
\begin{equation*}
\hat{\psi}(a)(g)=\beta_g(\psi(\alpha_{g^{-1}}(a)))
\end{equation*}
for all $a\in A$ and $g\in G$.\
The abuse of notation is reasonable as the actions appearing in the definition of $\hat{\psi}$ will be clear from context.\
By construction, we have that $\hat{\psi}$ is $\alpha$-to-$\bar{\beta}$-equivariant.\
In fact,
\begin{align*}
\bar{\beta}_g(\hat{\psi}(a))(h)& =\beta_g(\hat{\psi}(a)(g^{-1}h))\\
							& =\beta_g \circ \beta_{g^{-1}h} \circ \psi(\alpha_{(g^{-1}h)^{-1}}(a))\\
							& =\hat{\psi}(\alpha_g(a))(h)
\end{align*}
for all $a \in A$ and $g,h \in G$.
\end{notation}

The following lemma is well-known among experts.\
Since we struggled to pin down a direct reference, we give a short proof for the reader's convenience.

\begin{lemma} \label{lem:nuclear}
Let $X$ be a locally compact space, and $A,B$ \Cs-algebras.\
A c.p.\ map $\f:A\to\C_0(X,B)$ is nuclear if and only if $\ev_x \circ \f:A\to B$ is nuclear for all $x\in X$.
\end{lemma}
\begin{proof}
Only the ``if'' part is non-trivial.
It is well-known (see \cite[Proposition 2.4.2]{BO08} and its proof) that the nuclearity of $\C_0(X)$ is witnessed by two nets of c.p.c.\ maps $\kappa_\lambda: \C_0(X)\to\mathbb{C}^{n_\lambda}$ and $\psi_\lambda: \mathbb{C}^{n_\lambda}\to\C_0(X)$ with $\psi_\lambda\circ\kappa_\lambda\to\operatorname{id}_{\C_0(X)}$ in point-norm, where for each $\lambda$ the map $\kappa_\lambda$ is of the form $\kappa_\lambda=\ev_{x_1}\oplus\ev_{x_2}\oplus\dots\oplus\ev_{x_{n_\lambda}}$, for finitely many points $x_1,\dots,x_{n_\lambda}\in X$.\
By tensoring this approximate factorization with $B$, we obtain that $(\psi_\lambda\otimes\operatorname{id}_B)\circ(\kappa_\lambda\otimes\operatorname{id}_B)\to\operatorname{id}_{\C_0(X,B)}$ in point-norm.

Assuming that $\ev_x \circ \f:A\to B$ is nuclear for all $x\in X$, it thus follows that $(\kappa_\lambda\otimes\operatorname{id}_B)\circ\f$ is nuclear for every $\lambda$, since it is a finite direct sum of evaluation maps.\
We can thus conclude that $\f=\lim_\lambda (\psi_\lambda\otimes\operatorname{id}_B)\circ(\kappa_\lambda\otimes\operatorname{id}_B)\circ\f$ is the point-norm limit of nuclear c.p.\ maps, so it is itself nuclear.
\end{proof}

\begin{definition}
Let $A$ and $B$ be two \Cs-algebras.\
A c.p.\ map $\f:A\to\M(B)$ is said to be \textit{weakly nuclear} if the c.p.\ map $b^*\f(-)b:A\to B$ is nuclear for all $b\in B$.
\end{definition}

\begin{lemma}\label{lem:existence3}
Let $\alpha:G\curvearrowright A$ and $\beta:G\curvearrowright B$ be actions on \Cs-algebras.\
For any weakly nuclear $\ast$-homomorphism $\psi:A\to\M(B)$, the $\ast$-homomorphism $\hat{\psi}:A\to\M(\C_0(G,B))$ defined in Notation \ref{not:psihat} is weakly nuclear.
\end{lemma}
\begin{proof}
%Equivalently, $\hat{\psi}$ is nuclear if and only if the maps given by $e_n^*\hat{\psi}(-)e_n$ for all $n\in\N$ are nuclear for some approximate unit $\{e_n\}_{n\in\N}\subseteq\C_0(G,B)$.
%Since there exists an approximate unit of elementary tensors for $\C_0(G,B)$, it suffices to show that $x^*\hat{\psi}(-)x$ is nuclear for $x=f\otimes b\in\C_c(G)\odot B$.
For each $f\in\C_0(G,B)$, we know that
\[
\ev_g(f^*\hat{\psi}(-)f)=f(g)^*(\beta_g\psi(\alpha_{g^{-1}}(-)))f(g)= \beta_g\Big( \beta_g^{-1}(f(g))^* \psi(\alpha_{g^{-1}}(-)) \beta_g^{-1}(f(g)) \Big)
\]
is nuclear for all $g\in G$.\
Hence, we may apply Lemma \ref{lem:nuclear} to conclude that $f^*\hat{\psi}(-)f$ is nuclear.\
This implies that $\hat{\psi}$ is weakly nuclear.
\end{proof}

\begin{lemma}\label{lem:existence2}
Let $\alpha:G\curvearrowright A$ and $\beta:G\curvearrowright B$ be actions on separable \Cs-algebras.\
If $\psi:A\to\M(B)$ is a $\ast$-homomorphism inducing an equivariant $\Cu$-morphism $\overline{B\psi(-)B}:(\I(A),\alpha^{\sharp})\to(\I(B),\beta^{\sharp})$, then the $\ast$-homomorphism $\hat{\psi}:A\to\M(\C_0(G,B))$ defined in Notation \ref{not:psihat} satisfies
\begin{equation*}
\overline{\C_0(G,B)\hat{\psi}(I)\C_0(G,B)}=\C_0(G,\overline{B\psi(I)B})
\end{equation*}
for all $I\in\I(A)$.
\end{lemma}
\begin{proof}
Combining Remarks \ref{rem:prodprim} and \ref{rem:isomideals} we can write the ideal generated by $\hat{\psi}(I)$ in $\C_0(G,B)$ as
\begin{equation*}
\bigcap\big\{\ker(\ev_g\otimes\pi) \mid \pi\big(\overline{B (\beta_g \circ \psi \circ \alpha_{g^{-1}}(I))B}\big)=0\big\},
\end{equation*}
where $g\in G$ and $\pi$ ranges over all non-zero irreducible representations of $B$.\
Since we assumed $\overline{B\psi(-)B}$ to be equivariant, we have that
\begin{equation*}
\overline{B (\beta_g \circ \psi \circ \alpha_{g^{-1}}(I))B}=\beta^{\sharp}_g\big(\overline{B \psi( \alpha^{\sharp}_{g^{-1}}(I))B}\big)=\overline{B\psi(I)B}
\end{equation*}
for all $g\in G$.
Hence, the ideal generated by $\hat{\psi}(I)$ in $\C_0(G,B)$ is
\begin{equation*}
\bigcap\big\{\ker(\ev_g\otimes\pi) \mid \pi(\overline{B\psi(I)B})=0\big\}=\C_0(G,\overline{B\psi(I)B}),
\end{equation*}
where $g\in G$ and $\pi$ ranges over all non-zero irreducible representations of $B$.
\end{proof}

Now we prove an equivariant analogue of \cite[Corollary 5.7]{Gab20}.\
We record two versions of the lemma.\
First a more general result, which may be of independent interest, and then a second and more restrictive variant that is specifically tailored to our need.

\begin{lemma}\label{lem:existence1}
Let $\alpha:G\curvearrowright A$ and $\beta:G\curvearrowright B$ be a pair of actions on separable $\Cs$-algebras.\
For any $\ast$-homomorphism $\psi:A\to\M(B)$ inducing an equivariant $\Cu$-morphism $\overline{B\psi(-)B}:(\I(A),\alpha^{\sharp})\to(\I(B),\beta^{\sharp})$ there exists an equivariant $\ast$-homomorphism
\begin{equation*}
\kappa:(A,\alpha)\to(\M(\K(\Hil_G)\otimes B),\Ad(\lambda)\otimes\beta)
\end{equation*}
such that
\begin{equation}\label{eq:property-lemma}
\overline{(\K(\Hil_G)\otimes B)\kappa(I)(\K(\Hil_G)\otimes B)}=\K(\Hil_G)\otimes\overline{B\psi(I)B}
\end{equation}
for all $I\in\I(A)$.

If $A$ is moreover exact, and $\psi$ is weakly nuclear, then $\kappa$ is nuclear.\
\end{lemma}
\begin{proof}
Let us denote by $\pi:\C_0(G) \hookrightarrow \B(L^2(G))$ the representation of $\C_0(G)$ as multiplication operators on $L^2(G)$.
Note that the $\ast$-homomorphism
\begin{equation*}
\C_b^s(G,\M(B))=\M(\C_0(G)\otimes B) \lhook\joinrel\xrightarrow{\pi \otimes \id_B} \M(\K\otimes B)
\end{equation*}
is equivariant with respect to $\bar{\beta}$ and $\Ad(\lambda)\otimes\beta$.

We may now start the construction of an $\alpha$-to-($\Ad(\lambda)\otimes\beta$) equivariant map $A \to \M(\K(\Hil_G)\otimes B)$ satisfying \eqref{eq:property-lemma}.\
Consider the $\alpha$-to-$\bar{\beta}$-equivariant $\ast$-homomorphism $\hat{\psi}:A\to \M(\C_0(G,B))$ given as in Notation \ref{not:psihat}.\
Define an $\alpha$-to-$\Ad(\lambda)\otimes\beta$-equivariant $\ast$-homomorphism given by
\begin{equation*}
\kappa=(\pi\otimes\id_B) \circ \hat{\psi}: A\to\M(\K(\Hil_G)\otimes B).
\end{equation*}
Observing that $\kappa(I)$ and $(\pi\otimes\id_B)(\overline{\C_0(G,B)\hat{\psi}(I)\C_0(G,B)})$ generate the same ideal in $\K(\Hil_G)\otimes B$, and by Lemma \ref{lem:existence2}
\begin{equation*}
\overline{\C_0(G,B)\hat{\psi}(I)\C_0(G,B)}=\C_0(G,\overline{B\psi(I)B}),
\end{equation*}
we may conclude that $\kappa(I)$ generates the same ideal as $\pi(\C_0(G))\otimes\overline{B\psi(I)B}$ in $\K(\Hil_G)\otimes B$, which is $\K(\Hil_G) \otimes \overline{B\psi(I)B}$.

For the moreover part, assume that $\psi$ is weakly nuclear and $A$ exact.\
By Lemma \ref{lem:existence3}, it follows that $\hat{\psi}$ is weakly nuclear.\
We infer from \cite[Proposition 3.2]{Gab22} that $\hat{\psi}$ is nuclear if and only if it is weakly nuclear.\
In particular, this implies that $\kappa$ is nuclear.
\end{proof}

The following result is an adaptation of the previous lemma to a more specific setting that will ultimately lead us to the existence theorem at the end of this section.\
%Observe that Lemma \ref{lem:existence} holds for any actions, as we only have assumptions on the \Cs-algebras.
It is important to note that the central ingredient of our proof is \cite[Theorem 14.1]{Gab21}, which ultimately uses \cite[Lemma 4.5]{BGSW22}.

\begin{lemma}\label{lem:existence}
Let $\alpha:G\curvearrowright A$ be an action on a separable, exact $\Cs$-algebra, and $\beta:G\curvearrowright B$ an action on a separable $\Cs$-algebra.\
Then, for every equivariant $\Cu$-morphism $\Phi: (\I(A),\alpha^{\sharp})\to(\I(B),\beta^{\sharp})$, there exists an equivariant, nuclear $\ast$-homomorphism
\begin{equation*}
\kappa:(A,\alpha)\to(\M(\K(\Hil^\infty_G)\otimes B),\Ad(\lambda^\infty)\otimes\beta)
\end{equation*}
such that
\begin{equation}\label{eq:existence}
\overline{(\K(\Hil^\infty_G)\otimes B)\kappa(I)(\K(\Hil^\infty_G)\otimes B)}=\K(\Hil^\infty_G)\otimes\Phi(I)
\end{equation}
for all $I\in\I(A)$.
\end{lemma}
\begin{proof}
By \cite[Theorem 14.1]{Gab21} there exists a nuclear $\ast$-homomorphism $\psi:A\to \Oinf\otimes B$ such that $\overline{(\Oinf\otimes B)\psi(I)(\Oinf\otimes B)}=\Oinf\otimes\Phi(I)$ for all $I\in\I(A)$.\
Hence, we may apply Lemma \ref{lem:existence1} and get an $\alpha$-to-$(\Ad(\lambda)\otimes\id_{\Oinf}\otimes\beta)$-equivariant $\ast$-homomorphism $\kappa_1:A\to\M(\K(\Hil_G)\otimes\Oinf\otimes B)$ that satisfies property \eqref{eq:property-lemma} with $B$ replaced by $\Oinf\otimes B$.

We may now pick a unital copy of $\Oinf$ inside $\B(\ell^2(\N))$, and denote by $\iota$ the unital inclusion $\Oinf \hookrightarrow \B(\ell^2(\N))$.
Then, we have that
\begin{equation*}
\K(\Hil_G)\otimes\Oinf \xhookrightarrow{\id_{\K(\Hil_G)}\otimes\iota} \K(\Hil_G)\otimes\B(\ell^2(\N)) \subseteq \B(\Hil^\infty_G).
\end{equation*}
Observe that $\id\otimes\iota$ is equivariant with respect to $\Ad(\lambda)\otimes\id_{\Oinf}$ and $\Ad(\lambda^\infty)$.\
The map given by $\kappa=(\id_{\K(\Hil_G)}\otimes\iota\otimes\id_B) \circ \kappa_1$ is an $\alpha$-to-$(\Ad(\lambda^\infty)\otimes\beta)$-equivariant $\ast$-homomorphism satisfying
\eqref{eq:existence}.

Since $A$ is exact and $\psi$ is weakly nuclear, it follows from the moreover part of Lemma \ref{lem:existence1} that $\kappa_1$ is nuclear.\
This implies that $\kappa$ is nuclear as well.
\end{proof}

\begin{remark}\label{rem:cptinclusion}
It was observed in \cite[Remark 4.5]{GS22b} that there exists an equivariant embedding $\iota:(\K(\Hil^\infty_G),\Ad(\lambda^\infty))\hookrightarrow(\Oinf,\gamma)$, given on rank one operators\footnote{A rank one operator $E_{\xi,\eta}$ for $\xi,\eta\in\Hil_G^\infty\setminus\{0\}$, is defined by $E_{\xi,\eta}(\zeta)=\eta \cdot \langle \xi , \zeta \rangle$ for all $\zeta\in\Hil_G^\infty$.} by $\iota(E_{\xi,\eta})=\mathfrak{s}(\xi)\mathfrak{s}(\eta)^*$ for all $\xi,\eta$ non-zero vectors in $\Hil_G^\infty$.
\end{remark}

In the following lemma we establish a general property of separable \Cs-algebras that shows that it is enough to know a fullness-type condition on a countable set in order for it to be true on the entire set of positive elements.\
We note that the countable subset defined in the lemma is independent of the \Cs-algebra $B$ and the $\ast$-homomorphism $\f:A\to B$.

\begin{lemma}\label{lem:countabledense}
Let $A$ be a \Cs-algebra, $S\subset A_+$ a set of positive elements that admits a countable dense subset $Q_0\subseteq S$, and
\begin{equation*}
Q=\bigg\{\frac{(q-\e)_+}{\|(q-\e)_+\|} \in A_+ \mid q\in Q_0,\ \e\in\mathbb{Q},\ 0<\e<\|q\|\bigg\}.
\end{equation*}
Consider the following statements for any \Cs-algebra $B$, $\Cu$-morphism $\Phi:\I(A)\to\I(B)$, and $\ast$-homomorphism $\f:A\to B$.
\begin{enumerate}[label=\textup{(\roman*)}, leftmargin=*]
\item $\f(q)$ is a full element of $\Phi(\overline{AqA})$ for all $q\in Q$,\label{item:i}
\item $\f(a)$ is a full element of $\Phi(\overline{AaA})$ for all $a\in S$.\label{item:ii}
\end{enumerate}
Then, \ref{item:i}$\Longrightarrow$\ref{item:ii}.
\end{lemma}
\begin{proof}
Fix a \Cs-algebra $B$, a $\Cu$-morphism $\Phi:\I(A)\to\I(B)$, and a $\ast$-homomorphism $\f:A\to B$, and suppose that $\f(q)$ is a full element of $\Phi(\overline{AqA})$ for all $q\in Q$.\
Pick a positive element $a\in S$, and a sequence $(q_n)_{n\in\N}$ in $Q_0$ such that $\lim_{n\to\infty}q_n=a$.\
For any $\delta\in\mathbb{Q}_{>0}$, we may find $\e\in\mathbb{Q}_{>0}$ with $\delta>\e$, and by Remark \ref{rem:RordamLemma}(ii) there exists $n_\delta\in\N$ for which
\begin{equation*}
(a-\delta)_+ \in \overline{A(q_n-\e)_+A} \subseteq \overline{AaA}
\end{equation*}
for all $n\geq n_\delta$.
By assumption, $\frac{1}{\|(q_n-\e)_+\|}\f((q_n-\e)_+)$ is a full element of $\Phi\bigg(\overline{A\frac{(q_n-\e)_+}{\|(q_n-\e)_+\|}A}\bigg)=\Phi(\overline{A(q_n-\e)_+A})$ for all $n\in\N$, hence $\f((q_n-\e)_+)$ is also a full element of $\Phi(\overline{A(q_n-\e)_+A})$ for all $n\in\N$.\
Therefore, we have that
\begin{equation*}
\Phi(\overline{A(a-\delta)_+A}) \subseteq \Phi(\overline{A(q_n-\e)_+A})=\I_{\sigma}(\f)(\overline{A(q_n-\e)_+A}) \subseteq \I_{\sigma}(\f)(\overline{AaA})
\end{equation*}
for all $n\geq n_\delta$.\
As a consequence, the ideal generated by $\f(a)$ in $B$ contains $\Phi(\overline{A(a-\delta)_+A})$ for any given $\delta\in\mathbb{Q}_{>0}$.\
Moreover, we know that $\overline{B\f(a)B}$ contains $\Phi(\overline{AaA})=\overline{\bigcup_{\delta\in\mathbb{Q}_{>0}}\Phi(\overline{A(a-\delta)_+A})}$ as well, where in the last identity we used that $\Phi$ preserves countable increasing suprema.\
Note that
\begin{equation*}
\f((a-\delta)_+) \in \I_{\sigma}(\f)(\overline{A(q_n-\e)_+A})=\Phi(\overline{A(q_n-\e)_+A}) \subseteq \Phi(\overline{AaA})
\end{equation*}
for all $n\geq n_\delta$.
This implies that $(\f(a)-\delta)_+$ belongs to $\Phi(\overline{AaA})$ for all $\delta\in\mathbb{Q}_{>0}$, and therefore
\begin{equation*}
\overline{B\f(a)B} = \overline{\bigcup_{\delta\in\mathbb{Q}_{>0}}\overline{B(\f(a)-\delta)_+B}} \subseteq \Phi(\overline{AaA}).
\end{equation*}
The two observations together give that $\overline{B\f(a)B}=\Phi(\overline{AaA})$.\
In particular, $\f(a)$ is a full element of $\Phi(\overline{AaA})$.
\end{proof}

\begin{definition}
Let $A$ and $B$ be \Cs-algebras.\
We say that a c.p.c.\ map $\f:A\to B_{\infty}$ is \textit{nuclearly liftable} if it admits a lift $(\f_n)_{n\in\N}:A\to \prod_{\N}B$ consisting of nuclear c.p.c.\ maps.

Suppose that $B$ is equipped with an action $\beta:G\curvearrowright B$.\
Then, we say that a c.p.c.\ map $\f:A\to B_{\infty,\beta}$ is nuclearly liftable if it is nuclearly liftable as a map with range in $B_\infty$.
\end{definition}

\begin{remark}\label{rem:nuclift}
Let $A$ and $B$ be \Cs-algebras, and assume that $A$ is separable and exact.\
By combining \cite[Proposition 3.3]{Dad97} and the Choi--Effros lifting theorem, \cite[Theorem 3.10]{CE76}, a c.p.c.\ map $\f:A\to B_{\infty}$ is nuclearly liftable precisely when it is nuclear.
\end{remark}

\begin{lemma}\label{lem:seqargument}
Let $\alpha:G\curvearrowright A$ be an action on a separable, exact \Cs-algebra, and $\beta:G\curvearrowright B$ an equivariantly $\O2$-stable and isometrically shift-absorbing action on a separable \Cs-algebra.\
Let $\Phi:(\I(A),\alpha^{\sharp})\to(\I(B),\beta^{\sharp})$ be an equivariant $\Cu$-morphism.\
Assume that there exists a nuclearly liftable, equivariant $\ast$-homomorphism $\f:A\to (B\otimes\K(\Hil_G^\infty)\otimes\O2)_{\infty,\beta\otimes\Ad(\lambda^\infty)\otimes\id_{\O2}}$ with the property that $\f(a)$ is a full element of $\I(\iota_\infty) (\Phi(\overline{AaA})\otimes\K(\Hil_G^\infty)\otimes\O2)$ for all $a\in A_+$.\footnote{Recall from Definition \ref{def:sequencealgebra} that $\iota_{\infty}$ denotes the embedding into the sequence algebra, not its continuous part.}
Then, there exists a nuclearly liftable, equivariant $\ast$-homomorphism $\psi:A\to B_{\infty,\beta}$ such that $\psi(a)$ is a full element of $\I(\iota_\infty)(\Phi(\overline{AaA}))$ for all $a\in A_+$.
\end{lemma}
\begin{proof}
Since $A$ is separable, we may fix a countable dense subset $Q_0\subseteq A_+$, and consider the set $Q\subseteq A_+$ of norm one elements given as in Lemma \ref{lem:countabledense}.\
Since $Q$ is countable, we can write it as $Q=\{q_i\}_{i\in\N}$.\
For every $i\in\N$, if $\Phi(\overline{Aq_iA})\neq\{0\}$ we choose a strictly positive norm-one element $h_i^0\in\Phi(\overline{Aq_iA})$, otherwise we set $h_i^0=0$.\
Moreover, we fix a rank-one projection $e\in\K(\Hil_G^\infty)$, so that $h_i=h_i^0\otimes e\otimes 1_{\O2}$ is a full element of $\Phi(\overline{Aq_iA})\otimes\K(\Hil_G^\infty)\otimes\O2$, and thus also full in $\I(\iota_{\infty})(\Phi(\overline{Aq_iA})\otimes\K(\Hil_G^\infty)\otimes\O2)$ for all $i\in\N$.\
Moreover, for any given $i\in\N$, the ideal $\I(\iota_{\infty})(\Phi(\overline{Aq_iA})\otimes\K(\Hil_G^\infty)\otimes\O2)$ is purely infinite by \cite[Propositions 4.3+4.20]{KR00}.\
Therefore, $\f(q_i)$ and $h_i$ are Cuntz equivalent, and by Remark \ref{rem:RordamLemma}(ii) we can find $K_i\in\N$ and elements $c_{i,k}$, $d_{i,k}$ in $\I(\iota_{\infty})(\Phi(\overline{Aq_iA})\otimes\K(\Hil_G^\infty)\otimes\O2)$ such that
\begin{align*}
(h_i-2^{-k})_+ = c_{i,k}^* \f(q_i) c_{i,k} &&
\text{and} && (\f(q_i)-2^{-k})_+ = d_{i,k}^* (h_i-2^{-K_i})_+ d_{i,k}
\end{align*}
for all $i,k\in\N$.\

Recall from Remark \ref{rem:cptinclusion} that there exists a full equivariant embedding
\[
\iota:(\K(\Hil_G^\infty),\Ad(\lambda^\infty))\hookrightarrow(\Oinf,\gamma).
\]
Since $\beta$ is isometrically shift-absorbing and equivariantly $\O2$-stable, one can use Remark \ref{rem:isa} and \cite[Lemma 2.12]{Sza18} to obtain an equivariant $\ast$-homomorphism $\eta:(\Oinf\otimes\O2,\gamma\otimes\id_{\O2})\hookrightarrow (F_{\infty,\beta}(B),\tilde{\beta}_\infty)$.\
Note that we have a canonical commutative diagram of equivariant $\ast$-homomorphisms given by
\[\begin{tikzcd}
	{(B,\beta)} && {(B_{\infty,\beta},\beta_\infty)} \\
	& {(B \otimes_{\max} F_{\infty,\beta}(B),\beta\otimes\tilde{\beta}_\infty)}
	\arrow[from=1-1, to=1-3]
	\arrow["{\id_B\otimes\1}"'{pos=0.3}, from=1-1, to=2-2]
	\arrow["{\rho}"'{pos=0.7}from=2-2, to=1-3]
\end{tikzcd}\]
where $\rho$ is the natural $\ast$-homomorphism defined as $\rho(b\otimes (x+B_{\infty,\beta}\cap B^\perp))=bx$, for all $b\in B$, $x\in B_{\infty,\beta}\cap B'$.\
%In the following paragraph, let us denote by $D$ the image of $B\otimes\K(\Hil_G^\infty)\otimes\O2$ under $\rho$.\
We define the equivariant $\ast$-homomorphism
\begin{align*}
&\theta:(B\otimes\K(\Hil_G^\infty)\otimes\O2,\beta\otimes\Ad(\lambda^\infty)\otimes\id_{\O2})\to(B_{\infty,\beta},\beta_\infty),\\
&\theta=\rho \circ (\id_B\otimes\eta) \circ (\id_B\otimes\iota\otimes\id_{\O2}).
\end{align*}
Let us show that $\theta(I\otimes\K(\Hil_G^\infty)\otimes\O2)$ generates the ideal $\I(\iota_{\infty})(I)$ in $B_\infty$ for all $I\in\I(B)$.
\begin{align*}
\I(\theta)(I\otimes\K(\Hil_G^\infty)\otimes\O2) &= \I(\rho) \circ \I(\id_B\otimes\eta) \circ \I(\id_B\otimes\iota\otimes\id_{\O2})(I\otimes\K(\Hil_G^\infty)\otimes\O2) \\
&=\I(\rho) \circ \I(\id_B\otimes\eta)(I\otimes\Oinf\otimes\O2)\\
&=\I(\rho)(I\otimes F_{\infty,\beta}(B))=\I(\iota_{\infty})(I).
\end{align*}
In particular, for each $i\in\N$, we have that either $\Phi(\overline{Aq_iA})=0$ or $\theta(h_i)$ and $h_i^0$ are full, norm one elements of $\I(\iota_{\infty})(\Phi(\overline{Aq_iA}))$.\
By \cite[Propositions 4.3+4.5+4.20]{KR00} the latter \Cs-algebra is purely infinite.\
Hence, by using Remark \ref{rem:RordamLemma}(ii), we can find $L_i\in\N$ and elements $y_{i,k}$, $z_{i,k}$ in $B_\infty$ satisfying
\begin{align*}
(h_i^0-2^{-k})_+ = y_{i,k}^*\theta((h_i-2^{-L_i})_+)y_{i,k} &&
\text{and} && (\theta(h_i)-2^{-k})_+ = z_{i,k}^*h_i^0z_{i,k}
\end{align*}
for all $i,k\in\N$.\
Now, we consider
\begin{equation*}
\theta_\infty: (B\otimes\K(\Hil_G^\infty)\otimes\O2)_\infty \to (B_{\infty})_{\infty},
\end{equation*}
the componentwise application of $\theta$.\
The composite $\ast$-homomorphism
\begin{equation*}
\Theta=\theta_\infty \circ \f: A \to (B_{\infty,\beta})_{\infty,\beta_\infty}
\end{equation*}
is nuclearly liftable and equivariant with respect to $\alpha$ and $(\beta_\infty)_\infty$.\
Set $e_{i,k}=\theta_{\infty}(c_{i,L_i})y_{i,k}$ and $f_{i,k}=z_{i,K_i}\theta_{\infty}(d_{i,k})$ as elements of $(B_{\infty})_{\infty}$.
Then, we have that
\begin{align*}
(h_i^0-2^{-k})_+=e_{i,k}^*\Theta(q_i)e_{i,k} && \text{and} && (\Theta(q_i)-2^{-k})_+=f_{i,k}^*h_i^0f_{i,k}
\end{align*}
for all $i,k\in\N$.

Let $(\Theta_{t})_{t\in\N}:A\to \prod_{\N}B_{\infty}$ be a sequence of c.p.c.\ nuclear maps lifting $\Theta$, and for each pair $i,k\in\N$, let $(e_{i,k}^{(t)})_{t\in\N}$ and $(f_{i,k}^{(t)})_{t\in\N}$ be elements of $\ell^{\infty}(\N,B_\infty)$ representing $e_{i,k}$ and $f_{i,k}$, respectively.\
We may assume $\|e_{i,k}^{(t)}\|\leq\|e_{i,k}\|$ and $\|f_{i,k}^{(t)}\|\leq\|f_{i,k}\|$ for all $t\geq 1$.
It follows that
\begin{align*}
\limsup_{t\to\infty}\left(\|(h_i^0-2^{-k})_+-(e_{i,k}^{(t)})^*\Theta_t(q_i)e_{i,k}^{(t)}\|+\|(\Theta_t(q_i)-2^{-k})_+-(f_{i,k}^{(t)})^*h_i^0f_{i,k}^{(t)}\|\right)=0
\end{align*}
for all $i, k\in\N$.\
Moreover, we may assume that
\begin{align*}
\limsup_{t\to\infty}\|\Theta_t(ab)-\Theta_t(a)\Theta_t(b)\| &=0, \\
\limsup_{t\to\infty}\max_{g\in K}\|\beta_{\infty,g} \circ \Theta_t(a)-\Theta_t \circ \alpha_g(a)\| &=0
\end{align*}
for all $a,b\in A$ and compact subsets $K\subseteq G$.\
Note that the latter condition can be assumed as sequences lift to $\beta_\infty$-continuous sequences, as recalled in Definition \ref{def:sequencealgebra}.\
Fix an increasing sequence of finite subsets $\mathcal{F}_n\subseteq A$ such that $\overline{\bigcup_{n\in\N}\mathcal{F}_n}=A$, and an increasing sequence of compact subsets $K_n\subseteq G$ such that $\bigcup_{n\in\N}K_n=G$.\
For each $t\in\N$, let $(\Theta_{t}^{(j)})_{j\in\N}:A\to \prod_{\N}B$ be a sequence of c.p.c.\ nuclear maps representing $\Theta_{t}$, and for each $i,k,t\in\N$, let $(e_{i,k}^{(t,j)})_{j\in\N}$ and $(f_{i,k}^{(t,j)})_{j\in\N}$ elements of $\ell^{\infty}(\N,B)$ representing $e_{i,k}^{(t)}$ and $f_{i,k}^{(t)}$, respectively.\
We may assume $\|e_{i,k}^{(t,j)}\|\leq\|e_{i,k}\|$ and $\|f_{i,k}^{(t,j)}\|\leq\|f_{i,k}\|$ for all $t,j\geq 1$.
Then, inductively on $n\in\N$, we may find a strictly increasing sequence $(t_n)_n$ in $\N$ such that
\begin{align*}
\max_{1\leq i, k\leq n}\limsup_{j\to\infty}\|(e_{i,k}^{(t_n,j)})^*\Theta_{t_n}^{(j)}(q_i)e_{i,k}^{(t_n,j)} - (h_i^0-2^{-k})_+\| &\leq 2^{-n-1},\\
\max_{1\leq i, k\leq n}\limsup_{j\to\infty}\|(\Theta_{t_n}^{(j)}(q_i)-2^{-k})_+ - (f_{i,k}^{(t_n,j)})^*h_i^0f_{i,k}^{(t_n,j)}\| &\leq 2^{-n-1}, \\
\max_{a,b\in\mathcal{F}_n}\limsup_{j\to\infty}\|\Theta_{t_n}^{(j)}(ab)-\Theta_{t_n}^{(j)}(a)\Theta_{t_n}^{(j)}(b)\| &\leq 2^{-n-1},\\
\max_{a\in\mathcal{F}_n,g\in K_n}\limsup_{j\to\infty} \|\beta_{g} \circ \Theta_{t_n}^{(j)}(a)-\Theta_{t_n}^{(j)} \circ \alpha_g(a)\| &\leq 2^{-n-1}
\end{align*}
for all $n\in\N$.\
Therefore, for each index $t_n$, we may find $j_n\in\N$ (which we can arrange to form an increasing sequence) such that when $j\geq j_n$ one has that
\begin{align*}
\max_{1\leq i, k\leq n}\|(e_{i,k}^{(t_n,j)})^*\Theta_{t_n}^{(j)}(q_i)e_{i,k}^{(t_n,j)} - (h_i^0-2^{-k})_+\| &\leq 2^{-n},\\
\max_{1\leq i, k\leq n}\|(\Theta_{t_n}^{(j)}(q_i)-2^{-k})_+ - (f_{i,k}^{(t_n,j)})^*h_i^0f_{i,k}^{(t_n,j)}\| &\leq 2^{-n}, \\
\max_{a,b\in\mathcal{F}_n}\|\Theta_{t_n}^{(j)}(ab)-\Theta_{t_n}^{(j)}(a)\Theta_{t_n}^{(j)}(b)\| &\leq 2^{-n},\\
\max_{a\in\mathcal{F}_n,g\in K_n}\|\beta_{g} \circ \Theta_{t_n}^{(j)}(a)-\Theta_{t_n}^{(j)} \circ \alpha_g(a)\| &\leq 2^{-n}.
\end{align*}
Now, observe that the diagonal sequence of c.p.c.\ nuclear maps $\{\Theta_{t_n}^{(j_n)}\}_{n\in\N}:A\to\prod_{\N}B$ produces a nuclearly liftable equivariant $\ast$-homomorphism $\psi:A\to B_{\infty,\beta}$.\
Consider the elements $\bar{e}_{i,k}=[(e_{i,k}^{(t_n,j_n)})_n]\in B_\infty$ and $\bar{f}_{i,k}=[(f_{i,k}^{(t_n,j_n)})_n]\in B_\infty$.
From the approximate conditions above we may then conclude
\[
\bar{e}_{i,k}^*\psi(q_i)\bar{e}_{i,k} = (h_i^0-2^{-k})_+ \quad\text{and}\quad \bar{f}_{i,k}^* h_i^0 \bar{f}_{i,k} = (\psi(q_i)-2^{-k})_+.
\]
Since we may let $k\to\infty$, this implies that $h_i^0$ and $\psi(q_i)$ are Cuntz equivalent in $B_\infty$.
In particular, $\psi(q_i)$ belongs to and is a full element in $\I(\iota_{\infty})(\Phi(\overline{Aq_iA}))$, for all $i\in\N$.\
Since we chose $Q$ at the start of the proof to satisfy the conclusion of Lemma \ref{lem:countabledense} for $S$ consisting of all positive contractions in $A$, it follows that $\psi(a)$ is a full element of $\I(\iota_{\infty})(\Phi(\overline{AaA}))$ for all $a\in A_+$.
\end{proof}

\begin{remark}[{see \cite[Remark 4.6]{GS22b}}]\label{rem:trick}
Let $\sigma:\C_0(\R) \to \C_0(\R)$ be the shift automorphism given by $\sigma(f)(t) = f(t+1)$.\
It is well-known that $\C_0(\R)\rtimes_{\sigma}\Z \cong \C(\T)\otimes\K$.\
For any \Cs-algebra $A$, denote by $\tau$ the automorphism of its suspension algebra $SA=\C_0(\R)\otimes A$ given by $\sigma\otimes\id_A$.
We then have the following isomorphism,
\begin{equation*}
SA \rtimes_{\tau} \Z \cong A \otimes (\C_0(\R) \rtimes_\sigma \Z) \cong A \otimes \C(\T) \otimes \K.
\end{equation*}
For an action $\alpha:G\curvearrowright A$, we use the notation $S\alpha := \id_{\C_0(\R)}\otimes\alpha:G\curvearrowright SA$.\
Since $S\alpha:G\curvearrowright SA$ pointwise commutes with the action $\tau:\Z\curvearrowright SA$, the natural isomorphism in the previous paragraph shows that $\alpha\otimes\id:G\curvearrowright A\otimes\C(\T)\otimes\K$ corresponds to the unique action $S\alpha\rtimes\Z:G\curvearrowright SA\rtimes_{\tau}\Z$ that extends $S\alpha$ by acting trivially on the copy of $\Z$.
\end{remark}

\begin{lemma}\label{lem:SA}
Let $\alpha:G \curvearrowright A$ be an action on a separable, exact $\Cs$-algebra, and $\beta:G \curvearrowright B$ a strongly stable, equivariantly  $\O2$-stable, isometrically shift-absorbing action on a separable $\Cs$-algebra.\
Then, for every equivariant $\Cu$-morphism $\Phi:(\I(A),\alpha^{\sharp})\to(\I(B),\beta^{\sharp})$, there exists a nuclearly liftable, equivariant $\ast$-homomorphism
\begin{equation*}
\psi: (SA,S\alpha) \to ( B_{\infty,\beta}, \beta_\infty)
\end{equation*}
such that, for each positive $a\in A$ and positive non-zero $f\in\C_0(\R)$, $\psi(f \otimes a)$ is a full element of $\I(\iota_\infty)(\Phi(\overline{AaA}))$.\footnote{Recall that $\iota_\infty$ denotes the canonical embedding into the sequence algebra, not its continuous part.}
\end{lemma}
\begin{proof}
In the rest of the proof, we use the notation $B^s$ for $B\otimes\K(\Hil_G^\infty)\otimes\O2$, $\beta^s$ for the action $\beta\otimes\Ad(\lambda^\infty)\otimes\id_{\O2}$ on $B^s$, and $\Phi^s(I)$ for the ideal given by $\Phi(I)\otimes\K(\Hil_G^\infty)\otimes\O2$ for all $I\in\I(A)$.\
Moreover, we will identify $\K$ with the \Cs-algebra of compact operators on $\Hil_G^\infty$.

Let $\rho_{\mu}:SA\to A$ be the right slice map associated to a faithful state $\mu$ on $\C_0(\R)$.\
Then, the induced $\Cu$-morphism $\I(\rho_{\mu}):\I(SA)\to\I(A)$ has the property that for any positive non-zero function $f\in\C_0(\R)$ and positive element $a\in A$, $\I(\rho_{\mu})(\overline{SA(f\otimes a)SA})=\overline{AaA}$.\
Denote by $\mathcal{S}$ the set of positive elements of $SA$ of the form $f\otimes a$ for positive non-zero functions $f\in\C_0(\R)$ and positive elements $a\in A$.
In order to prove the lemma, it is enough to show that there exists an equivariant $\ast$-homomorphism
\begin{equation*}
\psi:(SA,S\alpha) \to ( B^s_{\infty,\beta^s}, \beta^s_\infty)
\end{equation*}
that is nuclearly liftable, and with the property that $\psi(f \otimes a)$ is a full element of
\begin{equation*}
\I(\iota_\infty)(\Phi^s(\I(\rho_\mu)(\overline{SA(f\otimes a)SA})))=\I(\iota_\infty)(\Phi^s(\overline{AaA}))
\end{equation*}
for all $(f\otimes a)\in\mathcal{S}$, and then apply Lemma \ref{lem:seqargument}.

From Lemma \ref{lem:existence}, there exists a nuclear, equivariant $\ast$-homomorphism $\kappa: (A,\alpha) \to (\M(B\otimes\K),\beta\otimes\Ad(\lambda^\infty))$ satisfying
\begin{equation*}
\overline{(B\otimes\K) \kappa(I) (B\otimes\K)} = \Phi(I)\otimes\K
\end{equation*}
for all $I\in\I(A)$.\
We use \cite[Lemma 1.4]{Kas88} to infer the existence of an approximate unit $(b_n)_{n\in\N}$ in $B\otimes\K$ such that
\begin{align*}
\lim_{n\to\infty} \max_{g\in K} \|(\beta\otimes\Ad(\lambda^\infty))_g(b_n) - b_n\|=0, \quad \lim_{n\to\infty} \|\left[ b_n, \kappa(A) \right]\|=0
\end{align*}
for every compact set $K\subseteq G$.\
The approximate unit $(b_n)_n$ induces an element $b\in (B\otimes\K)_{\infty,\beta\otimes\Ad(\lambda^\infty)}$ that commutes with $\kappa(A)$ and is fixed under $(\beta\otimes\Ad(\lambda^\infty))_\infty$.\
We may then pick a positive contraction $h \in\O2 \subseteq (\O2)_\infty$ with full spectrum $\left[0,1\right]$, and consider the elementary tensor $b\otimes h$, which is an element in $B^s_{\infty,\beta^s}$ that commutes with $\kappa(A)\otimes\1_{\O2}$.\
Pick a homeomorphism $\theta: (0,1)\to\R$, for instance $\theta(t)=\log(\frac1t-1)$.
We use $b\otimes h$ to define the $\ast$-homomorphism given by
\begin{equation*}
\psi: (SA,S\alpha) \to ( B^s_{\infty,\beta^s},\beta^s_\infty), \quad \psi(f\otimes a)=(f\circ\theta)(b\otimes h) \cdot (\kappa(a)\otimes\1_{\O2})
\end{equation*}
for all $f\in\C_0(\R)$ and $a\in A$, which is equivariant because $(f\circ\theta)(b\otimes h) \in B^s_{\infty,\beta^s}$ is fixed by $\beta^s_\infty$, and $\kappa$ is equivariant.\
Moreover, $\psi$ is nuclear by \cite[Lemma 6.9]{Gab20}.

For any positive element $a\in A$ and positive non-zero function $f\in\C_0(\R)$, we want to show that $\psi(f\otimes a)$ is a full element of $\I(\iota_{\infty})(\Phi^s(\overline{AaA}))$.\
In order to show that $\psi(f\otimes a)$ is contained in $\I(\iota_{\infty})(\Phi^s(\overline{AaA}))$, it is sufficient to prove that, for any given $\e>0$, the element $\psi(f\otimes(a-\e)_{+})$ belongs to $\I(\iota_{\infty})(\Phi^s(\overline{AaA}))$.\
Then, fix some $\e>0$.\
Recall that $\Cu$-morphisms preserve compact containment and that $\overline{A(a-\e)_{+}A} \Subset \overline{AaA}$  from Remark \ref{rem:ideallattice}(i).\
Then, we have that $\Phi^s(\overline{A(a-\e)_{+}A}) \Subset \Phi^s(\overline{AaA})$.\
Subsequently, we use Remark \ref{rem:ideallattice}(ii) to infer that
\begin{align*}
\psi(f\otimes(a-\e)_{+}) &= (f\circ\theta)(b\otimes h)\kappa(a-\e)_{+} \\
&\in \I(\iota_{\infty})\left(\overline{ B^s (\kappa\otimes\1_{\O2})(\overline{A(a-\e)_{+}A}) B^s }\right) \\
&\subseteq \I(\iota_{\infty})\left(\Phi^s(\overline{AaA})\right).
\end{align*}
Finally, we want to show that $\psi(f\otimes a)$ is also a full element.\
Recall that, since $B$ is separable, $\Phi(\overline{AaA})\otimes\K$ contains a full, positive element, which we denote by $e$.\
Moreover, since $\O2$ is simple, $f$ is non-zero and $h$ has full spectrum $\left[0,1\right]$, $e \otimes (f\circ\theta)(h)$ is full in $\Phi^s(\overline{AaA})$, and thus also in $\I(\iota_{\infty})(\Phi^s(\overline{AaA}))$.\
It is therefore sufficient to show that, for any $\e>0$, the element $(e-\e)_{+}\otimes (f\circ\theta)(h)$ belongs to the ideal generated by $\psi(f\otimes a)$.
We thus fix an arbitrary $\e>0$.\
Observe that $b\otimes h$ is in $B^s_{\infty} \cap (B\otimes\K\otimes 1_{\O2})'$, and that $b\otimes 1_{\O2} + (B\otimes\K\otimes 1_{\O2})^{\perp}$ is the unit of $F(B\otimes\K\otimes 1_{\O2},B^s_{\infty})$.\
In particular,
\begin{equation*}
b\otimes h + (B\otimes\K\otimes 1_{\O2})^{\perp} = 1\otimes h + (B\otimes\K\otimes 1_{\O2})^{\perp},
\end{equation*}
which furthermore implies that
\begin{equation*}
(f\circ\theta)(b\otimes h) + (B\otimes\K\otimes 1_{\O2})^{\perp} = 1\otimes (f\circ\theta)(h) + (B\otimes\K\otimes 1_{\O2})^{\perp}.
\end{equation*}
Hence, for any $c\in B\otimes\K$,
\begin{equation}\label{eq:commute}
(f\circ\theta)(b\otimes h)(c\otimes 1_{\O2}) = (c\otimes 1_{\O2})(f\circ\theta)(b\otimes h) = (c\otimes 1_{\O2})(1\otimes (f\circ\theta)(h)).\
\end{equation}
Since $e\in\overline{(B\otimes\K)\kappa(a)(B\otimes\K)}$, by \cite[Lemma 2.5(ii)]{KR00} there exist some $n\in\mathbb{N}$, and $c_1,\dots,c_n \in B\otimes\K$ such that $(e-\e)_{+} = \sum_{i=1}^{n} c_i^*\kappa(a)c_i$, and
\begin{align*}
 (e-\e)_{+}\otimes (f\circ\theta)(h) =  &\sum_{i=1}^{n} (c_i^*\otimes 1_{\O2})(\kappa(a)\otimes 1_{\O2})(c_i\otimes 1_{\O2})(1\otimes (f\circ\theta)(h))\\
 = &\sum_{i=1}^{n} (c_i^*\otimes 1_{\O2})\kappa(a)(f\circ\theta)(b\otimes h)(c_i\otimes 1_{\O2})\\
 = &\sum_{i=1}^{n} (c_i^*\otimes 1_{\O2})\psi(f\otimes a)(c_i\otimes 1_{\O2}),
\end{align*}
where in the second passage we used \eqref{eq:commute}.\
This finishes the proof.
\end{proof}

For the reader's convenience, we recall a result from \cite{GS22b} that will be heavily used later.\
Recall that nuclear liftability in the context of this lemma is equivalent to nuclearity by Remark \ref{rem:nuclift}

\begin{lemma}[see {\cite[Lemma 4.2]{GS22b}}]\label{lem:4.2}
Let $\alpha:G\curvearrowright A$ be an amenable action on a separable, exact \Cs-algebra, and $\beta:G\curvearrowright B$ a strongly stable, isometrically shift-absorbing action on a separable \Cs-algebra.\
Let $\f,\psi:A\to B_{\infty,\beta}$ be nuclearly liftable, equivariant $\ast$-homomorphisms.\
If $\f$ approximately 1-dominates $\psi$ as ordinary $\ast$-homomorphisms, then $(\f,\1)$ approximately 1-dominates $(\psi,\1)$.
\end{lemma}

The following can be considered as an ideal-related version of \cite[Lemma 4.3]{GS22b}.

\begin{lemma}\label{lem:4.3}
Let $\alpha:G\curvearrowright A$ be an amenable action on a separable, exact \Cs-algebra, and $\beta:G\curvearrowright B$ an isometrically shift-absorbing, strongly stable, equivariantly $\O2$-stable action on a separable \Cs-algebra.\
Let $\f,\psi:(A,\alpha)\to(B_{\infty,\beta},\beta_\infty)$ be two nuclearly liftable, equivariant $\ast$-homomorphisms.\
If $\I(\f)=\I(\psi)$ as maps into $\I(B_\infty)$, then $(\f,\1)$ and $(\psi,\1)$ are properly unitarily equivalent.
\end{lemma}
\begin{proof}
It follows from Theorem \ref{thm:weak-cont} that $\f$ and $\psi$ approximately 1-dominate each other.\
Hence, we conclude with Lemma \ref{lem:4.2} that $(\f,\1)$ and $(\psi,\1)$ approximately 1-dominate each other.\
By Lemma \ref{lem:inductivelim}\ref{item:first}, there exists a separable, $\beta_\infty$-invariant \Cs-subalgebra $D_0$ of $B_{\infty,\beta}$ containing the image of $\f$ and $\psi$, and such that $(\f,\1)$ and $(\psi,\1)$ approximately 1-dominate each other when corestricted to $D_0$.\
From Lemma \ref{lem:inductivelim}\ref{item:second}, we have that $\kappa \circ (\f,\1)$ and $\kappa \circ (\psi,\1)$ approximately 1-dominate each other, where $\kappa:(D_0,\beta_\infty)\to(D_0\otimes\K,\beta_\infty\otimes\id_{\K})$ is the equivariant inclusion given by $\kappa(d)=d\otimes e_{1,1}$ for all $d\in D_0$.\
Now, using that $\beta$ is strongly stable, we apply \cite[Lemma 1.21]{GS22b} to infer that there exists an equivariant $\ast$-homomorphism $\chi:(D_0\otimes\K,\beta_\infty\otimes\id_{\K})\to(B_{\infty,\beta},\beta_\infty)$ such that $\chi(d\otimes e_{1,1})=d$ for all $d\in D_0$.\
These maps then fit into the following commutative diagram,
\[\begin{tikzcd}
	{(D_0,\beta_\infty)} && {(B_{\infty,\beta},\beta_\infty)} \\
	& {(D_0 \otimes \K,\beta_\infty\otimes\id_{\K})}
	\arrow[from=1-1, to=1-3]
	\arrow["{\kappa}"'{pos=0.3}, from=1-1, to=2-2]
	\arrow["{\chi}"'{pos=0.7}from=2-2, to=1-3]
\end{tikzcd}\]
Let $D_1$ denote the image of $D_0\otimes\K$ in $B_{\infty,\beta}$ under $\chi$.\
It follows that $\beta_\infty|_{D_1}$ is strongly stable, and $(\f,\1)$, $(\psi,\1)$ approximately 1-dominate each other when corestricted to $D_1=\chi(D_0\otimes\K)$.\
Moreover, as $\beta\cc\beta\otimes\id_{\O2}$, it follows from \cite[Lemma 2.12]{Sza18} that there exists a unital embedding of $\O2$ into $F(D_1,B_{\infty,\beta})^{\tilde{\beta}_\infty}$.\
Consider the canonical commutative diagram of equivariant $\ast$-homomorphisms given by
\[\begin{tikzcd}
	{(D_1,\beta_\infty)} && {(B_{\infty,\beta},\beta_\infty)} \\
	& {(D_1 \otimes_{\max}F(D_1,B_{\infty,\beta}),\beta_\infty\otimes\tilde{\beta}_\infty)}
	\arrow[from=1-1, to=1-3]
	\arrow["{\id_{D_1}\otimes\1}"'{pos=0.3}, from=1-1, to=2-2]
	\arrow["{\rho}"'{pos=0.7}from=2-2, to=1-3]
\end{tikzcd}\]
where $\rho$ is the natural $\ast$-homomorphism defined as $\rho(d\otimes (x+B_{\infty,\beta}\cap D_1^\perp))=dx$, for all $d\in D_1$, $x\in B_{\infty,\beta}\cap D_1'$.\
Denote by $D$ the image of $D_1\otimes\O2$ under $\rho$.\
It follows that there exists a unital copy of $\O2$ in $\M(D)^{\beta_\infty}$ that commutes with the image of $\f$ and $\psi$, $\beta_\infty|_{D}$ is strongly stable, and $(\f,\1)$, $(\psi,\1)$ approximately 1-dominate each other when corestricted to $D$.\
We may now apply Lemma \ref{lem:2.8} and conclude that $(\f,\1)$ and $(\psi,\1)$ are strongly asymptotically unitarily equivalent when corestricted to $D$.\
Since the cocycles are trivial, the norm-continuous path implementing the equivalence is asymptotically $\beta_\infty|_{D}$-invariant.\
By performing a standard diagonal sequence argument inside $B_\infty$, one can show that $\f$ and $\psi$ are properly unitarily equivalent via a $\beta_\infty$-invariant unitary, and therefore $(\f,\1)$ and $(\psi,\1)$ are properly unitarily equivalent.
\end{proof}

\begin{lemma}\label{lem:trick}
Let $\alpha:G \curvearrowright A$ be an amenable action on a separable, exact $\Cs$-algebra, and $\beta:G \curvearrowright B$ a strongly stable, equivariantly $\O2$-stable and isometrically shift-absorbing action on a separable $\Cs$-algebra.\
Then, for every equivariant $\Cu$-morphism $\Phi:(\I(A),\alpha^{\sharp})\to(\I(B),\beta^{\sharp})$ there exists an equivariant, nuclearly liftable $\ast$-homomorphism
\begin{equation*}
\psi: (A,\alpha) \to (B_{\infty,\beta}, \beta_\infty)
\end{equation*}
that satisfies $\I(\psi) = \I(\iota_{\infty}) \circ \Phi$ as a map with range in $B_\infty$.
\end{lemma}
\begin{proof}
Consider the automorphisms $\sigma$ and $\tau=\sigma\otimes\id_A$ of $\C_0(\R)$ and $SA$, respectively, as in Remark \ref{rem:trick}.\
The assumptions of Lemma \ref{lem:SA} are fullfilled, hence there exists a nuclearly liftable, equivariant $\ast$-homomorphism
\begin{align*}
\psi_1: (SA,S\alpha) \to (B_{\infty,\beta},\beta_\infty)
\end{align*}
with the property that for each positive element $a\in A$ and positive non-zero function $f\in\C_0(\R)$,  $\psi_1(f\otimes a)$ is a full elemenet of $\I(\iota_{\infty})(\Phi(\overline{AaA}))$.\
Since also $\psi_1\circ\tau(f\otimes a)=\psi_1(\sigma(f)\otimes a)$ is full in $\I(\iota_{\infty})(\Phi(\overline{AaA}))$, it follows that $\I(\psi_1)$ and $\I(\psi_1\circ\tau)$, where $\psi_1$ is viewed as a map with range in $B_\infty$, agree on ideals generated by elements of the form $f\otimes a$, with $f$ and $a$ as before.\
Furthermore, since $\I(\psi_1)(I)$ is generated by $\psi_1(I_+)$ for any ideal $I\in\I(SA)$ (see \cite[Lemma 2.12(i)]{Gab20}), and $\I(\psi_1)$ preserves suprema, it follows that $\I(SA)$ is completely determined by ideals generated by $f\otimes a$ for positive non-zero $f\in\C_0(\R)$ and $a\in A_+$.\
Hence, we may conclude that $\I(\psi_1)=\I(\psi_1 \circ \tau)$.

We can therefore apply Lemma \ref{lem:4.3} to the equivariant $\ast$-homomorphisms $\psi_1$ and $\psi_1 \circ \tau$, and obtain that $(\psi_1,\1)$ and $(\psi_1 \circ \tau,\1)$ are properly unitarily equivalent via a unitary $u\in\U(\1+B_{\infty,\beta}^{\beta_\infty})$.

From the universal property of crossed products, there exists a $\ast$-homomorphism
\begin{equation*}
\psi_0: SA\rtimes_{\tau}\Z  \to B_{\infty,\beta}, \quad \psi_0|_{SA}=\psi_1,
\end{equation*}
which is also nuclear as a map $SA\rtimes_{\tau}\Z  \to B_{\infty}$ by \cite[Lemma 6.10]{Gab20}.\
Let us show that $\psi_0$ is $(S\alpha\rtimes\Z)$-to-$\beta_\infty$-equivariant\footnote{Recall from Remark \ref{rem:trick} that $S\alpha\rtimes\Z:G\curvearrowright SA\rtimes_{\tau}\Z$ is the action that extends $S\alpha$ by acting trivially on the copy of $\Z$.}.\
Let $v$ be the canonical unitary for the crossed product $SA\rtimes_{\tau}\Z$, then we have
\begin{align*}
\psi_0 \circ (S\alpha\rtimes\Z)_g(x v^n) &= \psi_0(S\alpha_g(x)v^n) \\
		&= \psi_1(S\alpha_g(x))u^n \\
		&= \beta_{\infty,g}(\psi_1(x)u^n) \\
		&= \beta_{\infty,g} \circ \psi_0(xv^n)
\end{align*}
for all $x\in SA$, $n\in\Z$ and $g\in G$.

Recall from Remark \ref{rem:trick} that there exists a natural $\ast$-isomorphism
\begin{equation*}
\theta: (\C_0(\R)\rtimes_{\sigma}\Z)\otimes A  \to  SA\rtimes_{\tau}\Z,
\end{equation*}
and the action $\id\otimes\alpha$ on the left corresponds to $S\alpha\rtimes\Z$, which acts trivially on the copy of $\Z$, on the right.\
For a full projection $p\in \C_0(\R)\rtimes_{\sigma}\Z\cong\C(\T)\otimes\K$, the $\ast$-homomorphism
\begin{equation*}
\psi:  A \to B_{\infty,\beta}, \quad a \mapsto \psi_0(\theta(p \otimes a)),
\end{equation*}
is nuclear as a map with range in $B_\infty$, and equivariant with respect to $\alpha$ on the left and $\beta_\infty$ on the right as $A\hookrightarrow \C(\T)\otimes\K\otimes A$ is $\alpha$-to-$\id\otimes\alpha$ equivariant, and
\begin{align*}
\psi \circ \alpha_g(a) &= \psi_0(\theta(p \otimes \alpha_g(a))) \\
							&= \psi_0(\theta \circ S\alpha_g(p \otimes a)) \\
							&= \psi_0((S\alpha\rtimes\Z)_g \circ \theta(p \otimes a)) \\
							&= \beta_{\infty,g} \circ \psi(a)
\end{align*}
for all $a\in A$, and $g\in G$.

We want to show that $\I(\psi) = \I(\iota_{\infty}) \circ \Phi$.\
Equivalently, we want to prove that for any positive element $a\in A$, its image $\psi(a)$ is contained in $\I(\iota_{\infty})(\Phi(\overline{AaA}))$ and it is full.

Pick a positive element $a\in A$, together with a positive non-zero function $f\in\C_0(\R)$.\
Denote by $w$ the canonical unitary in $\M(\C_0(\R)\rtimes_\sigma\Z)$.\
It follows that for every $n\in\Z$ we have
\begin{equation*}
\psi_0(\theta(fw^n\otimes a)) = \psi_1(f\otimes a)u^n \in \I(\iota_{\infty})(\Phi(\overline{AaA})).
\end{equation*}
Hence, we may conclude that
\begin{equation*}
\psi(a) = \psi_0(\theta(p\otimes a)) \in \I(\iota_{\infty})(\Phi(\overline{AaA})).
\end{equation*}
Since $\psi_1(f\otimes a)$ is full in $\I(\iota_{\infty})(\Phi(\overline{AaA}))$, and $p$ is full in $\C_0(\R)\rtimes_{\sigma}\Z$, we get the following inclusion,
\begin{align*}
\I(\iota_{\infty})(\Phi(\overline{AaA})) &\subseteq \overline{B_{\infty} \psi_1(f\otimes a) B_{\infty}} \\
&= \overline{B_{\infty} \psi_0(\theta(f\otimes a)) B_{\infty}} \\
&\subseteq \overline{B_{\infty} \psi_0(\theta(p\otimes a)) B_{\infty}} \\
&= \overline{B_{\infty} \psi(a) B_{\infty}}.
\end{align*}
Thus, $\psi(a)$ is full in $\I(\iota_{\infty})(\Phi(\overline{AaA}))$.
\end{proof}

\begin{remark}
Let $B$ be a \Cs-algebra.\
A sequence $\eta:\N\to\N$ such that $\lim_{n\to\infty}\eta(n)=\infty$ induces a $\ast$-endomorphism $\eta^*$ on the sequence algebra of $B$ as follows,
\begin{equation*}
\eta^*:B_{\infty}\to B_{\infty}, \quad \left[(x_n)_{n\in\N}\right] \mapsto \left[(x_{\eta(n)})_{n\in\N}\right].
\end{equation*}
Note that, for any action $\beta:G\curvearrowright B$, the $\ast$-endomorphism $\eta^*$ is equivariant with respect to $\beta_\infty$, and therefore restricts to an equivariant $\ast$-endomorphism on $B_{\infty,\beta}$.
\end{remark}

\begin{remark}\label{rem:thm4.10}
Let $\alpha:G\curvearrowright A$ and $\beta:G\curvearrowright B$ be two actions on \Cs-algebras with $A$ separable.\
Let $(\psi,\bbv):(A,\alpha)\to(B_{\infty,\beta},\beta_\infty)$ be a proper cocycle morphism.\
By \cite[Theorem 4.10]{Sza21}, the following are equivalent:
\begin{enumerate}[label=(\roman*),leftmargin=*]
\item For every sequence $\eta:\N\to\N$ such that $\lim_{n\to\infty}\eta(n)=\infty$, $(\psi,\bbv)$ is properly unitarily equivalent to $\eta^* \circ (\psi,\bbv)$;
\item $(\psi,\bbv)$ is properly unitarily equivalent to $\iota_\infty \circ (\f,\bbu)$ for a proper cocycle morphism $(\f,\bbu):(A,\alpha)\to(B,\beta)$.
\end{enumerate}
\end{remark}

We are now able to prove the existence result that will ultimately lead to our classification theorem.

\begin{theorem}[Existence]\label{thm:existence}
Let $\alpha:G\curvearrowright A$ be an amenable action on a separable, exact $\Cs$-algebra, and $\beta:G\curvearrowright B$ an equivariantly $\O2$-stable, isometrically shift-absorbing action on a separable $\Cs$-algebra.\
Then, for every equivariant $\Cu$-morphism $\Phi:(\I(A),\alpha^{\sharp})\to(\I(B),\beta^{\sharp})$, there exists a nuclear proper cocycle morphism $(\f,\bbu):(A,\alpha)\to(B,\beta)$ such that $\I(\f)=\Phi$.
\end{theorem}

\begin{proof}
By Remark \ref{rem:thm5.6} there exists a proper cocycle conjugacy
\begin{equation*}
(\theta,\mathbbm{x}) : (B,\beta) \to (B\otimes\O2, \beta\otimes\id_{\O2})
\end{equation*}
that is strongly asymptotically unitarily equivalent to the first factor embedding $\id_{B}\otimes\1_{\O2}$ with trivial cocycle.\
Then, for any embedding $\iota:\K\hookrightarrow\O2$, there exists a proper cocycle embedding $(B\otimes\K,\beta\otimes\id_{\K})\xhookrightarrow{(\theta,\mathbbm{x})^{-1} \circ (\id\otimes\iota)}(B,\beta)$ inducing the map $I\otimes\K \mapsto I$ for all $I\in\I(B)$ on the ideal lattice.\
For this reason, it is sufficient to prove the statement with the additional assumption that $\beta$ is strongly stable.\
In fact, this would allow us to find a proper cocycle morphism $(A,\alpha)\to(B\otimes\K,\beta\otimes\id_{\K})$ inducing $J\mapsto \Phi(J)\otimes\K$ for all $J\in\I(A)$ on the ideal lattice, which proves the general statement when composed with $\zeta \circ (\id\otimes\iota)$.
Therefore, let us assume that $\beta$ is strongly stable

Since the hypotheses of Lemma \ref{lem:trick} are satisfied, there exists an equivariant, nuclearly liftable $\ast$-homomorphism
\begin{equation*}
\psi:(A,\alpha) \to (B_{\infty,\beta},\beta_\infty),
\end{equation*}
that satisfies $\I(\psi)=\I(\iota_{\infty}) \circ \Phi$ as a map with range in $B_\infty$.

To obtain a proper cocycle morphism with range in $B$, we want to appeal to Remark \ref{rem:thm4.10}.\
Fix a sequence $\eta:\N\to\N$ such that $\lim_{n\to\infty}\eta(n)=\infty$.\
One has that that
\begin{align*}
\I(\eta^* \circ \psi)&=\I(\eta^*) \circ \I(\psi)\\
&=\I(\eta^*) \circ \I(\iota_{\infty}) \circ \Phi\\
&=\I(\iota_{\infty}) \circ \Phi\\
&=\I(\psi).
\end{align*}
By Lemma \ref{lem:4.3}, we get that $(\eta^* \circ \psi,\1)$ and $(\psi,\1)$ are properly unitarily equivalent.\
It follows from Remark \ref{rem:thm4.10} that there exists a proper cocycle morphism $(\f,\bbu):(A,\alpha)\to (B,\beta)$ such that $\iota_{\infty} \circ (\f,\bbu)$ is properly unitarily equivalent to $(\psi,\1)$.\
When considered as maps with range $B_\infty$, we have that
\begin{equation*}
\I(\iota_{\infty}) \circ \I(\f)=\I(\iota_{\infty} \circ \f)=\I(\psi)=\I(\iota_{\infty}) \circ \Phi,
\end{equation*}
and since $\I(\iota_{\infty})$ is injective, it follows that $\I(\f)=\Phi$.\
In particular, $\iota_{\infty} \circ \f$ is nuclearly liftable, and thus $\f$ is nuclear.
\end{proof}

\begin{remark}\label{rem:existence}
If in Theorem \ref{thm:existence} one assumes $G$ to be exact and removes the amenability assumption on $\alpha$, then the same conclusion holds true.\
Let us show why in the following paragraph.

Since $G$ is exact, we conclude with \cite[Theorem 6.6]{OS21} that there exists an amenable action $\delta:G\curvearrowright \O2$.\
Then from Theorem \ref{thm:existence} we get a proper cocycle morphism $(\f_0,\bbu_0):(A\otimes \O2,\alpha\otimes\delta) \to (B,\beta)$ such that $\I(\f_0)(I\otimes \O2)=\Phi(I)$ for all $I\in\I(A)$.\
Hence, one may define $(\f,\bbu):(A,\alpha)\to(B,\beta)$ to be the composition of $(\f_0,\bbu_0)$ with $\id_A\otimes\1_{\O2}$.
The proper cocycle morphism $(\f,\bbu)$ clearly satisfies $\I(\f)=\Phi$.
\end{remark}

\section{The classification theorem}

The uniqueness and existence results obtained in the previous sections yield a one-to-one correspondence between equivariant $\Cu$-morphisms and proper cocycle morphisms identified up to strong asymptotic unitary equivalence.

\begin{corollary}\label{cor:bijection}
Let $\alpha:G\curvearrowright A$ be an amenable action on a separable, exact $\Cs$-algebra, and $\beta:G\curvearrowright B$ an amenable, equivariantly $\O2$-stable, isometrically shift-absorbing, strongly stable action on a separable $\Cs$-algebra.\
Then, there exists a canonical one-to-one correspondence
\begin{equation*}
\frac{\begin{Bmatrix}\text{\rm{nuclear proper cocycle morphisms }}\\ (A,\alpha)\to(B,\beta)\end{Bmatrix}}{\text{\rm{strong asymptotic unitary equivalence}}}  \to \begin{Bmatrix}\text{\rm{equivariant $\Cu$-morphisms }}\\ (\I(A),\alpha^{\sharp})\to(\I(B),\beta^{\sharp}) \end{Bmatrix}
\end{equation*}
\end{corollary}
\begin{proof}
Let $\Phi:(\I(A),\alpha^{\sharp})\to(\I(B),\beta^{\sharp})$ be an equivariant $\Cu$-morphism.\
By Theorem \ref{thm:existence}, there exists a proper cocycle morphism $(\f,\bbu):(A,\alpha)\to(B,\beta)$ such that $\I(\f)=\Phi$.\
Suppose that there exists another proper cocycle morphism $(\psi,\bbv):(A,\alpha)\to(B,\beta)$ such that $\I(\psi)=\Phi$.\
Then by Theorem \ref{thm:uniqueness} we have that $(\f,\bbu)$ and $(\psi,\bbv)$ are strongly asymptotically unitarily equivalent.\
This finishes the proof.
\end{proof}

\begin{corollary}\label{cor:bijection2}
Assume $G$ to be exact.\
Let $\alpha:G\curvearrowright A$ be an action on a separable, exact, unital $\Cs$-algebra, and $\beta:G\curvearrowright B$ an amenable, equivariantly $\O2$-stable, isometrically shift-absorbing action on a separable, unital $\Cs$-algebra.\
Then, there exists a canonical one-to-one correspondence
\begin{equation*}
\frac{\begin{Bmatrix}\text{\rm{unital nuclear proper cocycle morphisms }}\\ (A,\alpha)\to(B,\beta)\end{Bmatrix}}{\text{\rm{strong asymptotic unitary equivalence}}} \to \begin{Bmatrix}\text{\rm{equivariant $\Cu$-morphisms }}\\ \Phi:(\I(A),\alpha^{\sharp})\to(\I(B),\beta^{\sharp}) \\ \text{\rm{such that }} \Phi(A)=B \end{Bmatrix}
\end{equation*}
\end{corollary}
\begin{proof}
From Theorem \ref{thm:existence} combined with Remark \ref{rem:existence} there exists a proper cocycle morphism $(\psi,\bbv):(A,\alpha)\to(B,\beta)$ inducing $\Phi$.\
Since $\Phi(A)=B$ and $B\cong B\otimes\O2$, we have that $\psi(1_A)$ is a properly infinite, full projection in $B$.\
Therefore, there exists an isometry $v\in B$ such that $vv^*=\psi(1_A)$ (cf.\ \cite[Proposition 4.1.4]{Ror02} and \cite[Theorems 1.4+1.9+2.3]{Cun81}).\
We then define a unital $\ast$-homomorphism
\begin{equation*}
\f:A\to B, \quad \f(a)=v^*\psi(a)v,
\end{equation*}
and a norm-continuous map
\begin{equation*}
\bbu:G\to\U(\1+B),\quad \bbu_g=v^*\bbv_g\beta_g(v).
\end{equation*}
Let us briefly show that $\bbu:G\to\U(\1+B)$ is a $\beta$-cocycle,
\begin{align*}
\bbu_{gh}=v^*\bbv_{gh}\beta_{gh}(v)&=v^*\bbv_g\beta_g(\bbv_h \beta_h(v)) \\
&=v^*\bbv_g\beta_g(vv^*\bbv_h\beta_h(v))=\bbu_g\beta_g(\bbu_h)
\end{align*}
for all $g,h\in G$, where we used that $\bbv_h\beta_h(vv^*)=vv^*\bbv_h$, which follows from $\Ad(\bbv_h) \circ \beta_h \circ \psi(1_A)=\psi(1_A)$.
It follows that $(\f,\bbu):(A,\alpha)\to(B,\beta)$ is a unital proper cocycle morphism inducing $\Phi$.\
By Corollary \ref{cor:uniqueness} we have that $(\f,\bbu)$ is the unique unital proper cocycle morphism inducing $\Phi$ up to strong asymptotic unitary equivalence.
\end{proof}

\begin{definition}[{see \cite[Definition 5.9]{GS22b}+\cite[Definition 2.4(iii)]{Sza17}}]
Let $\beta:G\curvearrowright B$ be an action on a \Cs-algebra.\
A norm-continuous $\beta$-cocycle $\bbu:G\to\U(\1+B)$ is said to be an \textit{asymptotic coboundary} if there exists a norm-continuous path of unitaries $v:[0,\infty)\to\U(\1+B)$ such that
\begin{equation*}
  \lim_{t\to\infty}\max_{g\in K} \|\bbu_g - v_t\beta_g(v_t)^*\|=0
\end{equation*}
for every compact set $K\subseteq G$.

If $\alpha:G\curvearrowright A$ is an action on a \Cs-algebra, then a proper cocycle conjugacy $(\f,\bbu):(A,\alpha)\to(B,\beta)$ is  a \textit{very strong cocycle conjugacy} if $\bbu$ is an asymptotic coboundary.
\end{definition}

\begin{remark}\label{rem:coboundary}
Let $\beta:G\curvearrowright B$ be an amenable, isometrically shift-absorbing, equivariantly $\O2$-stable action on a separable \Cs-algebra, and $\bbu:G\to\U(\1+B)$ a norm-continuous $\beta$-cocycle.\
Moreover, consider the following assumptions:
\begin{enumerate}[label=\textup{(\roman*)},leftmargin=*]
\item $\beta$ is strongly stable; \label{assumption1}
\item $B$ is unital, and $G$ exact. \label{assumption2}
\end{enumerate}
Assume \ref{assumption1} for the rest of this paragraph.\
Note that, if $D$ denotes the zero algebra, every proper cocycle morphism $(D,\id_D)\to(B,\beta)$ induces the same $\Cu$-morphism between ideal lattices.\
Hence, Theorem \ref{thm:uniqueness} implies that that $(0,\bbu)$ is strongly asymptotically unitarily equivalent to $(0,\1)$, where $0$ denotes the zero map.\
In other words, there exists a norm-continuous path of unitaries $v:[0,\infty)\to\U(\1+B)$ such that $v_0=\1$, and $\max_{g\in K}\|\bbu_g-v_t\beta_g(v_t)^*\|\xrightarrow{t\to\infty}0$ for all compact sets $K\subseteq G$.

Assume now \ref{assumption2}.\
Observe that any unital proper cocycle morphism $(\mathbb{C},\id_{\mathbb{C}})\to(B,\beta)$ induces the same $\Cu$-morphism between ideal lattices.\
Thus, by Corollary \ref{cor:uniqueness}, one obtains that $(\iota,\bbu)$ is strongly asymptotically unitarily equivalent to $(\iota,\1)$, where $\iota:\mathbb{C}\hookrightarrow B$ is the canonical unital map.\

Consequently, if $\alpha:G\curvearrowright A$ is an action on a \Cs-algebra, and $(\f,\bbu):(A,\alpha)\to(B,\beta)$ a proper cocycle conjugacy, the previous paragraph implies that whenever \ref{assumption1} or \ref{assumption2} holds true, $\bbu$ is an asymptotic coboundary and thus $(\f,\bbu)$ is a very strong cocycle conjugacy.
\end{remark}

We record here a dynamical version of Elliott's intertwining that will be used to prove our classification theorem.

\begin{theorem}[{see \cite[Corollary 4.6]{Sza21}}]\label{thm:Elliott}
Let $\alpha:G\curvearrowright A$ and $\beta:G\curvearrowright B$ be actions on separable \Cs-algebras.\
Let
\begin{equation*}
(\f,\bbu):(A,\alpha)\to(B,\beta), \quad (\psi,\bbv):(B,\beta)\to(A,\alpha)
\end{equation*}
be proper cocycle morphisms such that $(\psi,\bbv) \circ (\f,\bbu)$ and $(\f,\bbu) \circ (\psi,\bbv)$ are strongly asymptotically unitarily equivalent to $\id_A$ and $\id_B$, respectively.\
Then, $(\f,\bbu)$ is strongly asymptotically equivalent to a proper cocycle conjugacy.
\end{theorem}

The following theorem represents the main application of the previous sections, and should be considered as a generalization of Gabe and Kirchberg's $\O2$-stable classification (i.e., \cite{Kir00} and \cite[Theorem 6.13]{Gab20}) to the framework of $\Cs$-dynamical systems.

\begin{theorem}[Classification] \label{thm:classification}
Let $\alpha:G\curvearrowright A$ and $\beta:G\curvearrowright B$ be amenable, equivariantly $\O2$-stable, isometrically shift-absorbing actions on separable, nuclear $\Cs$-algebras.\
Then, the following statements hold true.
\begin{enumerate}[label=\textup{(\roman*)},leftmargin=*]
\item If both $A$ and $B$ are stable, then for every conjugacy $f:(\Prim(A),\alpha^{\sharp})\to(\Prim(B),\beta^{\sharp})$ there exists a cocycle conjugacy $(\f,\bbu):(A,\alpha)\to(B,\beta)$ such that $\f(\p)=f(\p)$ for all $\p\in\Prim(A)$. \label{class1}
\item If both $\alpha$ and $\beta$ are strongly stable, then for every conjugacy $f:(\Prim(A),\alpha^{\sharp})\to(\Prim(B),\beta^{\sharp})$ there exists a very strong cocycle conjugacy $(\f,\bbu):(A,\alpha)\to(B,\beta)$ such that $\f(\p)=f(\p)$ for all $\p\in\Prim(A)$.\label{class2}
\item If $G$ is exact, and both $A$ and $B$ are unital, then for every conjugacy $f:(\Prim(A),\alpha^{\sharp})\to(\Prim(B),\beta^{\sharp})$ there exists a very strong cocycle conjugacy $(\f,\bbu):(A,\alpha)\to(B,\beta)$ such that $\f(\p)=f(\p)$ for all $\p\in\Prim(A)$.\label{class3}
\end{enumerate}
\end{theorem}
\begin{proof}
We first observe that \ref{class1} follows from \ref{class2}.
From Remark \ref{rem:strstable} we know that when both $A$ and $B$ are stable, there exist cocycle conjugacies
\begin{align*}
(\theta_\alpha,\mathbbm{x}_\alpha):(A,\alpha) \to (A\otimes\K,\alpha\otimes\id_{\K}),\quad
(\theta_\beta,\mathbbm{x}_\beta):(B,\beta) \to (B\otimes\K,\beta\otimes\id_{\K}).
\end{align*}
Hence, for every conjugacy $f:(\Prim(A),\alpha^{\sharp})\to(\Prim(B),\beta^{\sharp})$, we may apply (ii) to the conjugacy given by
\begin{align*}
 (\Prim(A\otimes\K),\alpha\otimes\id_{\K}) \to (\Prim(B\otimes\K),\beta\otimes\id_{\K}),\quad
 \p\otimes\K \mapsto (\theta_\beta \circ f \circ \theta_\alpha^{-1}(\p))\otimes\K
\end{align*}
and get a very strong cocycle conjugacy $(\f_0,\bbu_0):(A\otimes\K,\alpha\otimes\id_{\K})\to(B\otimes\K,\beta\otimes\id_{\K})$ such that $\f_0(\p\otimes\K)=(\theta_\beta \circ f \circ \theta_\alpha^{-1}(\p))\otimes\K$ for all $\p\in\Prim(A)$.\
Then, the cocycle conjugacy given as
\begin{equation*}
(\f,\bbu)=(\theta_\beta,\mathbbm{x}_\beta)^{-1} \circ (\f_0,\bbu_0) \circ (\theta_\alpha,\mathbbm{x}_\alpha):(A,\alpha)\to(B,\beta)
\end{equation*}
satisfies $\f(\p)=f(\p)$ for all $\p\in\Prim(A)$.

We now prove the non-trivial implication in \ref{class2}.\
Assume that both $\alpha$ and $\beta$ are strongly stable, and that $f:(\Prim(A),\alpha^{\sharp})\to(\Prim(B),\beta^{\sharp})$ is a conjugacy.\
By Remark \ref{rem:invariants} there exists an equivariant order isomorphism $\Phi:(\I(A),\alpha^{\sharp})\to(\I(B),\beta^{\sharp})$ such that $\Phi(\p)=f(\p)$ for all $\p\in\Prim(A)$.\
From Corollary \ref{cor:bijection} there exist proper cocycle morphisms $(\f_0,\bbu_0):(A,\alpha)\to(B,\beta)$ and $(\psi_0,\bbv_0):(B,\beta)\to(A,\alpha)$ such that $\I(\f_0)=\Phi$ and $\I(\psi_0)=\Phi^{-1}$.\
Since we have that
\begin{equation*}
\I(\psi_0)\circ\I(\f_0)=\I(\id_A), \quad \I(\f_0)\circ\I(\psi_0)=\I(\id_B),
\end{equation*}
Corollary \ref{cor:bijection} guarantees that $(\psi_0,\bbv_0)\circ(\f_0,\bbu_0)$ is strongly asymptotically unitarily equivalent to $\id_A$, and $(\f_0,\bbu_0)\circ(\psi_0,\bbv_0)$ is strongly asymptotically unitarily equivalent to $\id_B$.\
It is therefore possible to apply Theorem \ref{thm:Elliott} to get a proper cocycle conjugacy $(\f,\bbu):(A,\alpha)\to(B,\beta)$ such that $(\f,\bbu)$ is strongly asymptotically unitarily equivalent to $(\f_0,\bbu_0)$, and therefore $\I(\f)=\Phi$.\
By Remark \ref{rem:coboundary}, $(\f,\bbu)$ is a very strong cocycle conjugacy with the property that $\f(\p)=\Phi(\p)=f(\p)$ for all $\p\in\Prim(A)$.

The non-trivial implication in \ref{class3} can be proved analogously as \ref{class2} by virtue of Corollary \ref{cor:bijection2}, which has to be used in place of Corollary \ref{cor:bijection}.
\end{proof}

\begin{remark}[{see \cite[Proposition 6.3]{GS22b}}]\label{rem:compact}
When $G$ is a second-countable, compact group, the following is true.\
If $\beta:G\curvearrowright B$ is an isometrically shift-absorbing action on a separable, stable \Cs-algebra, then it is strongly stable.
\end{remark}

In light of the above remark, we obtain a refined version of Theorem \ref{thm:classification} in case $G$ is compact.

\begin{corollary} \label{cor:compact-classification}
Let $G$ be a second-countable, compact group.\
Let $\alpha:G\curvearrowright A$ and $\beta:G\curvearrowright B$ be isometrically shift-absorbing and equivariantly $\O2$-stable actions on separable, nuclear \Cs-algebras.\
Then, the following statements hold true.
\begin{enumerate}[label=\textup{(\roman*)},leftmargin=*]
\item If both $A$ and $B$ are stable, then for every conjugacy $f:(\Prim(A),\alpha^{\sharp})\to(\Prim(B),\beta^{\sharp})$ there exists a conjugacy $\f:(A,\alpha)\to(B,\beta)$ such that $\f(\p)=f(\p)$ for all $\p\in\Prim(A)$.
\item If both $A$ and $B$ are unital, then for every conjugacy $f:(\Prim(A),\alpha^{\sharp})\to(\Prim(B),\beta^{\sharp})$ there exists a conjugacy $\f:(A,\alpha)\to(B,\beta)$ such that $\f(\p)=f(\p)$ for all $\p\in\Prim(A)$.
\end{enumerate}
\end{corollary}
\begin{proof}
Observe that, by Remark \ref{rem:compact}, if $A$ and $B$ are stable, $\alpha$ and $\beta$ are strongly stable.\
Now, in both cases (i) and (ii), a conjugacy $f:(\Prim(A),\alpha^{\sharp})\to(\Prim(B),\beta^{\sharp})$ lifts to a very strong cocycle conjugacy $(\f,\bbu):(A,\alpha)\to(B,\beta)$ by Theorem \ref{thm:classification}.\
By \cite[Corollary 5.11]{GS22b}, $(\f,\bbu)$ is properly unitarily equivalent to a conjugacy, which concludes the proof.
\end{proof}

\begin{remark}\label{rem:modelaction}
When $G$ is exact, it follows from Theorem \ref{thm:classification} (or from the main result of \cite{GS22b}) that there exists a unique amenable, isometrically shift-absorbing and equivariantly $\O2$-stable action $\delta:G\curvearrowright\O2$ up to very strong cocycle conjugacy.\
An example of such an action is given by
\begin{equation*}
\alpha\otimes\id_{\O2}\otimes\gamma^{\otimes\infty}:G\curvearrowright\O2\otimes\O2\otimes(\Oinf)^{\otimes\infty}\cong\O2,
\end{equation*}
where $\alpha:G\curvearrowright\O2$ is an amenable action on $\O2$, which exists by \cite[Theorem 6.6]{OS21}.\
Note that $\delta$ is automatically strongly self-absorbing.
\end{remark}

\begin{theorem} \label{thm:exact-char}
Suppose $G$ is exact.
Let $\beta:G\curvearrowright B$ be an action on a separable, nuclear \Cs-algebra.\
Let $\delta:G\curvearrowright\O2$ be an amenable, isometrically shift-absorbing, equivariantly $\O2$-stable action.\
Then, $\beta$ is amenable, isometrically shift-absorbing and equivariantly $\O2$-stable if and only if $\beta$ is cocycle conjugate to $\beta\otimes\delta$.
\end{theorem}
\begin{proof}
Since $\delta$ is strongly self-absorbing, the property of absorbing $\delta$ is invariant under stable cocycle conjugacy; see \cite[Theorem 4.30]{BarlakSzabo16}.
The same is the case for equivariant $\O2$-stability.
Since furthermore amenability and isometric shift-absorption are properties of the induced action on $F_{\infty,\beta}(B)$, which is also an invariant under stable cocycle conjugacy, we may assume without any loss of generality that $\beta$ is strongly stable.

Assume that $\beta$ is cocycle conjugate to $\beta\otimes\delta$.\
Since amenability, isometric shift-absorption and equivariant $\O2$-stability are preserved under cocycle conjugacy and tensor products, it follows that $\beta$ inherits these properties from $\beta\otimes\delta$.

Conversely, assume $\beta$ is amenable, isometrically shift-absorbing and equivariantly $\O2$-stable.\
Consider the equivariant first factor embedding $\id_B\otimes\1_{\O2}:(B,\beta)\to(B\otimes\O2,\beta\otimes\delta)$.\
Since $\I(\id_B\otimes\1_{\O2}):(\I(B),\beta^{\sharp})\to(\I(B\otimes\O2,(\beta\otimes\delta)^{\sharp}))$ is an equivariant order isomorphism, Theorem \ref{thm:classification} implies that there exists a cocycle conjugacy $(\theta,\mathbbm{x}):(B,\beta)\to(B\otimes\O2,\beta\otimes\delta)$.
\end{proof}

\begin{remark}
To conclude, let us discuss a few special cases of our main results that can be retrieved with known techniques.
\begin{enumerate}[leftmargin=*,label={$\bullet$},wide]
\item Suppose $G=\mathbb R$.
It is observed in \cite[Corollary 6.15]{GS22b} that when $G=\mathbb{R}$, an action on a separable $\mathcal{O}_\infty$-stable \Cs-algebra is isometrically shift-absorbing if and only if it has the Rokhlin property \cite{Kishimoto96R}.\
(Note that by the main result of \cite{Szabo19rd}, every Rokhlin flow on an $\O2$-stable \Cs-algebra is furthermore equivariantly $\O2$-stable.)
Then our main result follows from \cite[Theorem B+Remark 5.16]{Sza21b}.
\item When $G=\Z$, the unique outer action $\delta:\Z\curvearrowright\O2$ is well-known to have the Rokhlin property; see \cite[Theorem 1]{Nak00}.\
Since this action is strongly self-absorbing, the main result of \cite{Szabo19rd} implies that every $\Z$-action with the Rokhlin property on any $\O2$-stable \Cs-algebra absorbs $\delta$.
In light of Theorem \ref{thm:exact-char}, our main theorem \ref{thm:classification} is about automorphisms with the Rokhlin property on nuclear $\O2$-stable \Cs-algebras.\
This could have been proved directly using the Evans--Kishimoto intertwining method \cite{EK97} together with the Gabe--Kirchberg $\O2$-stable classification theorem \cite{Kir00,Kir,Gab20}.\footnote{This has not been explicitly carried out in the literature, but a sketch of proof can be found in the second author's lecture notes for a mini-course delivered at the NCGOA 2018, available at \url{https://gaborszabo.nfshost.com/Slides/2018-05-NCGOA-lectures-v2.pdf}.}
\item Suppose $G$ is compact.\
Then there exists an action $G\curvearrowright\O2$ with the Rokhlin property (an example being \cite[Example 2.9]{Gar19}).
Since the model action $\delta:G\curvearrowright\O2$ from Remark \ref{rem:modelaction} necessarily absorbs every $G$-action on $\O2$, it has the Rokhlin property as well.\
On the other hand, if $\beta:G\curvearrowright B$ is an action with the Rokhlin property on an $\O2$-stable \Cs-algebra, then $\beta$ is cocycle conjugate to $\beta\otimes\delta$ by \cite[Theorem 4.50]{GL18}.
Let $\alpha: G\curvearrowright A$ and $\beta: G\curvearrowright B$ be actions with the Rokhlin property on separable, nuclear, $\O2$-stable and stable (or unital) \Cs-algebras.
If $\alpha^{\sharp}:G\curvearrowright\Prim(A)$, and $\beta^{\sharp}:G\curvearrowright\Prim(B)$ are conjugate, one can use Gabe--Kirchberg's $\O2$-stable classification theorem \cite{Kir00,Kir,Gab20} to obtain some isomorphism $\f: A\to B$ (lifting such a conjugacy) such that for the action $\beta'_g:=\f\circ\alpha_g\circ\f^{-1}$, the $*$-homomorphism $\beta'_{\mathrm{co}}$ is approximately unitarily equivalent to $\beta_{\mathrm{co}}$.
Here $\beta_{\mathrm{co}}: B\to\C(G,B)$ is the $*$-homomorphism induced by the orbit maps of $\beta$.
It follows that $\beta'$ is conjugate to $\beta$ via an approximately inner automorphism of $B$ (see \cite[Theorem 3.5]{Izu04a}, \cite[Theorem 3.5]{Naw16} and \cite[Theorem 3.4]{GS16} for finite group actions, and \cite[Theorem 5.10]{BSV17} for all compact groups).\
This in turn implies that there is a conjugacy between $\alpha$ and $\beta$ that lifts any given conjugacy between $\alpha^\sharp$ and $\beta^\sharp$, recovering Corollary \ref{cor:compact-classification}.
\end{enumerate}
\end{remark}

%%%%%

\textbf{Acknowledgements.}
The first named author was supported by PhD-grant 1131623N funded by the Research Foundation Flanders (FWO).
The second named author was supported by the research project C14/19/088 funded by the research council of KU Leuven.
Both authors were supported by the project G085020N funded by the Research Foundation Flanders (FWO).

The first author would like to thank Joan Bosa, James Gabe, Francesc Perera, Hannes Thiel, and Stuart White for helpful discussions on the topic of this article during the research workshop \textit{C*-algebras:\ Structure and Dynamics}, held in Sde Boker in 2022.\
He is especially grateful to James Gabe for giving useful insight about classification of non-simple, purely infinite \Cs-algebras during a visit of the first author to the University of Southern Denmark.
Some of the work on this article has been carried out during the conference \textit{C*-Algebras:\ Tensor Products, Approximation \& Classification} in honour of Eberhard Kirchberg, and both authors would like to extend their gratitude to the Mathematics department at the University of Münster for their kind hospitality.

Both authors would like to thank the anonymous referee for their detailed remarks that led to numerous little improvements and corrections in this article.

%%%


\begin{thebibliography}{10}
\providecommand{\url}[1]{\texttt{#1}}
\providecommand{\urlprefix}{URL }
\expandafter\ifx\csname urlstyle\endcsname\relax
  \providecommand{\doi}[1]{doi:\discretionary{}{}{}#1}\else
  \providecommand{\doi}{doi:\discretionary{}{}{}\begingroup
  \urlstyle{rm}\Url}\fi

\bibitem{APT73}
C.~A. Akemann, G.~K. Pedersen and J.~Tomiyama.
\newblock Multipliers of \Cs-algebras.
\newblock J. Funct. Anal. 13 (1973), no.~3, pp. 277--301.

\bibitem{APT18}
R.~Antoine, F.~Perera and H.~Thiel.
\newblock Tensor products and regularity properties of Cuntz semigroups.
\newblock Mem. Amer. Math. Soc. 251 (2018), no. 1199, pp. 1--206.

\bibitem{APT11}
P.~Ara, F.~Perera and A.~S. Toms.
\newblock $K$-theory for operator algebras. Classification of $\Cs$-algebras.
\newblock Contemp. Math. 534 (2011), pp. 1--71.

\bibitem{BarlakSzabo16}
S.~Barlak and G.~Szab{\'o}.
\newblock Sequentially split $*$-homomorphisms between \Cs-algebras.
\newblock Int. J. Math. 42 (2016).
\newblock 48 pp.

\bibitem{BSV17}
S.~Barlak, G.~Szab{\'o} and C.~Voigt.
\newblock The spatial {R}okhlin property for actions of compact quantum groups.
\newblock J. Funct. Anal. 272 (2017), no.~6, pp. 2308--2360.

\bibitem{BK04}
{\'E}.~Blanchard and E.~Kirchberg.
\newblock Non-simple purely infinite $\Cs$-algebras: the Hausdorff case.
\newblock J. Funct. Anal. 207 (2004), no.~2, pp. 461--513.

\bibitem{BGSW22}
J.~Bosa, J.~Gabe, A.~Sims and S.~White.
\newblock The nuclear dimension of $\Oinf$-stable \Cs-algebras.
\newblock Adv. Math. 401 (2022).
\newblock Article 108250, 51 pp.

\bibitem{Bro00}
L.~G. Brown.
\newblock Continuity of actions of groups and semigroups on Banach spaces.
\newblock J. London Math. Soc. 62 (2000), pp. 107--116.

\bibitem{BO08}
N.~P. Brown and N.~Ozawa.
\newblock {$\Cs$}-algebras and finite-dimensional approximations,
  \emph{Graduate Studies in Mathematics}, vol.~88.
\newblock Amer. Math. Soc. (2008).

\bibitem{BEW20}
A.~Buss, S.~Echterhoff and R.~Willett.
\newblock Amenability and weak containment for actions of locally compact
  groups on \Cs-algebras.
\newblock To appear in Mem. Amer. Math. Soc.  (2020).
\newblock \urlprefix\url{https://arxiv.org/abs/2003.03469}.

\bibitem{BEW20a}
A.~Buss, S.~Echterhoff and R.~Willett.
\newblock Injectivity, crossed products, and amenable group actions.
\newblock Contemp. Math. 749 (2020), pp. 105--137.

\bibitem{CE76}
M.-D. Choi and E.~G. Effros.
\newblock The completely positive lifting problem for \Cs-algebras.
\newblock Ann. Math.  (1976), pp. 585--609.

\bibitem{Con75}
A.~Connes.
\newblock {Outer conjugacy classes of automorphisms of factors}.
\newblock Ann. Sci. Éc. Norm. Supér. 8 (1975), no. no. 3, pp. 383--419.

\bibitem{Con76}
A.~Connes.
\newblock {Classification of injective factors Cases \normalfont{II}$_1$,
  {II}$_\infty$, {III}$_\lambda$, $\lambda\neq 1$}.
\newblock Ann. Math. 104 (1976), pp. 73--115.

\bibitem{Con77}
A.~Connes.
\newblock {Periodic automorphisms of the hyperfinite factor of type II$_1$}.
\newblock Acta Sci. Math. 39 (1977), no.~1, pp. 39--66.

\bibitem{CEI08}
K.~T. Coward, G.~A. Elliott and C.~Ivanescu.
\newblock The Cuntz semigroup as an invariant for $\Cs$-algebras.
\newblock J. reine angew. Math. 623 (2008), pp. 161--193.

\bibitem{Cun77}
J.~Cuntz.
\newblock Simple \Cs-algebra generated by isometries.
\newblock Comm. Math. Phys. 57 (1977), no.~2, pp. 173--185.

\bibitem{Cun78}
J.~Cuntz.
\newblock Dimension functions on simple \Cs-algebras.
\newblock Math. Ann. 233 (1978), pp. 145--153.

\bibitem{Cun81}
J.~Cuntz.
\newblock $K$-theory for certain \Cs-algebras.
\newblock Ann. Math.  (1981), pp. 181--197.

\bibitem{Dad97}
M.~Dadarlat.
\newblock Quasidiagonal morphisms and homotopy.
\newblock J. Funct. Anal. 151 (1997), no.~1, pp. 213--233.

\bibitem{Eva80}
D.~E. Evans.
\newblock On $\mathcal{O}_n$.
\newblock Publ. RIMS, Kyoto Univ. 16 (1980), pp. 915--927.

\bibitem{EK97}
D.~E. Evans and A.~Kishimoto.
\newblock Trace scaling automorphisms of certain stable AF algebras.
\newblock Hokkaido Math. J. 26 (1997), pp. 211--224.

\bibitem{Gab20}
J.~Gabe.
\newblock A new proof of Kirchberg's $\mathcal{O}_2$-stable classification.
\newblock J. reine angew. Math. 761 (2020), pp. 247--289.

\bibitem{Gab22}
J.~Gabe.
\newblock Lifting theorems for completely positive maps.
\newblock J. Noncommut. Geom. 16 (2022), no.~2, pp. 391--421.

\bibitem{Gab21}
J.~Gabe.
\newblock Classification of $\Oinf$-stable \Cs-algebras.
\newblock Mem. Amer. Math. Soc. 293 (2024), no. 1461.
\newblock 106 pp.

\bibitem{GS22b}
J.~Gabe and G.~Szab{\'o}.
\newblock The dynamical Kirchberg--Phillips theorem.
\newblock Acta Math. 232 (2024), no.~1, pp. 1--77.

\bibitem{GS22a}
J.~Gabe and G.~Szab{\'o}.
\newblock The stable uniqueness theorem for equivariant {K}asparov theory.
\newblock To appear in Amer. J. Math.  (2024).
\newblock \urlprefix\url{https://arxiv.org/abs/2202.09809}.

\bibitem{Gar19}
E.~Gardella.
\newblock Compact group actions with the Rokhlin property.
\newblock Trans. Amer. Math. Soc. 371 (2019), no.~4, pp. 2837--2874.

\bibitem{GL16}
E.~Gardella and M.~Lupini.
\newblock Conjugacy and cocycle conjugacy of automorphisms of {$\O2$} are not
  {B}orel.
\newblock Münster J. Math.  (2016), no.~9, pp. 93--118.

\bibitem{GL18}
E.~Gardella and M.~Lupini.
\newblock Applications of model theory to {\Cs}-dynamics.
\newblock J. Funct. Anal. 275 (2018), pp. 1889--1942.

\bibitem{GS16}
E.~Gardella and L.~Santiago.
\newblock Equivariant $\ast$-homomorphisms, Rokhlin constraints and equivariant
  UHF-absorption.
\newblock J. Funct. Anal. 270 (2016), no.~7, pp. 2543--2590.

\bibitem{Haa87}
U.~Haagerup.
\newblock {Connes' bicentralizer problem and uniqueness of the injective factor
  of type III$_1$}.
\newblock Acta Math. 158 (1987), no.~1, pp. 95--148.

\bibitem{Izu04a}
M.~Izumi.
\newblock Finite group actions on $\Cs$-algebras with the Rohlin property, I.
\newblock Duke Math. J. 122 (2004), no.~2, pp. 233--280.

\bibitem{Izu04b}
M.~Izumi.
\newblock Finite group actions on \Cs-algebras with the Rohlin property II.
\newblock Adv. Math. 184 (2004), no.~1, pp. 119--160.

\bibitem{Izu11}
M.~Izumi.
\newblock Group actions on operator algebras.
\newblock In Proceedings of the ICM, Vol. 3 (Hyderabad, 2010). World
  Scientific, Singapore, pp. 1528--1548.

\bibitem{IM10}
M.~Izumi and H.~Matui.
\newblock $\Z^2$-actions on Kirchberg algebras.
\newblock Adv. Math. 224 (2010), pp. 355--400.

\bibitem{IM21a}
M.~Izumi and H.~Matui.
\newblock Poly-$\Z$ group actions on Kirchberg algebras I.
\newblock IMRN 2021 (2021), no.~16, pp. 12077--12154.

\bibitem{IM21b}
M.~Izumi and H.~Matui.
\newblock Poly-$\Z$ group actions on Kirchberg algebras II.
\newblock Invent. Math. 224 (2021), no.~3, pp. 699--766.

\bibitem{Jon83}
V.~F. Jones.
\newblock {A converse to Ocneanu's theorem}.
\newblock J. Oper. Theory  (1983), pp. 61--63.

\bibitem{Kas88}
G.~G. Kasparov.
\newblock Equivariant $KK$-theory and the Novikov conjecture.
\newblock Invent. Math. 91 (1988), no.~1, pp. 147--201.

\bibitem{KST98}
Y.~Katayama, C.~E. Sutherland and M.~Takesaki.
\newblock The characteristic square of a factor and the cocycle conjugacy of
  discrete group actions on factors.
\newblock Invent. Math. 132 (1998), no.~2, pp. 331--380.

\bibitem{KST92}
Y.~Kawahigashi, C.~E. Sutherland and M.~Takesaki.
\newblock {The structure of the automorphism group of an injective factor and
  the cocycle conjugacy of discrete abelian group actions}.
\newblock Acta Math. 169 (1992), pp. 105--130.

\bibitem{Kir}
E.~Kirchberg.
\newblock The Classification of Purely Infinite \Cs-Algebras Using Kasparov's
  Theory.
\newblock
  \urlprefix\url{https://www.uni-muenster.de/MathematicsMuenster/events/2023/cstar-algebras.shtml},
  preprint.

\bibitem{Kir00}
E.~Kirchberg.
\newblock Das nicht-kommutative Michael-Auswahlprinzip und die Klassifikation
  nicht-einfacher Algebren.
\newblock In $\Cs$-algebras. Springer-Verlag, Berlin Heidelberg (2000).
\newblock pp. 92--141.

\bibitem{Kir04}
E.~Kirchberg.
\newblock Central sequences in \Cs-algebras and strongly purely infinite
  algebras.
\newblock In Operator Algebras: The Abel Symposium, vol.~1. Springer, pp.
  175--231.

\bibitem{KP00}
E.~Kirchberg and N.~C. Phillips.
\newblock Embedding of exact \Cs-algebras in the Cuntz algebra $\O2$.
\newblock J. reine angew. Math.  (2000), no. 525, pp. 17--53.

\bibitem{KR00}
E.~Kirchberg and M.~R{\o}rdam.
\newblock Non-simple purely infinite \Cs-algebras.
\newblock Amer. J. Math. 122 (2000), no.~3, pp. 637--666.

\bibitem{KR02}
E.~Kirchberg and M.~R{\o}rdam.
\newblock Infinite non-simple $\Cs$-algebras: absorbing the Cuntz algebra
  $\mathcal{O}_\infty$.
\newblock Adv. Math. 167 (2002), no.~2, pp. 195--264.

\bibitem{Kishimoto96R}
A.~Kishimoto.
\newblock A {R}ohlin property for one-parameter automorphism groups.
\newblock Comm. Math. Phys. 179 (1996), no.~3, pp. 599--622.

\bibitem{Lan95}
E.~C. Lance.
\newblock Hilbert \Cs-modules: a toolkit for operator algebraists, vol. 210.
\newblock Cambridge University Press (1995).

\bibitem{MLM92}
S.~MacLane and I.~Moerdijk.
\newblock Sheaves in geometry and logic: A first introduction to topos theory.
\newblock Springer New York, NY (1992).

\bibitem{Mas13}
T.~Masuda.
\newblock Unified approach to the classification of actions of discrete
  amenable groups on injective factors.
\newblock J. reine angew. Math. 683 (2013), pp. 1--47.

\bibitem{Mat07}
H.~Matui.
\newblock Classification of outer actions of $\Z^N$ on $\O2$.
\newblock Adv. Math. 217 (2007), no.~6, pp. 2872--2896.

\bibitem{Nak00}
H.~Nakamura.
\newblock Aperiodic automorphisms of nuclear purely infinite simple
  {\Cs}-algebras.
\newblock Ergod. Theory Dyn. Syst. 20 (2000), no.~6, pp. 1749--1765.

\bibitem{Naw16}
N.~Nawata.
\newblock Finite group actions on certain stably projectionless \Cs-algebras
  with the Rohlin property.
\newblock Trans. Amer. Math. Soc. 368 (2016), no.~1, pp. 471--493.

\bibitem{Ocn85}
A.~Ocneanu.
\newblock {Actions of discrete amenable groups on von Neumann algebras}, vol.
  1138.
\newblock Springer (1985).

\bibitem{OS21}
N.~Ozawa and Y.~Suzuki.
\newblock On characterizations of amenable \Cs-dynamical systems and new
  examples.
\newblock Selecta Math. 27 (2021), no.~92.
\newblock 29 pp.

\bibitem{Ped18}
G.~K. Pedersen.
\newblock $\Cs$-algebras and Their Automorphism Groups.
\newblock Academic Press, Elsevier, London, UK, second ed. (2018).

\bibitem{Phi00}
N.~C. Phillips.
\newblock A classification theorem for nuclear purely infinite simple
  \Cs-algebras.
\newblock Doc. Math. 5 (2000), pp. 49--114.

\bibitem{Ror92}
M.~R{\o}rdam.
\newblock On the structure of simple \Cs-algebras tensored with a UHF-algebra,
  II.
\newblock J. Funct. Anal. 107 (1992), no.~2, pp. 255--269.

\bibitem{Ror94}
M.~R{\o}rdam.
\newblock A short proof of Elliott's theorem: $\O2\otimes\O2\cong\O2$.
\newblock Math. Rep. Acad. Sci. Canada 16 (1994), no.~1, pp. 31--36.

\bibitem{Ror02}
M.~R{\o}rdam and E.~St{\o}rmer.
\newblock Classification of nuclear \Cs-algebras. Entropy in operator algebras,
  vol. 126.
\newblock Springer (2002).

\bibitem{ST89}
C.~Sutherland and M.~Takesaki.
\newblock {Actions of discrete amenable groups on injective factors of type
  III$_\lambda$, $\lambda \neq 1$}.
\newblock Pacific J. Math. 137 (1989), no.~2, pp. 405--444.

\bibitem{Suz19}
Y.~Suzuki.
\newblock Simple equivariant \Cs-algebras whose full and reduced crossed
  products coincide.
\newblock J. Noncommut. Geom. 13 (2019), pp. 1577--1585.

\bibitem{Suz21}
Y.~Suzuki.
\newblock Equivariant $\O2$-absorption theorem for exact groups.
\newblock Comp. Math. 157 (2021), no.~7, pp. 1492--1506.

\bibitem{Sza17}
G.~Szab{\'o}.
\newblock Strongly self-absorbing \Cs-dynamical systems, III.
\newblock Adv. Math. 316 (2017), pp. 356--380.

\bibitem{Sza18b}
G.~Szab{\'o}.
\newblock Equivariant Kirchberg--Phillips-type absorption for amenable group
  actions.
\newblock Comm. Math. Phys. 361 (2018), no.~3, pp. 1115--1154.

\bibitem{Sza18a}
G.~Szab{\'o}.
\newblock Strongly self-absorbing $\Cs$-dynamical systems.
\newblock Trans. Amer. Math. Soc. 370 (2018), no.~1, pp. 99--130.

\bibitem{Sza18}
G.~Szab{\'o}.
\newblock Strongly self-absorbing \Cs-dynamical systems. II.
\newblock J. Noncommut. Geom. 12 (2018), no.~1, pp. 369--406.

\bibitem{Szabo19rd}
G.~Szab{\'o}.
\newblock Rokhlin dimension: absorption of model actions.
\newblock Anal. PDE 12 (2019), no.~5, pp. 1357--1396.

\bibitem{Sza21b}
G.~Szab{\'o}.
\newblock The classification of Rokhlin flows on \Cs-algebras.
\newblock Comm. Math. Phys. 382 (2021), no.~3, pp. 2015--2070.

\bibitem{Sza21}
G.~Szab{\'o}.
\newblock On a categorical framework for classifying $\Cs$-dynamics up to
  cocycle conjugacy.
\newblock J. Funct. Anal. 280 (2021), no.~8.
\newblock Article 108927, 66 pp.

\bibitem{TW07}
A.~Toms and W.~Winter.
\newblock Strongly self-absorbing \Cs-algebras.
\newblock Trans. Amer. Math. Soc. 359 (2007), no.~8, pp. 3999--4029.

\bibitem{Tom08}
A.~S. Toms.
\newblock On the classification problem for nuclear \Cs-algebras.
\newblock Ann. Math.  (2008), pp. 1029--1044.

\end{thebibliography}
\end{document}